\documentclass{article}

\usepackage{packages}

\newcommand{\tail}{\textnormal{tail}}
\newcommand{\head}{\textnormal{head}}
\newcommand{\tin}{\textnormal{in}}
\newcommand{\tout}{\textnormal{out}}

\newcommand{\tikzscale}{0.75}
\newcommand{\tikznodesize}{0.5cm}

\title{Simplicity of random hypergraphs}
\author{Yanna J. Kraakman$^*$ and Clara Stegehuis}
\affil{Faculty of Electrical Engineering, Mathematics and Computer Science, University of Twente, Drienerlolaan 5, 7522 NB, Enschede, The Netherlands}
\affil{$^*$Corresponding author. Email: y.j.kraakman@utwente.nl}
\date{April 7, 2026}

\begin{document}

\maketitle

\begin{abstract}
    Random hypergraphs extend the classical notion of random graphs by allowing hyperedges to join more than two vertices, making them well-suited for modeling higher-order interactions in complex systems. Despite their broad applicability, many structural properties of random hypergraphs remain less understood than in the graph setting. One such property is \emph{simplicity}: the absence of self-loops, multi-hyperedges, and, in the hypergraph context, degenerate hyperedges where hyperedges contain a copy of the same vertex at least twice. While the behaviour of the number of such self-loops and multi-hyperedges is well understood for random graphs through the configuration model, analogous results for hypergraphs are comparatively sparse. In this work, we study both undirected and directed hypergraphs generated by the configuration model with prescribed vertex and hyperedge degrees. We derive exact, explicit expressions for the expected number of self-loops, multi-hyperedges and degenerate hyperedges, extending classical results from the graph setting. In addition, an asymptotical analysis shows that, under mild moment conditions on the degree distribution, the expected fraction of self-loops, multi-hyperedges and degenerate hyperedges vanishes as the number of vertices grows. Our results provide a systematic understanding of simplicity in directed and undirected hypergraph models.
\end{abstract}

\section{Introduction}

Random hypergraphs provide a natural generalization of random graphs, by extending the notion of edges to hyperedges that may connect more than two vertices. They can model complex systems in which interactions occur among groups of entities rather than pairs, with applications ranging from network science and combinatorics to data analysis and statistical physics. Despite their importance, many structural properties of random hypergraphs remain less explored compared to their graph counterparts.

For graphs, one basic property that has received a lot of interest is their simplicity, i.e., there not being any edges from one vertex to itself, or multiple edges that connect the same pairs of vertices. 
For random graphs, classical results on the configuration model show that when degrees are bounded or have finite second moments, the expected number of self-loops and multi-edges remains tight and often converges to a Poisson distribution~\cite{Bollobas1980,Janson2009,angel2016}. More refined asymptotic analyses have established threshold phenomena: for heavy-tailed degree distributions, the probability of multi-edges and self-loops can grow significantly, affecting the simplicity of the resulting graph~\cite{Molloy1995,vanDerHofstad2016}. 

For hypergraphs, the literature on self-loops and multi-hyperedges is comparatively sparse. Next to loops and multi-hyperedges, another statistic arises that influences hypergraph simplicity, which is degeneracy~\cite{chodrow2020configuration}. A degenerate hyperedge contains the same vertex at least twice. Sampling uniform hypergraphs without loops, multi-hyperedges and degenerate hyperedges is possible by using a Markov Chain Monte Carlo approach~\cite{chodrow2020configuration,kraakman2025}, while constructive approaches allow to generate non-uniform simple hypergraphs~\cite{ascolese2024randomized}. However, for directed hypergraphs, this is already more involved~\cite{abuissa2025dinghy,kraakman2026}, but for certain classes of hypergraphs, the probability of generating a non-simple hypergraph tends to zero for undirected hypergraphs~\cite{DYER2021,greenhill2023degree}. This gives rise to the question: how many self-loops, multi-hyperedges and degenerate hyperedges appear in random hypergraphs? 

In this work, we focus on random hypergraphs with prescribed vertex and hyperedge degrees~\cite{chodrow2020configuration}, and on both \textit{undirected} and \textit{directed hypergraphs}. In undirected hypergraph, each hyperedge is a multiset of vertices, whereas in a directed hypergraph, each hyperedge is split in a tail- and head-multiset. Directed hypergraphs allow hyperedges to encode asymmetric relationships among groups of vertices, thereby capturing interactions that cannot be represented in ordinary graphs or undirected hypergraphs. For example, in biochemical reaction networks, a reaction may consume several molecules (the ``tail''
vertices) and produce several others (the ``head'' vertices), naturally giving rise to directed hyperedges. 

Our study investigates the expected values of several standard network statistics, extending classical results from
the graph configuration model to the hypergraph setting. We study two statistics in the undirected setting: the number of degenerate hyperedges and the number of multi-hyperedge pairs. In the directed setting, these statistics admit natural extensions, and we additionally introduce two direction-specific statistics: the number of self-loops and the number of weak self-loops. For each statistic, we obtain an exact expression for the expected value in a random undirected/directed hypergraph, which holds for arbitrary parameters and in all regimes. While the exact formulas are combinatorially involved, they admit a clean asymptotic form in specific regimes. An overview of the results is shown in Table \ref{table:results}. 

\begin{table}[tb]
    \centering
    \begin{tabular}{c|c|c|c|c}
        Statistic &  \multicolumn{2}{|c}{Undirected} & \multicolumn{2}{|c}{Directed} \\
        & Exact & Asymptotic* & Exact & Asymptotic* \\ \hline \hline
        $DH_n$ & Theorem \ref{thm:E[deg-edges]} & $\mathds{1}_{\{\delta \geq 2\}}\frac{\mathds{E}[d_U^2]}{\mathds{E}[d_U]}$  & Theorem \ref{thm:dir_E[deg-edges]} & $\mathds{1}_{\{\delta^{\tail} \geq 2\}}\frac{\mathds{E}[(d_U^{\tout})^2]}{\mathds{E}[d_U^{\tout}]}+ \mathds{1}_{\{\delta^{\head} \geq 2\}}\frac{\mathds{E}[(d_U^{\tin})^2]}{\mathds{E}[d_U^{\tin}]} $ \\
        \hline
        $M_n$ & Theorem \ref{thm:E[multi-edges]} & $\big(\frac{\mathds{E}[d_U^2]}{n \mathds{E}[d_U]^2} \big)^{\delta} (n \mathds{E}[d_U])^2 $ & Theorem \ref{thm:dir_E[multi-edges]} & $\big(\frac{\mathds{E}[(d_U^{\tout})^2]}{n \mathds{E}[d_U^{\tout}]^2} \big)^{\delta^{\tail}} \big(\frac{\mathds{E}[(d_U^{\tin})^2]}{n \mathds{E}[d_U^{\tin}]^2} \big)^{\delta^{\head}} n^2 \mathds{E}[d_U^{\tout}] \mathds{E}[d_U^{\tin}] $ \\ \hline
        $S_n$ & - & - & Theorem \ref{thm:E[sl]} & $ \big(\frac{\mathds{E}[d_U^{\tout} d_U^{\tin}]}{ n\mathds{E}[d_U^{\tin}]^2} \big)^{\delta} n \mathds{E}[d_U^{\tin}]$ \\ \hline
        $WS_n$ & - & - & Theorem \ref{thm:E[weak-sl]} & $ \frac{\mathds{E}[d_U^{\tout} d_U^{\tin}]}{\mathds{E}[d_U^{\tin}]}$
    \end{tabular}
    \caption{Results on the expected value of four statistics for a random hypergraph with $n$ vertices. The definitions of $DH_n$ and $M_n$ differ slightly between undirected and directed hypergraphs. $d_U$ is the degree of a randomly picked vertex. All exact results hold for all parameter choices and in all regimes. *The asymptotic results in this table are derived under the assumption that all hyperedges have the same size $\delta$ (or ($\delta^{\tail}, \delta^{\head}$) in the directed setting), which remains bounded as the number of vertices grows, the first number of moments of the vertex degree are sublinear in $n$ and a non-negligible number of vertices has a high enough degree.} 
    \label{table:results}
\end{table}

\begin{figure}[tb]
    \centering
    \begin{subfigure}[t]{0.45\textwidth}
    \centering
    \newcommand{\offsetx}{1}
\newcommand{\offsety}{0.5}

\begin{tikzpicture}[-Stealth, line width=1.3pt,auto,
                    thick,main node/.style={circle,draw, minimum size=\tikznodesize}, inner sep=1pt, hyperedge/.style={draw, rounded corners=7mm, thick, fill=gray!15}]

  \node[main node] (a) at (0,2*\tikzscale) {$a$};
  \node[main node] (b) at (2*\tikzscale,2*\tikzscale) {$b$};
  \node[main node] (c) at (4*\tikzscale,2*\tikzscale) {$c$};
  \node[main node] (d) at (0,0) {$d$};
  \node[main node] (e) at (2*\tikzscale,0) {$e$};
  \node[main node] (f) at (4*\tikzscale,0) {$f$};

\begin{scope}[on background layer]
\filldraw[hyperedge]
    ($(a)+(135:0.75cm)$) -- 
  ($(b)+(33.75:0.955cm)$) -- 
  ($(\tikzscale, \tikzscale) + (315:0.7cm)$) --
  ($(d)+(236.25:0.955cm)$) --
  cycle
  node[midway, left=1mm] {$e_1$};
\end{scope}

\begin{scope}[on background layer]
\filldraw[hyperedge, fill=gray!30]
  ($(a)+(135:0.60cm)$) -- 
  ($(b)+(33.75:0.755cm)$) -- 
  ($(\tikzscale, \tikzscale) + (315:0.6cm)$) --
  ($(d)+(236.25:0.755cm)$) --
  cycle
  node[midway, right=0.5mm] {$e_2$};
\end{scope}

\begin{scope}[on background layer]
\filldraw[hyperedge, fill=gray!15, rounded corners = 0.3cm]
 (3.5*\tikzscale, -0.5*\tikzscale) rectangle node[left=0.4cm] {$e_3$} (4.5*\tikzscale, 2.5*\tikzscale) ;
\end{scope}

\begin{scope}[on background layer]
\filldraw[hyperedge, fill=gray!15]
    (c) circle (0.3cm);
\end{scope}
\end{tikzpicture}
        \caption{Undirected hypergraph with vertex set $V=\{a,b,c,d,e,f\}$ and hyperedge set $E = \{e_1,e_2,e_3\}$, where $e_1 = \{a,b,d\},e_2 = \{a,b,d\}$ and $e_3 = \{c,c,f\}$.}
    \end{subfigure}
    \hfill
    \begin{subfigure}[t]{0.45\textwidth}
    \centering
    \newcommand{\offsetx}{1}
\newcommand{\offsety}{0.5}

\begin{tikzpicture}[-Stealth, line width=1.3pt,auto,
                    thick,main node/.style={circle,draw, minimum size=\tikznodesize}, inner sep=1pt]

  \node[main node] (a) at (0,2*\tikzscale) {$a$};
  \node[main node] (b) at (2*\tikzscale,2*\tikzscale) {$b$};
  \node[main node] (c) at (4*\tikzscale,2*\tikzscale) {$c$};
  \node[main node] (d) at (0,0) {$d$};
  \node[main node] (e) at (2*\tikzscale,0) {$e$};
  \node[main node] (f) at (4*\tikzscale,0) {$f$};

\coordinate (adab) at (\tikzscale,\tikzscale);

\draw[-,bend left] (adab) to (a.270);
\draw[-,bend right] (adab) to (d.90);
\draw[out=0, in=0] (adab) to node[right] {$e_1$} (a.0);
\draw[bend right] (adab) to (b.270);

\coordinate (dde) at (\tikzscale,0);

\draw[-,out=180, in=0] (dde) to (d.45);
\draw[-,out=180, in=0] (dde) to (d.315);
\draw (dde) to node[below left] {$e_2$} (e.180);

\draw (b.45) to node[above] {$e_3$} (c.135);
\draw (b.315) to node[below] {$e_4$} (c.225);

\coordinate (cf) at (4.5*\tikzscale,\tikzscale);

\draw[-,out=180, in=270] (cf.180) to (c.270);
\draw[-,out=180, in=90] (cf.180) to (f.90);
\draw[out=0, in=315] (cf) to node[below right] {$e_5$} (c.0);
\draw[out=0, in=45] (cf) to (f.0);

\begin{scope}[on background layer]
\draw[opacity=0]
    ($(a)+(135:0.75cm)$) -- 
  ($(b)+(33.75:0.955cm)$) -- 
  ($(\tikzscale, \tikzscale) + (315:0.7cm)$) --
  ($(d)+(236.25:0.955cm)$) --
  cycle
  node[midway, left=1mm] {$e_1$};
\end{scope}

\end{tikzpicture}
        \caption{Directed hypergraph with vertex set $V=\{a,b,c,d,e,f\}$ and hyperedge set $E=\{e_1,e_2,e_3,e_4,e_5\}$, where $e_1 = (\{a,d\},\{a,b\}), e_2 = (\{d,d\},\{e\}), e_3=(\{b\},\{c\}), e_4 = (\{b\},\{c\})$ and $e_5 = (\{c,f\},\{c,f\})$. \cite{kraakman2025}}
    \end{subfigure}
    \caption{Illustration of the studied statistics on both an undirected and a directed hypergraph. In the undirected example, the hyperedge $e_3$ is degenerate, and the hyperedges $e_1$ and $e_2$ are a multi-hyperedge pair. In the directed example, the hyperedge $e_2$ is degenerate, the hyperedges $e_3$ and $e_4$ are a multi-hyperedge pair, the hyperedge $e_1$ is a weak self-loop and the hyperedge $e_5$ is a self-loop and a weak self-loop.}
    \label{fig:statistics}
\end{figure}
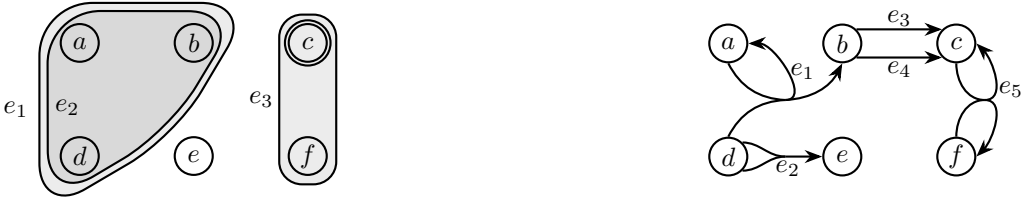

Figure \ref{fig:statistics} depicts the statistics we investigate. Firstly, we investigate the number of degenerate hyperedges in an undirected hypergraph with $n$ vertices, $DH_n$. A hyperedge is degenerate if it contains a copy of the same vertex more than once. If all hyperedges have the same size and under some mild conditions on the vertex degree moments, the expectation of $DH_n$ is primarily governed by the first and second moment of the vertex degree. Similarly, the directed version shows that the expectation of $DH_n$ is primarily governed by the first and second moment of the vertex in-degree and out-degree. 

We then study the number of multi-hyperedge pairs in an undirected hypergraph with $n$ vertices, $M_n$. Two hyperedges are a multi-hyperedge pair if they are equal as multisets. If all hyperedges have the same size and under some mild conditions on the vertex degree moments, the expectation of $M_n$ is primarily governed by the first and second moment of the vertex degree, as well as the number of vertices and the hyperedges size. The directed version shows similar results, using the vertex in- and out-degrees.

Thirdly, we investigate the number of self-loops in a directed hypergraph with $n$ vertices, $S_n$. A hyperedge is a self-loop if its tail and head are equal as multisets. If all hyperedges have equal tail and head sizes and under some mild conditions on the vertex in-degree and out-degree moments, the expectation of $S_n$ is primarily governed by the first moment of the product of the in- and out-degree of a vertex, as well as the first moment of the in-degree of a vertex, the hyperedge size and $n$. 

Finally, we study the number of weak self-loops in a directed hypergraph with $n$ vertices, $WS_n$. A hyperedge is a weak self-loop if its tail and head have a non-empty intersection. Note that every self-loop is also a weak self-loop. If all hyperedges have the same tail size and head size and under some mild conditions on the vertex in-degree and out-degree moments, the expectation of $WS_n$ is primarily governed by the first moment of the product of the in- and out-degree of a vertex, as well as the first moment of the in-degree of a vertex.

We show that for every analyzed statistic, under mild moment conditions, the expected fraction of such hyperedges goes to 0 as $n$ grows, i.e.,
\begin{align*}
    \frac{\mathbb{E}[DH_n]}{|E|}, \frac{\mathbb{E}[M_n]}{|E|}, \frac{\mathbb{E}[S_n]}{|E|}, \frac{\mathbb{E}[WS_n]}{|E|} \xrightarrow{n \rightarrow \infty} 0.
\end{align*}

\paragraph{Notation}
We consider $0 \in \mathds{N}$, and we use the formality $0^0=1$. Furthermore, we denote $[x]=\{1,2,\hdots,x\}$ and we denote by
\begin{align*}
    \sideset{}{^*}\sum_{\vb*{x} \in X^a} = \sum_{x_1 \in X} \sum_{\substack{x_2 \in X \\ x_2 \neq x_1}} \hdots \sum_{\substack{x_a \in X \\ x_a \neq x_1,x_2,\hdots,x_{a-1}}}
\end{align*}
a sum over all lists $\vb*{x} \in X^a$ that consist of non-repeating elements. Furthermore, we denote $||\vb*{x}||_0$ as the number of non-zero elements in $\vb*{x}$.

\paragraph{Organization of the paper}
Section \ref{section:model} introduces the hypergraph model as well as the definition of a random hypergraph. In Section \ref{section:undirected}, we provide our exact and asymptotic results for undirected hypergraphs, while Section \ref{section:directed} focuses on directed hypergraphs. Section \ref{section:conclusion} contains the conclusion. The proofs are in Section \ref{section:proofs} and Appendix \ref{app:pf_main_lemma} and \ref{app:pf_lemma_general_order}.

\section{Random hypergraph model}
\label{section:model}
A hypergraph $H = (V, E)$ consists of a vertex set $V$ and a multiset $E$ of hyperedges. 
Throughout this work we consider both undirected and directed hypergraphs. 
We first introduce these two hypergraph types together with their associated degree sequences, 
after which we define random hypergraph models in which the degree sequence is fixed while the incidences between vertices and hyperedges are random.

In an undirected hypergraph, each hyperedge $e \in E$ is a multiset of vertices in $V$. Each vertex $v 
\in V$ then has a degree $d_v$, possibly depending on $n=|V|$, counting the number of hyperedges that the vertex participates in. Here, we take into account multiplicity of the vertex in the hyperedges. To that end, let $m_i(j)$ count the number of occurrences of element $i$ in $j$. The degree of vertex $v$ is given by
\begin{align*}
    d_v = \sum_{e \in E} m_v(e).
\end{align*}
In addition, each hyperedge $e \in E$ has a degree $\delta_e$, possibly depending on $n$, counting the number of vertices that it contains. Thus,
\begin{align*}
    \delta_e = \sum_{v \in V} m_v(e).
\end{align*}
We define the vertex–degree sequence as $\vb*{d}_V = (d_v)_{v \in V}$ and the hyperedge–degree sequence as 
$\vb*{\delta}_E = (\delta_e)_{e \in E}$.  
The degree sequence of the undirected hypergraph is then $\vb*{d}=(\vb*{d}_V,\vb*{\delta}_E)$.

In a directed hypergraph, each hyperedge $e \in E$ is an ordered pair of multisets of vertices, 
\[
    e = (e^{\tail}, e^{\head}),
\]
where $e^{\tail}$ is the tail multiset and $e^{\head}$ is the head multiset.  
This generalizes directed graphs, where each directed edge has exactly one tail and one head vertex. For a vertex $v \in V$, the out-degree $d_v^{\tout}$ and in-degree $d_v^{\tin}$ count, respectively, the number of hyperedges in which $v$ appears in the tail and in the head (again taking into account multiplicity):
\begin{align*}
    d_v^{\tout} &= \sum_{e \in E} m_v(e^{\tail})\\
    d_v^{\tin} &= \sum_{e \in E} m_v(e^{\head}).
\end{align*}
Similarly, each hyperedge $e \in E$ has a tail-degree and head-degree defined by
\begin{align*}
    \delta_e^{\tail} &= \sum_{v \in V} m_v(e^{\tail}) \\
    \delta_e^{\head} &= \sum_{v \in V} m_v(e^{\head}).
\end{align*}
We define the vertex degree sequence as $\vb*{d}_V = ((d_v^{\tout}, d_v^{\tin}))_{v \in V}$ and the hyperedge degree sequence as $\vb*{\delta} = ((\delta_e^{\tail}, \delta_e^{\head}))_{e \in E}$. The degree sequence of the directed hypergraph is then $\vb*{d} = (\vb*{d}_V,\vb*{\delta}_E)$.

Given a directed or undirected hypergraph degree sequence $\vb*{d}$, we define a 
\emph{uniformly random hypergraph with degree sequence $\vb*{d}$} as a uniformly chosen element from the set of all stub-labeled hypergraphs realizing $\vb*{d}$. 
In the stub-labeled representation, each vertex $v$ with degree $d_v$ is assigned $d_v$ distinct \emph{stubs}.  
These stubs are considered non-interchangeable: attaching a particular stub of $v$ to a hyperedge counts as a distinct configuration \cite{fosdick2018}.  On the other hand, each hyperedge $e$ with degree $\delta_e$ is equipped with $\delta_e$ `vertex slots' that are considered interchangeable. Thus, there is only one way to connect a vertex stub to a hyperedge. Given the degree sequence $\vb*{d}$, a uniformly random stub-labeled hypergraph is generated by matching each hyperedge $e$ to exactly $\delta_e$ vertex stubs. In the directed setting, this procedure is performed separately for in- and out-stubs: the tail part $e^{\tail}$ of hyperedge $e$ is matched to $\delta_e^{\tail}$ vertex out-stubs, while the head part $e^{\head}$ is matched to $\delta_e^{\head}$ vertex in-stubs.

This matching process is the natural generalization of the classical stub-matching method for graphs~\cite{Molloy1995}, which is sometimes referred to as the \textit{configuration model}, and may equivalently be viewed as matching stubs in a bipartite configuration model between vertex stubs and hyperedge slots (in the undirected case). After matching the stubs, the resulting hypergraph may contain degenerate hyperedges, multi-hyperedges, self-loops, and/or weak self-loops. To prevent such structures, a Markov chain Monte Carlo edge-swapping method can be applied \cite{chodrow2020configuration, kraakman2026, kraakman2025}.  

\section{Undirected hypergraphs}
\label{section:undirected}
\subsection{Degenerate hyperedges and multi-hyperedges}
In this section, we consider a uniformly random directed hypergraph $H=(V,E)$ with $n=|V|$ vertices and some degree sequence $\vb*{d}$. We study the expected number of degenerate hyperedges, multi-hyperedge pairs, self-loops and weak self-loops in this hypergraph. We present exact results, as well as asymptotic approximations for regular hypergraphs in which all hyperedges have the same degrees. The definitions of degenerate hyperedges and multi-hyperedge pairs naturally extend from the definition for undirected hypergraphs introduced in Section \ref{section:undirected_stats}. To avoid notational clutter, we use the same symbols. The meaning should be clear from the context.

A hyperedge is considered \textit{degenerate} if its tail or head contains some vertex at least twice~\cite{kraakman2026}.
\\
\begin{definition}[Degenerate hyperedge (directed)]
    A directed hyperedge $e \in E$ is called degenerate if $\exists v \in V: m_v(e^{\tail}) \geq 2 \lor m_v(e^{\head}) \geq 2$.
\end{definition}
We denote the total number of directed degenerate hyperedges by
\begin{align*}
    DH_n = \sum_{e \in E} \mathds{1}_{\{\exists v \in V: m_v(e^{\tail}) \geq 2 \, \lor \,  m_v(e^{\head}) \geq 2\}}.
\end{align*}
Two hyperedges are considered a \textit{multi-hyperedge pair} if their tails are equal as multisets and their heads are equal as multisets~\cite{kraakman2026}.
\\
\begin{definition}[Multi-hyperedge pair (directed)]
    Two directed hyperedges $e_1,e_2 \in E$ are a multi-hyperedge pair if $e_1^{\tail}=e_2^{\tail} \land e_1^{\head}=e_2^{\head}$.
\end{definition}
We denote the total number of multi-hyperedge pairs by 
\begin{align*}
    M_n = \frac{1}{2} \sum_{e_1 \in E} \sum_{e_2 \in E \backslash \{e_1\}} \mathds{1}_{\{e_1^{\tail}=e_2^{\tail} \land e_1^{\head}=e_2^{\head}\}}.
\end{align*}

A hyperedge is a \textit{self-loop} if its tail equals its head as a multiset~\cite{kraakman2026}.
\\
\begin{definition}[Self-loop]
    A directed hyperedge $e \in E$ is called a self-loop if $e^{\tail} = e^{\head}$.
\end{definition}
We denote the total number of self-loops by
\begin{align*}
    S_n = \sum_{e \in E} \mathds{1}_{\{e^{\tail}=e^{\head}\}}.
\end{align*}

A hyperedge is a \textit{weak self-loop} if its tail and head have a nonempty intersection (in~\cite{abuissa2025dinghy}, the authors refer to such a hyperedge as a \textit{degenerate hyperedge}).
\\
\begin{definition}[Weak self-loop]
    A directed hyperedge $e \in E$ is called  weak a self-loop if $\exists v \in V: v \in e^{\tail} \land v \in e^{\head}$.
\end{definition}
We denote the total number of weak self-loops by
\begin{align*}
    WS_n = \sum_{e \in E} \mathds{1}_{\{e^{\tail} \cap e^{\head} \neq \emptyset\}}.
\end{align*}

\subsection{Expected number of degenerate hyperedges and multi-hyperedges}

We begin by analyzing the expected number of degenerate hyperedges in a uniformly random hypergraph with prescribed degree sequence.  

To evaluate this expectation, we consider each hyperedge $e \in E$ independently and compute the probability that $e$ is non-degenerate under the random stub-matching process.  
A hyperedge of degree $\delta_e$ consists of $\delta_e$ vertex slots, each of which is filled by one vertex stub chosen uniformly from all of stubs.  
The resulting multiset of vertices in $e$ can be described by a vector 
\[
    \vb*{a} = (a_1,\dots,a_{\delta_e}) \in \mathbb{N}^{\delta_e},
\]
where $a_i = |\{v \in V: m_v(e) = i\}|$ denotes the number of distinct vertices that appear in $e$ with multiplicity exactly $i$.  
These multiplicities satisfy the constraint
\[
    \sum_{i=1}^{\delta_e} i a_i = \delta_e.
\]

For example, when $\delta_e = 4$, the possible multiplicity patterns are:  
(i) four distinct vertices ($\vb*{a} = (4,0,0,0)$);  
(ii) three distinct vertices, one of which with multiplicity 2 ($\vb*{a} = (2,1,0,0)$);  
(iii) two vertices, each with multiplicity 2 ($\vb*{a} = (0,2,0,0)$);  
(iv) two vertices, one with multiplicity 3 ($\vb*{a} = (1,0,1,0)$);  
and (v) a single vertex with multiplicity 4 ($\vb*{a} = (0,0,0,1)$).  
A hyperedge is \emph{degenerate} precisely when $a_i > 0$ for some $i \ge 2$, that is, when at least one vertex appears with multiplicity at least 2.

Let $U$ be a uniformly chosen vertex from $V$, and denote by $d_U$ its degree.  
The following theorem provides an explicit expression for the expected number of degenerate hyperedges in a random hypergraph.
\\
\begin{theorem}
\label{thm:E[deg-edges]}
For a uniformly random undirected hypergraph with degree sequence $\vb*{d}$, the expected number of degenerate hyperedges is
    \begin{align}
    \label{eq:E[DH]}
        \mathds{E}[DH_n] = |E| - \sum_{e \in E} \frac{\delta_e!(n\mathds{E}[d_U]-\delta_e)!}{(n\mathds{E}[d_U])!} \sum_{\substack{\vb*{a} \in \mathds{N}^{\delta_e}: \\ \sum_{i=1}^{\delta_e}ia_i=\delta_e}} (-1)^{\sum_{i=1}^{\delta_e} (i-1)a_i} \prod_{i=1}^{\delta_e} \Bigg(\frac{1}{a_i!}\Big(\frac{n\mathds{E}[d_U^i]}{i}\Big)^{a_i} \Bigg).
    \end{align}
\end{theorem}
The proof of this theorem is in Section \ref{section:pf_degenerate}.
For graphs, where $\forall e \in E: \delta_e=2 $, Theorem \ref{thm:E[deg-edges]} describes the number of self-loops and reduces to
\begin{align*}
    \mathds{E}[\#\textnormal{self-loops}] &= \frac{\mathds{E}[d_U^2] - \mathds{E}[d_U]}{2\mathds{E}[d_U]-\frac{2}{n}} = \frac{\mathds{E}[d_U^2] - \mathds{E}[d_U]}{2\mathds{E}[d_U]}(1+o(1)),
\end{align*}
which aligns with the result in \cite{angel2016}.

For regular hypergraphs, where every hyperedge has the same degree, we present the following asymptotic result, which simplifies Theorem~\ref{thm:E[deg-edges]} significantly:
\\
\begin{lemma}
\label{lemma:E[deg-edges]_asymp}
If $\forall e \in E: \delta_e = \delta\in O(1)$ and
\begin{enumerate}
    \item $\mathds{E}[d_U^{\delta}] \in o(n)$
    \item $\exists c>0: \lim_{n \rightarrow \infty} \mathds{P}(d_U \geq 1) \geq c$
\end{enumerate}
then 
\begin{align*}
    \mathds{E}[DH_n] =  \frac{\delta-1}{\delta} \frac{\delta \mathds{E}[d_U^2] - 2 \mathds{E}[d_U]}{2 \mathds{E}[d_U]} (1+o(1)).
\end{align*}
\end{lemma}

The bound in Lemma~\ref{lemma:E[deg-edges]_asymp} provides the scaling behavior of degenerate hyperedges in regular undirected hypergraphs, under mild moment conditions on the degree distribution. For $\delta=1$, degenerate hyperedges cannot exist, and we observe $\mathds{E}[DH_n]=0$. For $\delta \geq 2$, the expected number of degenerate hyperedges grows with $\delta$, as larger hyperedges have a higher chance to contain the same vertex twice. In addition, the fraction of degenerate hyperedges vanishes as $n$ grows, since the moments of the vertex degrees grow slower than $n$. Asymptotically, we obtain $\mathds{E}[DH_n] = \Theta \Big( \mathds{1}_{\{\delta \geq 2\}}\frac{ \mathds{E}[d_U^2]}{\mathds{E}[d_U]} \Big)$. This is independent of $\delta$, except requiring $\delta \geq 2$, as the probability of creating a degenerate hyperedge is governed by the probability of picking two stubs from the same vertex when picking two stub uniformly at random. For each hyperedge, the number of tries to pick two vertex stubs from the same vertex is a combinatorial constant depending on $\delta$, which can be ignored since we assume $\delta = O(1)$. For this same reason, the asymptotic behavior of the expected number of degenerate hyperedges equals the asymptotic behavior of the expected number of self-loops in an undirected graph. For sparse regular hypergraphs with finite second moment, the number of degenerate hyperedges is constant.

To evaluate the expected number of multi-hyperedge pairs in a uniformly random hypergraph, 
we proceed by considering each ordered pair of hyperedges and computing the probability that they are identical under the stub-matching process.  
Since the probability of realizing a specific hyperedge depends both on the multiset of vertices it contains and on the multiplicities of those vertices, 
we again describe the configuration of a hyperedge of degree $\delta_e$ by a multiplicity vector
\[
    \vb*{a} = (a_1,\dots,a_{\delta_e}) \in \mathbb{N}^{\delta_e},
\]
where $a_i = |\{v \in V: m_v(e) = i\}|$ denotes the number of distinct vertices that appear in the hyperedge with multiplicity $i$, subject to the constraint
\[
    \sum_{i=1}^{\delta_e} i a_i = \delta_e.
\]

Unlike in the computation of the expected number of degenerate hyperedges, here we must take a closer look at the vertex assignments in $e$ that meet the multiplicity vector $\vb*{a}$. Each such $\vb*{a}$ describes a multiset of vertices. For instance, if $\delta_e = 4$ and $\vb*{a} = (2,1,0,0)$, then the hyperedge consists of two vertices with multiplicity 1 and one vertex with multiplicity 2. This corresponds to a multiset of the form $\{v_1, v_2, v_3, v_3\}$. Now, we would like to consider all possible vertex assignments $v_1,v_2,v_3 \in V$, but some assignments may not use distinct vertices, e.g., $v_1=v_2$, in which case the multiplicity pattern no longer matches $\vb*{a}$. To isolate only those instances in which the vertices with multiplicities described by $\vb*{a}$ are indeed distinct, we apply an inclusion–exclusion argument over all ways in which they may coincide. To formalize this, for a fixed multiplicity pattern $\vb*{a}$ we define the set
\begin{align}
    \label{def:B}
    R(\vb*{a}) = \Big\{\alpha(\cdot): \sum_{\vb*{x} \in \mathds{N}^{\delta_e}} x_i \alpha(\vb*{x}) = a_i \, \, \forall i \in [\delta_e] \Big\}
\end{align}
which consists of all partitions of the vertices described by~$\vb*{a}$ into groups of coinciding vertices.  
 Any $\alpha(\cdot) \in R(\vb*{a})$ describes how many vertices with a specific multiplicity coincide. Let the vector $\vb*{x} \in \mathds{N}^{\delta_e}$ be a subset of coinciding vertices with specific multiplicities, where $x_i$ is the number of vertices with multiplicity $i$. For example, returning to the case $\delta_e = 4$ with $\vb*{a} = (2,1,0,0)$, possible coinciding patterns are:
\begin{itemize}
    \item $\vb*{x} = (1,0,0,0)$ ($v_1$ or $v_2$ is unique),
    \item $\vb*{x} = (0,1,0,0)$ ($v_3$ is unique),
    \item $\vb*{x} = (2,0,0,0)$ ($v_1=v_2$),
    \item $\vb*{x} = (1,1,0,0)$ ($v_1=v_3$ or $v_2=v_3$),
    \item $\vb*{x} = (2,1,0,0)$ ($v_1=v_2=v_3$).
\end{itemize}
For any such $\vb*{x}$, let $\alpha(\vb*{x})$ count how many such coinciding sets exist in $e$. Note that $\alpha(\cdot)$ must satisfy
\begin{align*}
    \sum_{\vb*{x} \in \mathds{N}^{\delta_e}} x_i \alpha(\vb*{x}) = a_i \, \forall i \in [\delta_e],
\end{align*}
since each vertex has to appear in exactly one such pattern. In the running example, $\alpha(\cdot)$ has to be one of the following:
\begin{itemize}
    \item $\alpha((1,0,0,0))=2, \alpha((0,1,0,0))=1$ ($v_1$, $v_2$ and $v_3$ are unique)
    \item $\alpha((1,0,0,0))=1, \alpha((1,1,0,0))=1$ ($v_1$ is unique and $v_2=v_3$, or $v_2$ is unique and $v_1=v_3$)
    \item $\alpha((0,1,0,0))=1, \alpha((2,0,0,0))=1$ ($v_3$ is unique and $v_1=v_2$)
    \item $\alpha((2,1,0,0))=1$ ($v_1=v_2=v_3$).
\end{itemize}

This allows us to apply an inclusion–exclusion sum that isolates only those assignments in which all vertices with the same multiplicity class are distinct, leading to the following expression for the expected number of multi-hyperedge pairs in the random hypergraph.
\\
\begin{theorem}
\label{thm:E[multi-edges]}
   Let $\forall e \in E: \delta_e \leq \frac{1}{2} n \mathds{E}[d_U]$ and let $R(\cdot)$ be as in \eqref{def:B}. Then the expected number of multi-hyperedge pairs in a uniformly random undirected hypergraph with degree sequence $\vb*{d}$ is
    \begin{align}
    \label{eq:E[M]}
        \mathds{E}[M_n]&= \frac{1}{2} \sum_{e \in E} \sum_{\substack{e' \in E\backslash\{e\}:\\\delta_{e'}=\delta_e}}\frac{(n\mathds{E}[d_U]-2\delta_e)!}{(n \mathds{E}[d_U])!}\sum_{\substack{\vb*{a} \in \mathds{N}^{\delta_e}:\\\sum_{i=1}^{\delta_e} ia_i = \delta_e}}  \frac{\delta_e!^2}{\prod_{k=1}^{\delta_e} k!^{2a_k} }\sum_{\alpha(\cdot) \in R(\vb*{a})} (-1)^{\sum_{\vb*{x} \in \mathds{N}^{\delta_e}} (\sum_{i=1}^{\delta_e} x_i - 1)\alpha(\vb*{x})} \\
     &\hspace{1cm}\times 
    \prod_{\vb*{y} \in \mathds{N}^{\delta_e}} \Bigg( \frac{1}{\alpha(\vb*{y})!} \Big(\frac{n\mathds{E}\big[\prod_{i=1}^{\delta_e} (\mathds{1}_{\{d_U \geq 2i\}}\frac{d_U!}{(d_U-2i)!})^{y_i} \big]\big(\sum_{i=1}^{\delta_e} y_i-1 \big)!}{\prod_{i=1}^{\delta_e} y_i!} \Big)^{\alpha(\vb*{y})} \Bigg). \nonumber
    \end{align}
\end{theorem}
The proof of this theorem is in Section \ref{section:pf_multi}.

For graphs, where $\forall e \in E: \delta_e = 2$, Theorem \ref{thm:E[multi-edges]} reduces to
\begin{align*}
    \mathds{E}[M_n] &= \frac{(\mathds{E}[d_U(d_U-1)])^2 - \frac{1}{n} \mathds{E}[(d_U(d_U-1))^2] + \frac{1}{2n} \mathds{E}[d_U(d_U-1)(d_U-2)(d_U-3)]}{4(\mathds{E}[d_U]-\frac{1}{n})(\mathds{E}[d_U]-\frac{3}{n})}\nonumber\\
    & =  \frac{(\mathds{E}[d_U^2] - \mathds{E}[d_U])^2}{4\mathds{E}[d_U]^2}(1+o(1)),
\end{align*}
which aligns with the result in \cite{angel2016}.

For regular hypergraphs, where every hyperedge has the same degree, we present the following asymptotic result.
\\
\begin{lemma}
\label{lemma:E[multi_edges]_asymp}
If $\forall e \in E: \delta_e = \delta\in O(1)$ and 
\begin{enumerate}
    \item $\mathds{E}[d_U^{2\delta}] \in o(n)$
    \item $\exists c > 0: \lim_{n \rightarrow \infty} \mathds{P}(d_U \geq 2 \delta) \geq c$ 
\end{enumerate}
then
\begin{align*}
    \mathds{E}[M_n] = \frac{(\delta-1)!}{2\delta}  (n \mathds{E}[d_U])^2 \Bigg(\frac{\mathds{E}[d_U^2] - \mathds{E}[d_U]}{n\mathds{E}[d_U]^2}\Bigg)^{\delta} (1+o(1)).
\end{align*}
\end{lemma}

The asymptotic bound in Lemma~\ref{lemma:E[multi_edges]_asymp} highlights the scaling behavior of multi-hyperedge pairs in regular hypergraphs under mild moment conditions on the degree distribution. Observe that the factor $(n \mathds{E}[d_U])^2$ scales as the number of hyperedge pairs, and the factor $\frac{(\delta-1)!}{2\delta}\Big(\frac{\mathds{E}[d_U^2] - \mathds{E}[d_U]}{n\mathds{E}[d_U]^2}\Big)^{\delta}$ computes the probability of each such pair to form a multi-hyperedge pair. The expected number of multi-hyperedge pairs decreases when $\delta$ grows, as for a larger hyperedges to form a multi-hyperedge pair with another hyperedge, more vertices need to be picked twice. Asymptotically, we obtain $\mathds{E}[M_n] = \Theta \Big( \big(\frac{ \mathds{E}[d_U^2]}{n\mathds{E}[d_U]^2} \big)^{\delta} (n \mathds{E}[d_U])^2 \Big)$. Here, $\delta$ appears as a power, since all of the $\delta$ vertices in one hyperedge of the multi-hyperedge pair also need to appear in the other. We observe that the fraction of multi-hyperedge pairs vanishes as $n$ grows, since the moments of the vertex degrees grow slower than $n$. In addition, for sparse hypergraphs with bounded degrees, the probability of observing multi-hyperedges vanishes as $n \to \infty$. Thus, the random hypergraph is asymptotically simple in the sense that almost all hyperedges are distinct.

\section{Directed hypergraphs}
\label{section:directed}
We now turn to directed hypergraphs, and investigate the number of directed degenerate hyperedges, multi-hyperedge pairs and self-loops.
\subsection{Degenerate hyperedges, multi-hyperedges, self-loops and weak self-loops}

\subsection{Expected number of degenerate hyperedges, multi-hyperedges and self-loops}
For directed hypergraphs, the expected number of degenerate hyperedges can be analyzed similarly as in undirected hypergraphs. but the tail and head parts of each hyperedge must be treated separately.  A directed hyperedge $e$ consists of two multisets of vertices: a tail $e^{\tail}$ of size $\delta_e^{\tail}$ and a head $e^{\head}$ of size $\delta_e^{\head}$.  
To characterize all possible multiplicity patterns arising from the random stub-matching process, we use two multiplicity vectors:
\[
    \vb*{a} = (a_1,\dots,a_{\delta_e^{\tail}}) \in \mathds{N}^{\delta_e^{\tail}}, 
    \qquad
    \vb*{b} = (b_1,\dots,b_{\delta_e^{\head}}) \in \mathds{N}^{\delta_e^{\head}},
\]
where $a_i = |\{v \in V: m_v(e^{\tail}) = i\}|$ counts the number of distinct vertices appearing in the tail with multiplicity $i$, and $b_j = |\{v \in V: m_v(e^{\head}) = j\}|$ counts the number of distinct vertices appearing in the head with multiplicity $j$. A directed hyperedge is \emph{degenerate} when either its tail or its head contains repeated vertices, or, equivalently, when $a_i > 0$ for some $i \ge 2$ or $b_j > 0$ for some $j \ge 2$.  
Let $U$ be a uniformly chosen vertex from $V$, and let $d_U^{\tout}$ and $d_U^{\tin}$ denote its out-degree and in-degree, respectively.  
The following theorem gives the expected number of degenerate directed hyperedges.
\\
\begin{theorem}
\label{thm:dir_E[deg-edges]}
For a uniformly random directed hypergraph with degree sequence $\vb*{d}$,
    \begin{align*}
         \mathds{E}[DH_n] &= |E| -  \sum_{e \in E} \frac{\delta_e^{\tail}!\delta_e^{\head}!(n\mathds{E}[d_v^{\tout}]-\delta_e^{\tail})!(n\mathds{E}[d_v^{\tin}]-\delta_e^{\head})!}{(n\mathds{E}[d_v^{\tout}])!(n\mathds{E}[d_v^{\tin}])!}\\
         &\quad \times \sum_{\substack{\vb*{a} \in \mathds{N}^{\delta_e^{\tail}}:\\ \sum_{i=1}^{\delta_e^{\tail}}ia_i= \delta_e^{\tail}}} (-1)^{\sum_{i=1}^{\delta_e^{\tail}}(i-1)a_i}  \prod_{i=1}^{\delta_e^{\tail}}  \Bigg( \frac{1}{a_i!}\Big(\frac{n\mathds{E}[(d_v^{\tout})^i]}{i}\Big)^{a_i} \Bigg)\\
    &\quad \times\sum_{\substack{\vb*{b} \in \mathds{N}^{\delta_e^{\head}}:\\ \sum_{i=1}^{\delta_e^{\head}}ib_i = \delta_e^{\head}}} (-1)^{\sum_{i=1}^{\delta_e^{\head}}(i-1)b_i} \prod_{i=1}^{\delta_e^{\head}} \Bigg( \frac{1}{b_i!}\Big(\frac{n\mathds{E}[(d_v^{\tin})^i]}{i}\Big)^{b_i} \Bigg).
    \end{align*} 
\end{theorem}

The proof of Theorem \ref{thm:dir_E[deg-edges]} is in Section \ref{section:pf_dir_degenerate}. For digraphs, degenerate hyperedges cannot exist, and Theorem \ref{thm:dir_E[deg-edges]} reduces to $0$, as expected.
For regular directed hypergraphs, where all hyperedges have the same degrees, we present the following asymptotic result.\\
\begin{lemma}
\label{lemma:dir_E[deg-edges]_asymp}
If $\forall e \in E: \delta_e^{\tail} = \delta^{\tail}\in O(1)$ and $\delta_e^{\head}=\delta^{\head}\in O(1)$ and
\begin{enumerate}
    \item $\mathds{E}[(d_U^{\tout})^{\delta^{\tail}}], \mathds{E}[(d_U^{\tin})^{\delta^{\head}}] \in o(n)$
    \item $\exists c>0: \lim_{n \rightarrow \infty} \mathds{P}(d_U^{\tout}, d_U^{\tin} \geq 1) \geq c$
\end{enumerate}
then 
\begin{align*}
    \mathds{E}[DH_n] = \Big(\frac{\delta^{\tail}-1}{\delta^{\tail}} \cdot \frac{\delta^{\tail} \mathds{E}[(d_U^{\tout})^2] - 2 \mathds{E}[d_U^{\tout}]}{2 \mathds{E}[d_U^{\tout}]} + \frac{\delta^{\head}-1}{\delta^{\head}}  \cdot \frac{\delta^{\head} \mathds{E}[(d_U^{\tin})^2] - 2 \mathds{E}[d_U^{\tin}]}{2 \mathds{E}[d_U^{\tin}]}\Big)(1+o(1)).
\end{align*}
\end{lemma}

The asymptotic bound in Lemma~\ref{lemma:dir_E[deg-edges]_asymp} provides the scaling behavior of degenerate hyperedges in regular directed hypergraphs under mild moment conditions on the degree distribution. The expression is a sum of a term concerning tail- and out-degrees, and a term concerning head- and in-degrees. Both terms equal the asymptotic expression found for undirected degenerate hyperedges, albeit with the applicable degrees. This can be explained by the definition of a directed degenerate hyperedge: if either the tail or the head of a hyperedge is degenerate, then the hyperedge is considered degenerate. Therefore, the expected number of directed degenerate hyperedges equals the sum of the expected number of degenerate tails and the expected number of degenerate heads, asymptotically. The term describing the expected number of directed hyperedges with both a degenerate tail and a degenerate head is contained in the $o(1)$.

Similarly as for the undirected degenerate hyperedges, if $\delta^{\tail}=\delta^{\head}=1$ then no degenerate hyperedges are possible, and we obtain $\mathds{E}[DH_n] = 0$. Else, the expected number of degenerate hyperedges grows with $\delta^{\tail}$ and $\delta^{\head}$, as with larger tail- or head-degree, there is a larger chance to contain the same vertex twice. In addition, the fraction of degenerate hyperedges vanishes as $n$ grows. Asymptotically, we obtain $\mathds{E}[DH_n] = \Theta \Big( \mathds{1}_{\{\delta^{\tail} \geq 2\}}\frac{ \mathds{E}[(d_U^{\tout})^2]}{\mathds{E}[d_U^{\tout}]} + \mathds{1}_{\{\delta^{\head} \geq 2\}}\frac{ \mathds{E}[(d_U^{\tin})^2]}{\mathds{E}[d_U^{\tin}]} \Big)$, which is independent of $\delta^{\tail}$ and $\delta^{\head}$, except checking $\delta^{\tail} \geq$ and $\delta^{\head} \geq 2$, for the same reason as explained in the section about undirected degenerate hyperedges. For sparse regular hypergraphs with finite second moment of both the out- and in-degrees, the number of degenerate hyperedges is constant.

The expected number of multi-hyperedge pairs in a directed hypergraph is computed similarly as in an undirected hypergraph.
\\
\begin{theorem}
\label{thm:dir_E[multi-edges]}
    Let $\forall e \in E: \delta_e^{\tail} \leq \frac{1}{2}n \mathds{E}[d_U^{\tout}]$ and $\delta_e^{\head} \leq \frac{1}{2}n \mathds{E}[d_U^{\tin}]$ and let $R(\cdot)$ be as in \eqref{def:B}. Then, for a uniformly random directed hypergraph with degree sequence $\vb*{d}$,
    \begin{align*}
        &\mathds{E}[M_n] = \frac{1}{2} \sum_{e \in E} \sum_{\substack{e' \in E\backslash\{e\}:\\\delta_{e'}=\delta_e}} \frac{(n\mathds{E}[d_U^{\tout}]-2\delta_e^{\tail})!(n\mathds{E}[d_U^{\tin}]-2\delta_e^{\head})!}{(n\mathds{E}[d_U^{\tout}])!(n\mathds{E}[d_U^{\tin}])!}\\
        &\hspace{1cm}\times \sum_{\substack{\vb*{a} \in \mathds{N}^{\delta_e^{\tail}}:\\\sum_{i=1}^{\delta_e^{\tail}}ia_i = \delta_e^{\tail}}}  \frac{\delta_e^{\tail}!^2}{\prod_{k=1}^{\delta_e^{\tail}} k!^{2a_k} }
        \sum_{\alpha(\cdot) \in R(\vb*{a})} (-1)^{\sum_{\vb*{x} \in \mathds{N}^{\delta_e^{\tail}}} (\sum_{i=1}^{\delta_e^{\tail}} x_i - 1)\alpha(\vb*{x})} \\
     &\hspace{1cm}\times 
    \prod_{\vb*{y} \in \mathds{N}^{\delta_e^{\tail}}} \Bigg( \frac{1}{\alpha(\vb*{y})!} \Big(\frac{n\mathds{E}\big[\prod_{i=1}^{\delta_e^{\tail}} (\mathds{1}_{\{d_U^{\tout} \geq 2i\}}\frac{d_U^{\tout}!}{(d_U^{\tout}-2i)!})^{y_i} \big]\big(\sum_{i=1}^{\delta_e^{\tail}} y_i-1 \big)!}{\prod_{i=1}^{\delta_e^{\tail}} y_i!} \Big)^{\alpha(\vb*{y})} \Bigg)\\
    &\hspace{1cm}\times \sum_{\substack{\vb*{b} \in \mathds{N}^{\delta_e^{\head}}:\\\sum_{i=1}^{\delta_e^{\head}}ib_i = \delta_e^{\head}}}  \frac{\delta_e^{\head}!^2}{\prod_{k=1}^{\delta_e^{\head}} k!^{2b_k} }\sum_{\beta(\cdot) \in R(\vb*{b})} (-1)^{\sum_{\vb*{x} \in \mathds{N}^{\delta_e^{\head}}} (\sum_{i=1}^{\delta_e^{\head}} x_i - 1)\beta(\vb*{x})}\\
     &\hspace{1cm}\times
    \prod_{\vb*{z} \in \mathds{N}^{\delta_e^{\head}}} \Bigg( \frac{1}{\beta(\vb*{z})!} \Big(\frac{n\mathds{E}\big[\prod_{i=1}^{\delta_e^{\head}} (\mathds{1}_{\{d_U^{\tin} \geq 2i\}}\frac{d_U^{\tin}!}{(d_U^{\tin}-2i)!})^{z_i} \big]\big(\sum_{i=1}^{\delta_e^{\head}} z_i-1 \big)!}{\prod_{i=1}^{\delta_e^{\head}} z_i!} \Big)^{\beta(\vb*{z})} \Bigg).
    \end{align*}
\end{theorem}

The proof of Theorem \ref{thm:dir_E[multi-edges]} is in Section \ref{section:pf_dir_multi-edges}. For digraphs, where $\forall e \in E: \delta_e^{\tail}=\delta_e^{\head}=1$ Theorem \ref{thm:dir_E[multi-edges]} reduces to
\begin{align*}
    \mathds{E}[M_n] &= \frac{n (\mathds{E}[(d_U^{\tout})^2] - \mathds{E}[d_U^{\tout}]) (\mathds{E}[(d_U^{\tin})^2] - \mathds{E}[d_U^{\tin}])}{2\mathds{E}[d_U^{\tin}](n \mathds{E}[d_U^{\tout}]-1)} \\
    &= \frac{(\mathds{E}[(d_U^{\tout})^2] - \mathds{E}[d_U^{\tout}]) (\mathds{E}[(d_U^{\tin})^2] - \mathds{E}[d_U^{\tin}])}{2\mathds{E}[d_U^{\tout}] \mathds{E}[d_U^{\tin}]}(1+o(1)),
\end{align*}
which aligns with the result in \cite{angel2016}.

For regular hypergraphs, where all hyperedges have the same degrees, we present the following asymptotic result.\\
\begin{lemma}
\label{lemma:dir_E[multi_edges]_asymp}
If $\forall e \in E: \delta_e^{\tail} = \delta^{\tail}\in O(1)$ and $\delta_e^{\head} = \delta^{\head}\in O(1)$ and
\begin{enumerate}
    \item $\mathds{E}[(d_U^{\tout})^{2\delta^{\tail}}], \mathds{E}[(d_U^{\tin})^{2\delta^{\head}}] \in o(n)$
    \item $\exists c>0: \mathds{P}(d_U^{\tout} \geq 2\delta^{\tail}),\mathds{P}(d_U^{\tin} \geq 2\delta^{\head}) \geq c$ 
\end{enumerate}
then
\begin{align*}
    \mathds{E}[M_n] &= \frac{(\delta^{\tail}-1)!(\delta^{\head}-1)!}{2} n^2 \mathds{E}[d_U^{\tout}] \mathds{E}[d_U^{\tin}] \\
    &\hspace{0.5cm} \times \Big( \frac{\mathds{E}[(d_U^{\tout})^2] - \mathds{E}[d_U^{\tout}]}{n \mathds{E}[d_U^{\tout}]^2} \Big)^{\delta^{\tail}} \Big(  \frac{\mathds{E}[(d_U^{\tin})^2] - \mathds{E}[d_U^{\tin}]}{n \mathds{E}[d_U^{\tin}]^2} \Big)^{\delta^{\head}} (1+o(1)).
\end{align*}
\end{lemma}

The asymptotic bound in Lemma~\ref{lemma:dir_E[multi_edges]_asymp} highlights the scaling behavior of multi-hyperedge pairs in regular directed hypergraphs under mild moment conditions on the degree distribution. The expression has the same components as the expression found for undirected multi-hyperedge pairs: the factor $n^2 \mathds{E}[d_U^{\tout}] \mathds{E}[d_U^{\tin}]$ scales as the number of hyperedge pairs, and the factors $\Big( \frac{\mathds{E}[(d_U^{\tout})^2] - \mathds{E}[d_U^{\tout}]}{n \mathds{E}[d_U^{\tout}]^2} \Big)^{\delta^{\tail}}$ and $ \Big(  \frac{(\mathds{E}[d_U^{\tin})^2 - \mathds{E}[d_U^{\tin}]]}{n \mathds{E}[d_U^{\tin}]^2} \Big)^{\delta^{\head}}$ compute the probability of each such pair to form a multi-hyperedge pair. 

Similarly as for the undirected multi-hyperedge pairs, the expected number of multi-hyperedge pairs decreases when $\delta^{\tail}$ or $\delta^{\head}$ grows, as in both cases more vertices need to be picked twice. Asymptotically, we obtain $\mathds{E}[M_n] = \Theta \Big( \big(\frac{ \mathds{E}[(d_U^{\tout})^2]}{n\mathds{E}[d_U^{\tout}]^2} \big)^{\delta^{\tail}} \big(\frac{ \mathds{E}[(d_U^{\tin})^2]}{n\mathds{E}[d_U^{\tin}]^2} \big)^{\delta^{\head}} n^2 \mathds{E}[d_U^{\tout}] \mathds{E}[d_U^{\tin}] \Big)$, where $\delta^{\tail}$ and $\delta^{\head}$ appear as a power for the same reason as explained in the section about undirected multi-hyperedge pairs. We observe that the fraction of multi-hyperedge pairs vanishes as $n$ grows, since the moments of the vertex degrees grow slower than $n$. In addition, for sparse hypergraphs with bounded degrees, the probability of observing multi-hyperedges vanishes as $n \rightarrow \infty$.

The following theorem characterizes the expected number of self-loops:
\\
\begin{theorem}
\label{thm:E[sl]}
    Let $R(\cdot)$ be as in \eqref{def:B} and let $E^* = \{e \in E: \delta_e^{\tail} = \delta_e^{\head}\}$. We denote $\delta_e = \delta_e^{\tail}=\delta_e^{\head}$ for $e \in E^*$. Then, for a uniformly random directed hypergraph with degree sequence $\vb*{d}$,
    \begin{align}
    \label{eq:E[S]}
         \mathds{E}[S_n] &= \sum_{e \in E^*} \frac{(n\mathds{E}[d_U^{\tin}]-\delta_e)!(n\mathds{E}[d_U^{\tout}]-\delta_e)!}{(n\mathds{E}[d_U^{\tin}])!(n\mathds{E}[d_U^{\tout}])!} \sum_{\substack{\vb*{a} \in \mathds{N}^{\delta_e}:\\\sum_{i=1}^{\delta_e} ia_i = \delta_e}}  \frac{\delta_e!^2}{\prod_{k=1}^{\delta_e} k!^{2a_k} } \sum_{\alpha(\cdot) \in R(\vb*{a})} (-1)^{\sum_{\vb*{x} \in \mathds{N}^{\delta_e}} (\sum_{i=1}^{\delta_e} x_i - 1)\alpha(\vb*{x})}\nonumber\\
     &\hspace{1cm}
    \times\prod_{\vb*{y} \in \mathds{N}^{\delta_e}} \Bigg( \frac{1}{\alpha(\vb*{y})!} \Big(\frac{n\mathds{E}\big[\prod_{i=1}^{\delta_e} (\mathds{1}_{\{d_U^{\tout},d_U^{\tin} \geq i\}}\frac{d_U^{\tout}!}{(d_U^{\tout}-i)!}\frac{d_U^{\tin}!}{(d_U^{\tin}-i)!})^{y_i} \big] \big(\sum_{i=1}^{\delta_e} y_i-1 \big)!}{\prod_{i=1}^{\delta_e} y_i!} \Big)^{\alpha(\vb*{y})} \Bigg).
    \end{align} 
\end{theorem}
The proof of Theorem \ref{thm:E[sl]} is in Section~\ref{section:pf_sl}.

For digraphs, where $\forall e \in E: \delta_e^{\tail}=\delta_e^{\head}=1$, Theorem \ref{thm:E[sl]} reduces to
\begin{align*}
    \mathds{E}[S_n] &= \frac{\mathds{E}[d_U^{\tout}d_U^{\tin}]}{\mathds{E}[d_U^{\tin}]},
\end{align*}
which aligns with the result in \cite{angel2016}. 

For regular hypergraphs, where every hyperedge has the same degrees, we present the following asymptotic result:
\\
\begin{lemma}
\label{lemma:E[sl]_asymp}
If $\forall e \in E: \delta_e^{\tail} = \delta_e^{\head} = \delta \in O(1)$ and
\begin{enumerate}
    \item $\mathds{E}[(d_U^{\tin}d_U^{\tout})^{\delta}] \in o(n)$
    \item $\exists c>0: \mathds{P}(d_U^{\tout},d_U^{\tin} \geq \delta) \geq c$
\end{enumerate}
then
\begin{align*}
    \mathds{E}[S_n] = (\delta-1)! n \mathds{E}[d_U^{\tin}] \Bigg( \frac{\mathds{E}[d_U^{\tin}d_U^{\tout}]}{n\mathds{E}[d_U^{\tin}]^2} \Bigg)^{\delta} (1+o(1)).
\end{align*}
\end{lemma}

The asymptotic bound in Lemma \ref{lemma:E[sl]_asymp} highlights the scaling behavior of self-loops in regular hypergraphs under mild moment conditions on the degree distribution. Observe that the factor $n \mathds{E}[d_U^{\tin}]$ equals the number of hyperedges, and the factor $(\delta-1)!\Big( \frac{\mathds{E}[d_U^{\tin}d_U^{\tout}]}{n\mathds{E}[d_U^{\tin}]^2} \Big)^{\delta}$ computes the probability of each such hyperedge to form a self-loop. The expected number of self-loops decreases when $\delta$ grows, as for larger hyperedges to form a self-loop, more vertices need to be picked in both the head and the tail of the hyperedge. Asymptotically, we obtain $\mathds{E}[S_n] = \Theta \Big( \big(\frac{ \mathds{E}[d_U^{\tout} d_U^{\tin}]}{n\mathds{E}[d_U^{\tin}]^2} \big)^{\delta} n \mathds{E}[d_U^{\tin}] \Big)$. Here, $\delta$ appears as a power, since all of the $\delta$ vertices in the tail of a hyperedge also need to appear in the head of the hyperedge in order for a self-loop to be formed. Observe that the expression for the expected number of multi-hyperedge pairs in an undirected hypergraph has a similar form, since requiring a tail and head of a hyperedge to be equal is similar as requiring two undirected hyperedges to be equal. We observe that the fraction of self-loops vanishes as $n$ grows, since the moments of the vertex degrees grow slower than $n$. In addition, for sparse hypergraphs with bounded degrees, the probability of observing a self-loop vanishes as $n \rightarrow \infty$.

Lemma \ref{lemma:E[sl]_asymp} shows that the expected number of self-loops is largely determined by the correlation between the out- and in-degrees of vertices. This makes sense, as a self-loop with some vertex $v$ in the tail and head is only likely if this vertex has both a large out-degree and a large in-degree.

Lastly, we consider weak self-loops. Again, the two multiplicity vectors $\vb*{a}$ and $\vb*{b}$ are needed to describe the multiplicities of the tail resp. head vertices in a hyperedge. In addition, it is now relevant whether any vertex in the tail of a hyperedge, described by $\vb*{a}$, also appears in the head, described by $\vb*{b}$, as this results in a weak self-loop. Therefore, we extend the definition of $R(\cdot)$ in \eqref{def:B} to distinguish between all ways the vertices described by $\vb*{a}$ and $\vb*{b}$ may coincide:
\begin{align}
\label{def:B_hat}
    \hat{R}(\vb*{a},\vb*{b}) = \{\gamma(\cdot): \sum_{\substack{\vb*{y} \in \mathds{N}^{\delta_e^{\tail}}\\ \vb*{z} \in \mathds{N}^{\delta_e^{\head}}}} y_i\gamma(\vb*{y}, \vb*{z}) = a_i \, \forall i \in [\delta_e^{\tail}] \land \sum_{\substack{\vb*{y} \in \mathds{N}^{\delta_e^{\tail}}\\ \vb*{z} \in \mathds{N}^{\delta_e^{\head}}}} z_j\gamma(\vb*{y}, \vb*{z}) = b_j  \, \forall j \in [\delta_e^{\head}]\}.
\end{align}
Any $\gamma(\cdot) \in \hat{R}(\vb*{a},\vb*{b})$ describes how many vertices with a specific multiplicity coincide. For instance, if $\delta_e^{\tail}=\delta_e^{\head}=2$, $\vb*{a} = (0,1)$ and $\vb*{b} = (2,0)$, then the hyperedge is of the form $(\{v_1,v_1\}, \{v_2,v_3\})$. Let the vectors $\vb*{x} \in \mathds{N}^{\delta_e^{\tail}}$ and $\vb*{y} \in \mathds{N}^{\delta_e^{\head}}$ be subsets of coinciding vertices in the tail and head of hyperedge $e$ with specific multiplicity, where $x_i$ is the number of vertices with multiplicity $i$ in the tail and $y_j$ is the number of vertices with multiplicity $j$ in the head. The possible coinciding patterns are
\begin{itemize}
    \item $\vb*{x} = (0,1)$, $\vb*{y} = (0,0)$ ($v_1$ is unique)
    \item $\vb*{x} = (0,0)$, $\vb*{y} = (1,0)$ ($v_2$ or $v_3$ is unique)
    \item $\vb*{x} = (0,1)$, $\vb*{y} = (1,0)$ ($v_1=v_2$ or $v_1=v_3$)
    \item $\vb*{x} = (0,0)$, $\vb*{y} = (2,0)$ ($v_2=v_3$)
    \item $\vb*{x} = (0,1)$, $\vb*{y} = (2,0)$ ($v_1=v_2=v_3$).
\end{itemize}
For any such $\vb*{x}$ and $\vb*{y}$, $\gamma(\vb*{x}, \vb*{y})$ counts how many such coinciding sets exist in the hyperedge $e$. Note that $\gamma(\cdot)$ must satisfy
\begin{align*}
    \sum_{\vb*{x} \in \mathds{N}^{\delta_e^{\tail}}} \sum_{\vb*{y} \in \mathds{N}^{\delta_e^{\head}}} x_i \gamma(\vb*{x}, \vb*{y}) = a_i \, \, \forall i \in [\delta_e^{\tail}], \qquad \sum_{\vb*{x} \in \mathds{N}^{\delta_e^{\tail}}} \sum_{\vb*{y} \in \mathds{N}^{\delta_e^{\head}}} y_j \gamma(\vb*{x}, \vb*{y}) = b_j \, \, \forall j \in [\delta_e^{\head}].
\end{align*}
Then, $\gamma(\cdot)$ has to be one of the following:
\begin{itemize}
    \item $\gamma((0,1),(0,0)) = 1, \gamma((0,0),(1,0)) = 2$ ($v_1,v_2$ and $v_3$ are unique)
    \item $\gamma((0,1),(0,0)) = 1, \gamma((0,0),(2,0)) = 1$ ($v_1$ is unique, $v_2=v_3$)
    \item $\gamma((0,1),(1,0)) = 1, \gamma((0,0),(1,0)) = 1$ ($v_3$ is unique and $v_1=v_2$, or $v_2$ is unique and $v_1=v_3$)
    \item $\gamma((0,1),(2,0)) = 1$ ($v_1=v_2=v_3$).
\end{itemize}

The following theorem characterizes the expected number of weak self-loops:
\\
\begin{theorem}
\label{thm:E[weak-sl]}
    Let $\hat{R}(\cdot)$ be as in \eqref{def:B_hat}. For a uniformly random directed hypergraph with degree sequence $\vb*{d}$,
    \begin{align*}
         \mathds{E}[WS_n] &= |E| - \sum_{e \in E} \frac{(n\mathds{E}[d_U^{\tout}]-\delta_e^{\tail})!(n\mathds{E}[d_U^{\tin}]-\delta_e^{\head})!}{(n\mathds{E}[d_U^{\tout}])!(n\mathds{E}[d_U^{\tin}])!}\sum_{\substack{\vb*{a} \in \mathds{N}^{\delta_e^{\tail}}:\\ \sum_{i=1}^{\delta_e^{\tail}} ia_i = \delta_e^{\tail}}}  \frac{\delta_e^{\tail}!}{\prod_{k=1}^{\delta_e^{\tail}} k!^{a_k}} \sum_{\substack{\vb*{b} \in \mathds{N}^{\delta_e^{\head}}:\\ \sum_{i=1}^{\delta_e^{\head}} ib_i = \delta_e^{\head}}}  \frac{\delta_e^{\head}!}{\prod_{k=1}^{\delta_e^{\head}} k!^{b_k}} \nonumber \\
    &\hspace{0.5cm} \times \sum_{\gamma(\cdot) \in \hat{R}(\vb*{a},\vb*{b})} (-1)^{\sum{\vb*{g} \in \mathds{N}^{\delta_e^{\tail}}, \vb*{h} \in \mathds{N}^{\delta_e^{\head}}}( \sum_{i=1}^{\delta_e^{\head}} g_i + \sum_{j=1}^{\delta_e^{\head}} h_j -1) \gamma(\vb*{g},\vb*{h})} \nonumber \\
    &\hspace{0.5cm}  \times\prod_{\substack{\vb*{y} \in \mathds{N}^{\delta_e^{\tail}}\\ \vb*{z} \in \mathds{N}^{\delta_e^{\head}}}} \Bigg( \frac{1}{\gamma(\vb*{y},\vb*{z})!} \Big(\frac{n\mathds{E}\big[\prod_{i=1}^{\delta_e^{\tail}} (f(d_U^{\tout},i))^{y_i} \prod_{j=1}^{\delta_e^{\head}} (f(d_U^{\tin} ,j)^{z_j} \big] \big(\sum_{i=1}^{\delta_e^{\tail}} y_i + \sum_{j=1}^{\delta_e^{\head}}z_j-1 \big)!}{\prod_{i=1}^{\delta_e^{\tail}} y_i! \prod_{j=1}^{\delta_e^{\head}} z_j!} \Big)^{\gamma(\vb*{y},\vb*{z})} \Bigg),
    \end{align*} 
    where $f(u,i) = \mathds{1}_{\{u \geq i\}} \frac{u!}{(u-i)!}$
\end{theorem}
The proof of Theorem \ref{thm:E[weak-sl]} is in Section \ref{section:pf_weak_sl}. For digraphs, where $\forall e \in E: \delta_e^{\tail} = \delta_e^{\head}=1$, Theorem \ref{thm:E[weak-sl]} describes the number of self-loops and reduces to
\begin{align*}
    \mathds{E}[WS_n] &= \frac{\mathds{E}[d_U^{\tout}d_U^{\tin}]}{\mathds{E}[d_U^{\tin}]},
\end{align*}
which aligns with the result in \cite{angel2016}. 

For regular hypergraphs, where every hyperedge has the same degree, we present the following asymptotic result:
\\
\begin{lemma}
\label{lemma:E[weak-sl]_asymp}
If $\forall e \in E: \delta_e^{\tail} = \delta^{\tail}\in O(1)$ and $\delta_e^{\head} = \delta^{\head}\in O(1)$ and 
\begin{enumerate}
    \item $\mathds{E}[(d_U^{\tout})^{\delta^{\tail}}(d_U^{\tin})^{\delta^{\head}}] \in o(n)$
    \item $\exists c>0 \textnormal{ s.t.}    \lim_{n \rightarrow \infty} \mathds{P}(d_U^{\tout} \geq \delta^{\tail}), \mathds{P}(d_U^{\tin} \geq \delta^{\head}) \geq c$
\end{enumerate}
then
\begin{align*}
    \mathds{E}[WS_n] &= \Big( \frac{\delta^{\tail} + \delta^{\head}-2}{2} + \frac{\delta^{\head} \mathds{E}[d_U^{\tout}d_U^{\tin}]}{\mathds{E}[d_U^{\tin}]} \Big)(1+o(1)).
\end{align*}
\end{lemma}

The bound in Lemma \ref{lemma:E[weak-sl]_asymp} provides the scaling behavior of weak self-loops in regular directed hypergraphs, under mild moment conditions on the degree distribution. The expected number of weak self-loops grows with $\delta^{\tail}$ and with $\delta^{\head}$, as larger hyperedges have a higher chance to contain the same vertex in the tail and the head. In addition, the fraction of weak self-loops vanishes as $n$ grows, since the moments of the vertex degrees grow slower than $n$. Asymptotically, we obtain $\mathds{E}[WS_n] = \Theta \Big( \frac{ \mathds{E}[d_U^{\tout} d_U^{\tin}]}{\mathds{E}[d_U^{\tin}]} \Big)$. This is independent of $\delta$, as the probability of creating a weak self-loop is governed by the probability of picking both an out-stub and an in-stub of the same vertex when picking an out-stub and an in-stub uniformly at random. For each hyperedge, the number of tries to pick an out-stub and an in-stub from the same vertex is a combinatorial constant depending on $\delta$, which can be ignored since we assume $\delta=O(1)$. The expression is similar to the expression for the expected number of degenerate hyperedges in an undirected hypergraph, as picking an out- and in-stub from the same vertex is similar to picking two undirected stubs from the same vertex. In addition, for sparse regular hypergraphs with finite $\mathds{E}[d_U^{\tout} d_U^{\tin}]$, the number of weak self-loops is constant.

\section{Conclusion and Discussion}
\label{section:conclusion}
In this work, we computed the expected number of degenerate hyperedges and multi-hyperedge pairs for both undirected and directed hypergraphs, and the expected number of self-loops and weak self-loops for directed hypergraphs. For every statistic, we obtained an exact result, valid in all regimes. These exact results can be used as null models for the analysis of a network, or can be studied to find the asymptotic behavior of these statistics. In this work, for regular hypergraphs, where the first number of vertex degree moments are sublinear in $n$ and a sufficient number of vertices has a high enough degree, we presented and interpreted this asymptotic behavior. In particular, in undirected hypergraphs, $\mathds{E}[DH_n]$ is governed by the first and second moment of the vertex degree, and $\mathds{E}[M_n]$ is governed by the first and second moment of the vertex degree as well as the number of vertices, and the hyperedge size appears as a power. In directed hypergraphs, we observe similar asymptotics, where the out- and in-degree vertex moments are used. In addition, $\mathds{E}[S_n]$ is governed by the first moment of the product of the out- and in-degree of vertices, as well as the first moment of the vertex in-degrees, $n$, and the size of the hypergraphs, which appears as a power. Lastly, $\mathds{E}[WS_n]$ is governed by the first moment of the product of the out- and in-degree of vertices, as well as the first moment of the vertex in-degrees. For all statistics, the fraction of these hyperedges converges to 0 as $n$ grows to infinity. 

Several avenues for further research remain open. One natural direction is to investigate the limiting distribution of the number of degenerate hyperedges, multi-hyperedge pairs, self-loops, and weak self-loops, rather than only their expectations. In analogy with the graph case, one may expect Poisson limits to arise when the hyperedge degree and the vertex degree have sufficiently many bounded moments. Bollobás~\cite{Bollobas1980} and Janson~\cite{Janson2009} established that the numbers of self-loops and multi-edges in the graph configuration model converge to independent Poisson random variables under finite second-moment degree conditions. Molloy and Reed~\cite{Molloy1995} further showed that heavy-tailed degree distributions can lead to a non-negligible probability of degeneracies, affecting the simplicity of the graph. Our results demonstrate that in hypergraphs, such statistics have constant expectation as long as specific moments of the degree distribution are finite. In particular, the conditions in Lemmas \ref{lemma:E[deg-edges]_asymp}, \ref{lemma:E[multi_edges]_asymp}, \ref{lemma:dir_E[deg-edges]_asymp}, \ref{lemma:dir_E[multi_edges]_asymp}, \ref{lemma:E[sl]_asymp} and \ref{lemma:E[weak-sl]_asymp} play the same role as the finite-moment assumptions in the graph case, ensuring that the considered structures remain rare. The main difference is that here, to ensure boundedness of these statistics, both the vertex degree, and the hyperedge degrees need to be bounded.

Another direction concerns hypergraphs with heavy-tailed degree sequences, where the behaviour of self-loops and multi-edges may differ significantly, like has been shown for graphs \cite{angel2016, vanDerHofstad2016}. It would be interesting to investigate the behavior of the distribution of these statistics.

Finally, extending the analysis to more complex network statistics, such as clustering would be interesting. Even for graphs, several definitions of clustering exist, and for hypergraphs, even more such definitions exist, some considering statistics based on the clustering of two vertices~\cite{Miyashita2023,Ha2024}, while others focus on hyperedges rather than vertices, reflecting pairwise relationships within hyperedges~\cite{Miyashita2024}. Other approaches have emphasized the role of large hyperedges in clustering~\cite{Purkait2014}. It would be interesting to see the behavior of these different variants of clustering coefficient in random graphs. Since all results in this work follow from Lemmas \ref{lemma:main_lemma} and \ref{lemma:general_order}, there may be other statistics that can easily be analyzed with using these lemmas.

\section{Proofs}
\label{section:proofs}

\stepcounter{theorem}
We first show two lemmas which will be used in many of the main proofs.
Lemma~\ref{lemma:main_lemma} provides a general combinatorial identity that allows us to rewrite sums over all
$d_e$-tuples of vertices where vertices may coincide and thus create multiplicities in terms of multiplicity
vectors and corresponding moments of the underlying degree variables. The left-hand side represents a
direct enumeration over all possible assignments of vertices to a hyperedge, weighted by the multiplicity-correcting
factor $D(\vb*{v})$ and by the functions $f^{(i)}$ evaluated at the distinct vertices. The lemma shows that this can be rewritten as a sum indexed by the vectors $\vb*{a}$ and $\vb*{b}$ describing how many vertices appear with each
multiplicity in the tail resp. head of a hyperedge, together with the combinatorial objects $\gamma(\cdot)$ that encode how these multiplicities interact. This identity will be used repeatedly in later sections to derive exact expressions for the expected number of degenerate hyperedges, multi-hyperedge pairs and self-loops, as well as their asymptotics. Lemma~\ref{lemma:general_order} establishes the asymptotically dominant terms in Lemma~\ref{lemma:main_lemma}.

Recall that $m_i(\vb*{v})$ denotes the multiplicity of element $i$ in vector $\vb*{v}$.
\\
\begin{lemma}
\label{lemma:main_lemma}
Let $\delta_1,\delta_2 \in \mathds{N}$. Given (possibly non-unique) vertices $v_1,v_2,\hdots,v_{\delta_1} \in V$, let the unique vertices be denoted by $v'_1,v'_2,\hdots,v'_{u(\vb*{v})}$, for some $u(\vb*{v}) \in \mathds{N}$. Similarly, let the unique vertices in $w_1,w_2,\hdots,w_{\delta_2}$ be denoted by $w'_1,w'_2,\hdots,w'_{u(\vb*{w})}$. Let $D(\vb*{v})=\frac{\delta_1!}{\prod_{i=1}^{u(\vb*{v})} m_{v'_i}(\vb*{v})!}$ and $D(\vb*{w})=\frac{\delta_2!}{\prod_{i=1}^{u(\vb*{w})} m_{w'_i}(\vb*{w})!}$. Let $\hat{R}(\cdot)$ be as in \eqref{def:B_hat} and let $U$ denote a uniformly chosen vertex in $V$. For any functions $f_1^{(i)}(v):V \rightarrow \mathds{N}$, $f_2^{(i)}(v):V \rightarrow \mathds{N}$ and $w \in \mathds{N}$,
    \begin{align}
        \label{eq:main_lemma}
        &\sum_{\vb*{v} \in V^{\delta_1}}  \sum_{\vb*{w} \in (V\backslash\{v_1,\hdots,v_{\delta_1}\})^{\delta_2}} (D(\vb*{v})D(\vb*{w}))^w \prod_{i=1}^{u(\vb*{v})} f_1^{(m_{v'_i}(\vb*{v}))}(v'_i) \prod_{j=1}^{u(\vb*{w})} f_2^{(m_{w'_j}(\vb*{w}))}(w'_j) \nonumber \\
        &=\sum_{\substack{\vb*{a} \in \mathds{N}^{\delta_1}:\\ \sum_{i=1}^{\delta_1} ia_i = \delta_1}}  \frac{\delta_1!^{1+w}}{\prod_{k=1}^{\delta_1} k!^{(1+w)a_k}} \sum_{\substack{\vb*{b} \in \mathds{N}^{\delta_2}:\\ \sum_{i=1}^{\delta_2} ib_i = \delta_2}}  \frac{\delta_2!^{1+w}}{\prod_{k=1}^{\delta_2} k!^{(1+w)b_k} } \nonumber \\
    &\hspace{1cm} \times \sum_{\gamma(\cdot) \in \hat{R}(\vb*{a},\vb*{b})} (-1)^{\sum{\vb*{g} \in \mathds{N}^{\delta_1}, \vb*{h} \in \mathds{N}^{\delta_2}}( \sum_{i=1}^{\delta_1} g_i + \sum_{j=1}^{\delta_2} h_j -1) \gamma(\vb*{g},\vb*{h})} \nonumber \\
    &\hspace{1cm} \times  \prod_{\substack{\vb*{y} \in \mathds{N}^{\delta_1}\\ \vb*{z} \in \mathds{N}^{\delta_2}}} \Bigg( \frac{1}{\gamma(\vb*{y},\vb*{z})!} \Big(\frac{n\mathds{E} \big[\prod_{i=1}^{\delta_1} (f_1^{(i)}(U))^{y_i} \prod_{j=1}^{\delta_2} (f_2^{(j)}(U))^{z_j} \big] \big(\sum_{i=1}^{\delta_1} y_i + \sum_{j=1}^{\delta_2} z_j-1 \big)!}{\prod_{i=1}^{\delta_1} y_i! \prod_{j=1}^{\delta_2} z_j!} \Big)^{\gamma(\vb*{y},\vb*{z})} \Bigg).
    \end{align}
\end{lemma}

The proof of Lemma \ref{lemma:main_lemma} is in Appendix \ref{app:pf_main_lemma}. The following lemma identifies the asymptotically dominant term in Lemma~\ref{lemma:main_lemma}. Let
\begin{align}\label{eq:Hterms} 
&\hat{H}(\vb*{a},\vb*{b},\gamma(\cdot)) \nonumber \\
&= \frac{\delta_1!^{1+w}}{\prod_{k=1}^{\delta_1} k!^{(1+w)a_k}}  \frac{\delta_2!^{1+w}}{\prod_{k=1}^{\delta_2} k!^{(1+w)b_k} } (-1)^{\sum{\vb*{g} \in \mathds{N}^{\delta_1}, \vb*{h} \in \mathds{N}^{\delta_2}}( \sum_{i=1}^{\delta_1} g_i + \sum_{j=1}^{\delta_2} h_j -1) \gamma(\vb*{g},\vb*{h})} \nonumber \\
    &\quad \times  \prod_{\substack{\vb*{y} \in \mathds{N}^{\delta_1}\\ \vb*{z} \in \mathds{N}^{\delta_2}}} \Bigg( \frac{1}{\gamma(\vb*{y},\vb*{z})!} \Big(\frac{n\mathds{E}\big[\prod_{i=1}^{\delta_1} (f_1^{(i)}(U))^{y_i} \prod_{j=1}^{\delta_2} (f_2^{(j)}(U))^{z_j} \big]\big(\sum_{i=1}^{\delta_1} y_i + \sum_{j=1}^{\delta_2} z_j-1 \big)!}{\prod_{i=1}^{\delta_1} y_i! \prod_{j=1}^{\delta_2} z_j!} \Big)^{\gamma(\vb*{y},\vb*{z})} \Bigg).
\end{align}

\begin{lemma}
\label{lemma:general_order}
Let $r \in [\delta_1], t \in [\delta_2]$. For any functions $f_1^{(i)}(v):V \rightarrow \mathds{N},f_2^{(i)}(v):V \rightarrow \mathds{N}$ and $w \in \mathds{N}$, if
\begin{enumerate}
    \item $\delta_1, \delta_2 \in O(1)$
    \item $\forall (\vb*{y}, \vb*{z}) \in \mathds{N}^{\delta_1}\times \mathds{N}^{\delta_2}$ such that $\sum_{i=1}^{\delta_1} iy_i \leq \delta_1,  \sum_{j=1}^{\delta_2} jz_j \leq \delta_2$ and $ \|\vb*{y}\|_0 + \|\vb*{z}\|_0 \ge 2$:
    $$\mathds{E}[\prod_{i=1}^{\delta_1} (f_1^{(i)}(U))^{y_i}\prod_{j=1}^{\delta_2} (f_2^{(j)}(U))^{z_i}] \in o(n)$$
    \item $\forall k \in [\delta_1], \forall l \in [\delta_2]:$
    $$\mathds{E}[(f_1^{(k)}(U))^{\lfloor \frac{\delta_1}{k} \rfloor}],\mathds{E}[(f_2^{(l)}(U))^{\lfloor \frac{\delta_2}{l} \rfloor}] \in o(n)$$
    \item $\forall k \in [\delta_1]$ s.t. $\exists v \in V: f_1^{(k)}(v) > 0 $ and $ \forall l \in [\delta_2] $ s.t. $\exists v \in V: f_2^{(l)}(v) > 0 :\\
    \exists c>0 \textnormal{ s.t. }  \lim_{n \rightarrow \infty}  \mathds{P}(f_1^{(k)}(U) \geq 1), \lim_{n \rightarrow \infty} \mathds{P}(f_2^{(l)}(U) \geq 1) \geq c$
\end{enumerate}
then
\begin{align*}
 &\sum_{\substack{\vb*{a} \in \mathds{N}^{\delta_1}:\\\sum_{i=1}^{\delta_1} ia_i = \delta_1 }} \sum_{\substack{\vb*{b} \in \mathds{N}^{\delta_2}:\\\sum_{i=1}^{\delta_2} ib_i = \delta_2}} \sum_{\gamma(\cdot) \in \hat{R}(\vb*{a},\vb*{b})} \hat{H}(\vb*{a},\vb*{b},\gamma(\cdot)) \\
 &= (\delta_1!\delta_2!)^w (n\mathds{E}[f_1^{(1)}(U) ])^{\delta_1} (n\mathds{E}[f_2^{(1)}(U) ])^{\delta_2} \nonumber\\
 &\hspace{0.5cm}+ \Big( -\frac{1}{2} \frac{\delta_1!^{1+w}\delta_2!^w}{(\delta_1-2)!}(n\mathds{E}[f_1^{(1)}(U)])^{\delta_1-2} n\mathds{E}[(f_1^{(1)}(U))^2] (n\mathds{E}[f_2^{(1)}(U)])^{\delta_2} \nonumber\\
 &\hspace{1.15cm} -\frac{1}{2} \frac{\delta_1!^w \delta_2!^{1+w}}{(\delta_2-2)!} (n\mathds{E}[f_1^{(1)}(U)])^{\delta_1} (n\mathds{E}[f_2^{(1)}(U)])^{\delta_2-2} n\mathds{E}[(f_2^{(1)}(U))^2] \nonumber\\
 &\hspace{1.15cm} - \frac{(\delta_1!\delta_2!)^{1+w}}{(\delta_1-1)!(\delta_2-1)!}  (n\mathds{E}[f_1^{(1)}(U)])^{\delta_1-1} (n\mathds{E}[f_2^{(1)}(U)])^{\delta_2-1} n\mathds{E}[f_1^{(1)}(U)f_2^{(1)}(U)] \nonumber\\
 &\hspace{1.15cm} + \frac{1}{2^{1+w}} \frac{\delta_1!^{1+w}\delta_2!^w}{(\delta_1-2)!} (n\mathds{E}[f_1^{(1)}(U)])^{\delta_1-2} n\mathds{E}[f_1^{(2)}(U)] (n\mathds{E}[f_2^{(1)}(U)])^{\delta_2} \nonumber\\
 &\hspace{1.15cm}+ \frac{1}{2^{1+w}} \frac{\delta_1!^w \delta_2!^{1+w}}{(\delta_2-2)!} (n\mathds{E}[f_1^{(1)}(U)])^{\delta_1} (n\mathds{E}[f_2^{(1)}(U)])^{\delta_2-2} n\mathds{E}[f_2^{(2)}(U)] \Big) (1+o(1))\\
   &= (\delta_1!\delta_2!)^w (n\mathds{E}[f_1^{(1)}(U) ])^{\delta_1} (n\mathds{E}[f_2^{(1)}(U) ])^{\delta_2}(1+o(1)).
\end{align*}
\end{lemma}
The proof of Lemma \ref{lemma:general_order} is in Appendix \ref{app:pf_lemma_general_order}. For most results, we can apply simpler corollaries of Lemmas \ref{lemma:main_lemma} and \ref{lemma:general_order}, by using $\delta_2=0$. This results in Corollary \ref{cor:main_cor} and \ref{cor:general_order}.
\\
\begin{corollary}
\label{cor:main_cor}
Let $\delta \in \mathds{N}$. Given vertices $v_1,v_2,\hdots,v_{\delta} \in V$, which may not all be unique, let the unique vertices be denoted by $v'_1,v'_2,\hdots,v'_{u(\vb*{v})}$, for some $u(\vb*{v}) \in \mathds{N}$. Let $D(\vb*{v})=\frac{\delta!}{\prod_{i=1}^{u(\vb*{v})} m_{v'_i}(\vb*{v})!}$, let $R(\cdot)$ be as in \eqref{def:B} and let $U$ denote a uniformly chosen vertex in $V$. For any function $f^{(i)}(v):V \rightarrow \mathds{N}$ and $w \in \mathds{N}$,
\begin{align}
     \sum_{\vb*{v} \in V^{\delta}} (D(\vb*{v}))^w \prod_{i=1}^{u(\vb*{v})} f^{(m_{v'_i}(\vb*{v}))}(v'_i) &=     \sum_{\substack{\vb*{a} \in \mathds{N}^{\delta}:\\\sum_{i=1}^{\delta} ia_i = \delta}}  \frac{\delta!^{1+w}}{\prod_{k=1}^{\delta} k!^{(1+w)a_k} } \sum_{\alpha(\cdot) \in R(\vb*{a})} (-1)^{\sum_{\vb*{x} \in \mathds{N}^{\delta}} (\sum_{i=1}^{\delta} x_i - 1)\alpha(\vb*{x})} \nonumber\\
     &\hspace{1cm} \times
    \prod_{\vb*{y} \in \mathds{N}^{\delta}} \Bigg( \frac{1}{\alpha(\vb*{y})!} \Big(\frac{n\mathds{E}\big[\prod_{i=1}^{\delta} (f^{(i)}(U))^{y_i}\big]\big(\sum_{i=1}^{\delta} y_i-1\big)!}{\prod_{i=1}^{\delta} y_i!} \Big)^{\alpha(\vb*{y})} \Bigg).
\end{align}
\end{corollary}
\begin{proof}
    The result follows from using $\delta_1=\delta$, $\delta_2=0$, $f_1^{(i)}(U)=f^{(i)}(U)$ and $f_2^{(i)}(U)=0$ in Lemma \ref{lemma:main_lemma}.
\end{proof}

Let 
\begin{align}
\label{eq:H_hat_terms}
    H(\vb*{a},\alpha(\cdot))&  = \frac{d!^{1+w}}{\prod_{k=1}^{d} k!^{(1+w)a_k} } (-1)^{\sum_{\vb*{x} \in \mathds{N}^{d}} (\sum_{i=1}^{d} x_i - 1)\alpha(\vb*{x})} \nonumber\\
    & \quad \times
    \prod_{\vb*{y} \in \mathds{N}^{d}} \Bigg( \frac{1}{\alpha(\vb*{y})!} \Big(\frac{n\mathds{E}[\prod_{i=1}^{d} (f^{(i)}(U))^{y_i}](\sum_{i=1}^{d} y_i-1)!}{\prod_{i=1}^{d} y_i!} \Big)^{\alpha(\vb*{y})} \Bigg).
\end{align}

\begin{corollary}
\label{cor:general_order}
For any functions $f^{(i)}(v):V \rightarrow \mathds{N}$ and $w \in \mathds{N}$, if
\begin{enumerate}
    \item $\delta\in O(1)$
    \item $\forall \vb*{y} \in \mathds{N}^{\delta}$ such that $ \sum_{i=1}^{\delta} iy_i \leq \delta$ and $ \|\vb*{y}\|_0\geq 2:$
    $$\mathds{E}\big[\prod_{i=1}^{\delta} (f^{(i)}(U))^{y_i}\big] \in o(n)$$
    \item $\forall k \in [\delta]:$
    $$\mathds{E}[(f^{(k)}(U))^{\lfloor \frac{\delta}{k} \rfloor}] \in o(n)$$
    \item $\forall k \in [\delta]$ s.t. $\exists v \in V$ with $f_1^{(k)}(v) > 0:$
    $$\exists c>0 \textnormal{ s.t. }  \lim_{n \rightarrow \infty} \mathds{P}(f^{(k)}(U) \geq 1) \geq c$$
\end{enumerate}
then
\begin{align*}
 \sum_{\substack{\vb*{a} \in \mathds{N}^{\delta}:\\\sum_{i=1}^{\delta} ia_i = \delta}} \sum_{\alpha(\cdot) \in R(\vb*{a})} H( \vb*{a},\alpha(\cdot)) &= \delta!^w (n\mathds{E}[f^{(1)}(U) ])^{\delta} + \Big( -\frac{1}{2} \frac{\delta!^{1+w}}{(\delta-2)!}(n\mathds{E}[f^{(1)}(U)])^{\delta-2} n\mathds{E}[(f^{(1)}(U))^2] \nonumber\\
 &\hspace{0.5cm} + \frac{1}{2^{1+w}} \frac{\delta!^{1+w}}{(\delta-2)!} (n\mathds{E}[f^{(1)}(U)])^{\delta-2} n\mathds{E}[f^{(2)}(U)] \Big) (1+o(1))\nonumber \\
   &= \delta!^w (n\mathds{E}[f^{(1)}(U) ])^{\delta}(1+o(1)).
\end{align*}
\end{corollary}
\begin{proof}
The result follows from using $\delta_1=\delta$, $\delta_2=0$, $f_1^{(i)}(U)=f^{(i)}(U)$ and $f_2^{(i)}(U)=0$ in Lemma \ref{lemma:general_order}.
\end{proof}

\subsection{Undirected hypergraphs}

\subsubsection{Degenerate hyperedges}
\label{section:pf_degenerate}
In this section, we prove Theorem \ref{thm:E[deg-edges]}. The proof follows a general strategy that will reappear
throughout the paper: we first express the number of degenerate hyperedges in terms of the number of non-degenerate
ones, and then compute the latter by evaluating the probability that a given hyperedge contains no repeated
vertices. This probability can be written as a sum over all ordered $\delta_e$-tuples of vertices, which we
can analyze by applying Corollary~\ref{cor:main_cor}, which yields the desired formula for $\mathds{E}[DH_n]$.

\begin{proof}[Proof of Theorem \ref{thm:E[deg-edges]}]
    First,
\begin{align}
    \label{eq:E[deg_edges]}
    \mathds{E}[DH_n] &= |E| - \mathds{E}[\#\textnormal{non-degenerate hyperedges}].
\end{align}
We will now analyze the expected number of non-degenerate hyperedges. We have that
\begin{align}
\label{eq:E[non-deg_edges]}
    \mathds{E}[\#\textnormal{non-degenerate hyperedges}] &= \sum_{e \in E} \mathds{P}(e \textnormal{ is non-degenerate}).
\end{align}
Now,
\begin{align*}
    \mathds{P}(e \textnormal{ is non-degenerate}) &= \frac{1}{\delta_e!} \sideset{}{^*}\sum_{\vb*{v} \in V^{\delta_e}} \mathds{P}(e = \{v_1,v_2,\hdots,v_{\delta_e}\})\\
    &= \frac{1}{\delta_e!} \sideset{}{^*}\sum_{\vb*{v} \in V^{\delta_e}} \delta_e! \frac{d_{v_1}}{n\mathds{E}[d_U]} \frac{d_{v_2}}{n\mathds{E}[d_U]-1}\hdots \frac{d_{v_{\delta_e}}}{n\mathds{E}[d_U]-(\delta_e-1)}\\
    &= \frac{(n\mathds{E}[d_U]-\delta_e)!}{(n\mathds{E}[d_U])!} \sideset{}{^*}\sum_{\vb*{v} \in V^{\delta_e}} d_{v_1} d_{v_2} \hdots d_{v_{\delta_e}},
\end{align*}
where the $1/\delta_e!$ term ensures that every distinct combination of $\delta_e$ vertices is considered exactly once and the $\delta_e!$ term covers all possible orderings in which the vertices $v_1,v_2,\hdots,v_{\delta_2}$ can be matched to the hyperedge $e$. Although this equation sums over all distinct vertex lists, thus excluding lists with duplicate vertices, we may apply Corollary \ref{cor:main_cor} by introducing a function $f^i(v)$ which is 0 for $i \neq 1$. In particular, we apply Corollary \ref{cor:main_cor} with $w=0$, $\delta=\delta_e$, $f^i(v) = d_{v}$ if $i = 1$ and $0$ if $i\neq 1$. We obtain
\begin{align}
\label{eq:E[DH]_incl_b}
    &\mathds{P}(e \textnormal{ is non-degenerate})\nonumber\\
    &= \frac{(n\mathds{E}[d_U]-\delta_e)!}{(n\mathds{E}[d_U])!} \sum_{\substack{\vb*{a} \in \mathds{N}^{\delta_e}:\\ \sum_{i=1}^{\delta_e} ia_i = \delta_e}}  \frac{\delta_e!}{\prod_{k=1}^{\delta_e} k!^{a_k} }\nonumber\\
     &\hspace{1cm} \times \hspace{-0.25cm} \sum_{\alpha(\cdot) \in R(\vb*{a})} (-1)^{\sum_{\vb*{x} \in \mathds{N}^{\delta_e}} (\sum_{i=1}^{\delta_e} x_i - 1)\alpha(\vb*{x})}
    \prod_{\vb*{y} \in \mathds{N}^{\delta_e}}  \Bigg( \frac{1}{\alpha(\vb*{y})!} \Big(\frac{n\mathds{E}\big[\prod_{i=1}^{\delta_e} (f^{(i)}(v))^{y_i}\big]\big(\sum_{i=1}^{\delta_e} y_i-1 \big)!}{\prod_{i=1}^{\delta_e} y_i!} \Big)^{\alpha(\vb*{y})} \Bigg).
\end{align}
Note that any $\vb*{a}$ with $a_i > 0$ for some $i > 1$ results in values for $\alpha(\vb*{x})$ with $\alpha(\vb*{x})>0$ for some $\vb*{x}$ with $x_i>0$. The result of the products is then 0, since then $f^{(i)}(v) = 0$. Therefore, only $\vb*{a} = \delta_e \vb*{e}_1$ results in non-zero terms in the first sum. We denote $\mathcal{S} =  \{\vb*{0},\vb*{e}_1,\dots,\delta_e\vb*{e}_1\}$ and obtain
\begin{align*}
    &\mathds{P}(e \textnormal{ is non-degenerate})\\
    &= \frac{(n\mathds{E}[d_U]-\delta_e)!}{(n\mathds{E}[d_U])!}  \delta_e! \sum_{\substack{\alpha(\cdot) \in R(\delta_e \vb*{e}_1):\\
    \forall i >1:\sum_{\vb*{x} \in \mathds{N}^{\delta_e}} x_i \alpha(\vb*{x}) = 0 }} (-1)^{\sum_{\vb*{x} \in \mathds{N}^{\delta_e}} (x_1 - 1)\alpha(\vb*{x})}
    \prod_{\vb*{y} \in \mathcal{S}} \Bigg( \frac{1}{\alpha(\vb*{y})!} \Big(\frac{n\mathds{E}[(f^{(1)}(v))^{y_1}](y_1-1)!}{y_1!} \Big)^{\alpha(\vb*{y})} \Bigg)\\
    &= \frac{(n\mathds{E}[d_U]-\delta_e)!}{(n\mathds{E}[d_U])!} \delta_e! \sum_{\substack{\vb*{a} \in \mathds{N}^{\delta_e}:\\ \sum_{i=1}^{\delta_e} ia_i = \delta_e}} (-1)^{\sum_{i=1}^{\delta_e}(i-1)a_i} \prod_{i=1}^{\delta_e} \Bigg( \frac{1}{a_i!}\Big(\frac{n\mathds{E}[d_U^i]}{i}\Big)^{a_i} \Bigg).
\end{align*}
Combining with Equation (\ref{eq:E[non-deg_edges]}) gives
\begin{align*}
    &\mathds{E}[DH_n] = \sum_{e \in E} \frac{(n\mathds{E}[d_U]-\delta_e)!}{(n\mathds{E}[d_U])!} \delta_e! \sum_{\substack{\vb*{a} \in \mathds{N}^{\delta_e}:\\ \sum_{i=1}^{\delta_e} ia_i = \delta_e}} (-1)^{\sum_{i=1}^{\delta_e}(i-1)a_i} \prod_{i=1}^{\delta_e} \Bigg( \frac{1}{a_i!}\Big(\frac{n\mathds{E}[d_U^i]}{i}\Big)^{a_i} \Bigg),
\end{align*}
which yields the desired result when combined with Equation (\ref{eq:E[deg_edges]}).
\end{proof}

Now we prove Lemma \ref{lemma:E[deg-edges]_asymp}, by applying Corollary \ref{cor:general_order}.

\begin{proof}[Proof of Lemma \ref{lemma:E[deg-edges]_asymp}]
    We consider the expression for $\mathds{E}[DH]$ using Equation \eqref{eq:E[DH]_incl_b}, as it includes both the sum over $\vb*{a}$ and the sum over $\alpha(\cdot)$:
    \begin{align*}
        \mathds{E}[DH_n] &= |E| - |E|\frac{(n\mathds{E}[d_U]-\delta)!}{(n\mathds{E}[d_U])!} \sum_{\substack{\vb*{a} \in \mathds{N}^d:\\\sum_{i=1}^{\delta} a_i= \delta}} \sum_{\alpha(\cdot) \in R(\vb*{a})} H(\vb*{a}, \alpha(\cdot)),
    \end{align*}
    where $H(\cdot, \cdot)$ is as in \eqref{eq:H_hat_terms}, with $w=0$ and $f^{(i)}(v)=d_v$ if $i=1$ and $0$ else. All requirements for Corollary \ref{cor:general_order} are satisfied.  
Indeed, for all $\vb*{y} \in \mathds{N}^{\delta}$ with $\sum_{i=1}^{\delta} i y_i \leq \delta$ and $\sum_{i=1}^{\delta} \mathds{1}_{\{y_i \geq 1\}}$, we have
\begin{align*}
    \mathds{E} \! \left[\prod_{i=1}^{\delta} (f^{(i)}(U))^{y_i} \right]
    \leq \mathds{E}[
        d_U^{y_1}] \leq \mathds{E}[d_U^{\delta}] = o(n).
\end{align*}
Moreover, for all $k \in [\delta]$:
\begin{align*}
\mathds{E}\big[(f^{(k)}(U))^{\lfloor \frac{\delta}{k} \rfloor} \big] \leq \mathds{E}[d_U^{\delta}] = o(n).
\end{align*}
Lastly, only for $k=1$ holds $\exists v \in V: f^{(k)}(v) > 0$. Then, for $n$ large enough,
\begin{align*}
    \mathds{P}(f^{(1)}(U) \geq 1) = \mathds{P}(d_U \geq 1) \geq c.
\end{align*}

By Corollary \ref{cor:general_order}, 
    \begin{align*}
        \sum_{\substack{\vb*{a} \in \mathds{N}^{\delta}:\\\sum_{i=1}^{\delta} ia_i = \delta}} \sum_{\alpha(\cdot) \in R(\vb*{a})} H(\vb*{a},\alpha(\cdot))
 &= (n\mathds{E}[d_U])^{\delta} + \Big( -\frac{1}{2} \frac{\delta!}{(\delta-2)!}(n\mathds{E}[d_U])^{\delta-2} n\mathds{E}[d_U^2]  \Big) (1+o(1))
    \end{align*}
We obtain
    \begin{align*}
        \mathds{E}[DH_n] &= |E| - |E|\frac{(n\mathds{E}[d_U]-\delta)!}{(n\mathds{E}[d_U])!} \Big( (n\mathds{E}[d_U])^{\delta} + \Big( -\frac{\delta(\delta-1)}{2} (n\mathds{E}[d_U])^{\delta-2} n\mathds{E}[d_U^2]  \Big) (1+o(1))\Big).
    \end{align*}
    Since $\mathds{E}[d_U] \geq c > 0$, we obtain
    \begin{align*}
        \frac{(n\mathds{E}[d_U]-\delta)!}{(n\mathds{E}[d_U])!} &= \frac{1}{(n\mathds{E}[d_U])^{\delta}\Big(1 -  \frac{\delta-1}{n\mathds{E}[d_U]}+O\big( \frac{1}{(n\mathds{E}[d_U])^2}\big)\Big)} = \frac{1}{(n\mathds{E}[d_U])^{\delta}}\Big(1+\frac{\delta-1}{n\mathds{E}[d_U]}(1+o(1))\Big).
    \end{align*}
    Then,
    \begin{align*}
        \mathds{E}[DH_n] &= |E| - |E|\frac{1}{(n\mathds{E}[d_U])^{\delta}}\Big(1+\frac{\delta-1}{n\mathds{E}[d_U]}(1+o(1))\Big) \nonumber\\
        & \quad \times \Big((n \mathds{E}[d_U])^{\delta} - \frac{\delta(\delta-1)}{2} (n \mathds{E}[d_U])^{\delta-2} n \mathds{E}[d_U^2](1+o(1)) \Big)\\
        &= |E|\Big(-\frac{\delta-1}{n\mathds{E}[d_U]} + \frac{\delta(\delta-1)\mathds{E}[d_U^2]}{2n \mathds{E}[d_U]^2}\Big) (1+o(1))\\
        &= \Big(\frac{(\delta-1)\mathds{E}[d_U^2]}{2 \mathds{E}[d_U]}-\frac{\delta-1}{\delta}\Big)(1+o(1)) ,
    \end{align*}
    where we have used that $|E| = \frac{n \mathds{E}[d_U]}{\delta}$.
\end{proof}  

\subsubsection{Multi-hyperedge pairs}
\label{section:pf_multi}
We prove Theorem \ref{thm:E[multi-edges]}.

\begin{proof}[Proof of Theorem \ref{thm:E[multi-edges]}]
First,
\begin{align}
\label{eq:E[multi-edges]}
    \mathds{E}[M_n] &= \frac{1}{2} \sum_{e \in E} \sum_{\substack{e' \in E \backslash \{e\}:\\\delta_{e'}=\delta_e}} \mathds{P}(e=e').
\end{align}
Now, 
\begin{align*}
    \mathds{P}(e=e') = \sum_{\vb*{v} \in V^{\delta_e}} \frac{1}{D(\vb*{v})} \mathds{P}(e=e'= \{v_1,v_2,\hdots,v_{\delta_e}\}),
\end{align*}
where $D(\vb*{v})$ counts the number of ways that the vertex set $\{v_1,v_2,\hdots,v_{\delta_e}\}$ appears as tuple $\vb*{v} = (v_1,\hdots,v_{\delta_e})$. Let $v'_1,v'_2,\hdots,v'_{u(\vb*{v})}$ denote the unique vertices in $\vb*{v}$. Then, $D(\vb*{v}) = \frac{\delta_e!}{m_{v'_1}(\vb*{v})!m_{v'_2}(\vb*{v})!\hdots m_{v'_{u(\vb*{v})}}(\vb*{v})!}$. Now, if for some $i \in [u(\vb*{v})]$ holds $d_{v_i'} \leq 2m_{v'_i}(\vb*{v}) -1$ then $\mathds{P}(e=e'= \{v_1,v_2,\hdots,v_{\delta_e}\}) = 0$. Let $v \rightarrow e$ denote that vertex $v$ connects to hyperedge $e$. If $\forall i \in [u(\vb*{v})]$ holds $ d_{v_i'} \geq 2m_{v'_i}(\vb*{v})$ then
\begin{align*}
    \mathds{P}(e=e'= \{v_1,v_2,\hdots,v_{\delta_e}\}) &= (D(\vb*{v}))^2\prod_{i=1}^{\delta_e} \mathds{P}(v_i \rightarrow e \cap v_i \rightarrow e' | \forall j \in [i-1]: v_j \rightarrow e \cap v_j \rightarrow e')\\
    &= (D(\vb*{v}))^2 \prod_{i=1}^{\delta_e} \Big(\frac{d_{v_i} - 2\sum_{j=1}^{i-1} \mathds{1}_{\{v_j=v_i\}}}{n\mathds{E}[d_U] -2(i-1)} \cdot \frac{d_{v_i} - 2\sum_{j=1}^{i-1} \mathds{1}_{\{v_j=v_i\}}-1}{n\mathds{E}[d_U] -2(i-1)-1} \Big)\\
    &= \mathds{1}_{\{n\mathds{E}[d_U] \geq 2\delta_e\}}\frac{(n\mathds{E}[d_U]-2\delta_e)!}{(n\mathds{E}[d_U])!}(D(\vb*{v}))^2 \prod_{i=1}^{u(\vb*{v})}  \frac{d_{v'_i}!}{(d_{v'_i}-2m_{v'_i}(\vb*{v}))!}.
\end{align*}

Therefore,
\begin{align*}
    \mathds{P}(e=e') = \mathds{1}_{\{n\mathds{E}[d_U] \geq 2\delta_e\}} \frac{(n\mathds{E}[d_U]-2\delta_e)!}{(n\mathds{E}[d_U])!}\sum_{\vb*{v} \in V^{\delta_e}} D(\vb*{v}) \prod_{i=1}^{u(\vb*{v})} \Big( \mathds{1}_{\{d_{v'_i} \geq 2m_{v'_i}(\vb*{v})\}}\frac{d_{v'_i}!}{(d_{v'_i}-2m_{v'_i}(\vb*{v}))!} \Big).
\end{align*}
Now we apply Corollary \ref{cor:main_cor}, with $w=1$ and $f^i(v) = \mathds{1}_{\{d_v \geq 2i\}} \frac{d_v!}{(d_v-2i)!}$. We obtain
\begin{align*}
    &\mathds{P}(e=e')\\
    &= \mathds{1}_{\{n\mathds{E}[d_U] \geq 2\delta_e\}} \frac{(n\mathds{E}[d_U]-2\delta_e)!}{(n \mathds{E}[d_U])!}\sum_{\substack{\vb*{a} \in \mathds{N}^{\delta_e}:\\\sum_{i=1}^{\delta_e} ia_i = \delta_e}}  \frac{\delta_e!^2}{\prod_{k=1}^{\delta_e} k!^{2a_k} }\sum_{\alpha(\cdot) \in R(\vb*{a})} (-1)^{\sum_{\vb*{x} \in \mathds{N}^{\delta_e}} (\sum_{i=1}^{\delta_e} x_i - 1)\alpha(\vb*{x})}\\
     &\hspace{1cm}\times 
    \prod_{\vb*{y} \in \mathds{N}^{\delta_e}} \Bigg( \frac{1}{\alpha(\vb*{y})!} \Big(\frac{n\mathds{E} \big[\prod_{i=1}^{\delta_e} (\mathds{1}_{\{d_U \geq 2i\}}\frac{d_U!}{(d_U-2i)!})^{y_i}\big] \big(\sum_{i=1}^{\delta_e} y_i-1 \big)!}{\prod_{i=1}^{\delta_e} y_i!} \Big)^{\alpha(\vb*{y})} \Bigg),
\end{align*}
which yields the result when plugged into Equation (\ref{eq:E[multi-edges]}).
\end{proof}

We prove Lemma \ref{lemma:E[multi_edges]_asymp} using Corollary \ref{cor:general_order}.
\\
\begin{proof}[Proof of Lemma \ref{lemma:E[multi_edges]_asymp}]
We consider
\begin{align*}
    \mathds{E}[M_n]
    &= \frac{1}{2}\,|E|(|E|-1)\,
       \frac{(n\mathds{E}[d_U] - 2\delta)!}{(n\mathds{E}[d_U])!}
       \sum_{\substack{\vb*{a} \in \mathds{N}^{\delta} \\ \sum_{i=1}^{\delta} i a_i = \delta}}
       \ \sum_{\alpha(\cdot) \in R(\vb*{a})} H(\vb*{a}, \alpha(\cdot)),
\end{align*}
where $H(\cdot,\cdot)$ is as in \eqref{eq:Hterms}, with $w=1$ and

\[
    f^{(i)}(v) = \mathds{1}_{\{d_v \geq 2i\}}\frac{d_v!}{(d_v - 2i)!}.
\]

All requirements for Corollary \ref{cor:general_order} are satisfied.  
Indeed, for all $\vb*{y} \in \mathds{N}^{\delta}$ with $\sum_{i=1}^{\delta} i y_i \leq \delta$ and $\sum_{i=1}^{\delta} \mathds{1}_{\{y_i \geq 1\}}$, we have
\begin{align*}
    \mathds{E}\!\left[\prod_{i=1}^{\delta} (f^{(i)}(U))^{y_i}\right]
    &= \mathds{E}\!\left[
        \prod_{i=1}^{\delta}
        \left(\mathds{1}_{\{d_U \geq 2i\}}\frac{d_U!}{(d_U - 2i)!}\right)^{y_i}
      \right] \\
    &\le \mathds{E}\!\left[
        \prod_{i=1}^{\delta} d_U^{2iy_i}
      \right]
     = \mathds{E}\!\left[d_U^{2\sum_{i=1}^{\delta} iy_i}\right]
     \le \mathds{E}[d_U^{2\delta}]
     = o(n).
\end{align*}
Moreover, for all $k \in [\delta]$:
\begin{align*}
\mathds{E}\big[(f^{(k)}(U))^{\lfloor \frac{\delta}{k} \rfloor} \big] &= \mathds{E}\Big[\big(\mathds{1}_{\{d_U \geq 2k\}}\frac{d_U!}{(d_U - 2k)!} \big)^{\lfloor \frac{\delta}{k} \rfloor} \Big] \leq \mathds{E}\big[d_U^{2\lfloor \frac{\delta}{k} \rfloor} \big] \leq \mathds{E} \big[d_U^{2\delta} \big] = o(n).
\end{align*}
Lastly, for all $k \in [\delta]$:
\begin{align*}
    \mathds{P}(f^{(k)}(U) \geq 1) = \mathds{P}\Big(\mathds{1}_{\{d_U \geq 2k\}}\frac{d_U!}{(d_U - 2k)!} \geq 1 \Big) = \mathds{P}(d_U \geq 2k) \geq \mathds{P}(d_U \geq 2d) \geq c.
\end{align*}

By Corollary \ref{cor:general_order},
\begin{align*}
    \mathds{E}[M_n]
    &= \frac{1}{2}\,|E|(|E|-1)\,
       \frac{(n\mathds{E}[d_U] - 2\delta)!}{(n\mathds{E}[d_U])!}\,
       \delta!\,
       \bigl(n\,\mathds{E}[d_U(d_U - 1)]\bigr)^{\delta}
       (1+o(1)).
\end{align*}

Since $\mathds{E}[d_U] \geq \mathds{P}(d_U \geq 1) \geq \mathds{P}(d_U \geq 2\delta) \geq c$,
\begin{align*}
    \frac{(n\mathds{E}[d_U] - 2\delta)!}{(n\mathds{E}[d_U])!}
    = \frac{1}{
        (n\mathds{E}[d_U])^{2\delta}
        \left(1 - O\!\left(\frac{(2\delta)^2}{n\mathds{E}[d_U]}\right)\right)
       } = \frac{1}{
        (n\mathds{E}[d_U])^{2\delta}\,(1 - O(1/|E|))
       }.
\end{align*}

Thus,
\begin{align*}
    \mathds{E}[M_n]
    &= \frac{1}{2}\,|E|(|E|-1)\,
       \frac{1}{(n\mathds{E}[d_U])^{2\delta}(1 - O(1/|E|))}\,
       \delta!\,
       \bigl(n\,\mathds{E}[d_U(d_U - 1)]\bigr)^{\delta}
       (1+o(1)) \\
    &= \frac{(\delta-1)!}{2\delta}\,
       \frac{\mathds{E}[d_U(d_U - 1)]^{\delta}}{
         n^{\delta-2}\,\mathds{E}[d_U]^{2\delta-2}
       }
       (1+o(1)),
\end{align*}
where we used $|E| = \frac{n\mathds{E}[d_U]}{\delta}$.
\end{proof}

\subsection{Directed hypergraphs}

\subsubsection{Degenerate hyperedges}
\label{section:pf_dir_degenerate}

We prove Theorem \ref{thm:dir_E[deg-edges]} by using Corollary \ref{cor:main_cor} separately for the tail and head of a hyperedge.
\begin{proof}[Proof of Theorem \ref{thm:dir_E[deg-edges]}]
    First,
\begin{align}
\label{eq:dir_E[deg_edges]}
    \mathds{E}[DH_n] &= |E| - \sum_{e \in E^*} \mathds{P}(e \textnormal{ is non-degenerate}).
\end{align}
Now,
\begin{align*}
    &\mathds{P}(e \textnormal{ is non-degenerate}) \\
    &= \frac{1}{\delta_e^{\tail}!} \sideset{}{^*}\sum_{\vb*{v} \in V^{\delta_e^{\tail}}} \mathds{P}(e = \{v_1,v_2,\hdots,v_{\delta_e^{\tail}}\}) \frac{1}{\delta_e^{\head}!} \sideset{}{^*}\sum_{\vb*{w} \in V^{\delta_e^{\head}}} \mathds{P}(e = \{w_1,w_2,\hdots,w_{\delta_e^{\head}}\})\\
    &= \frac{1}{\delta_e^{\tail}!}\sideset{}{^*}\sum_{\vb*{v} \in V^{\delta_e^{\tail}}} \delta_e^{\tail}! \frac{d_{v_1}}{n\mathds{E}[d_v^{\tout}]} \frac{d_{v_2}}{n\mathds{E}[d_v^{\tout}]-1}\hdots \frac{d_{v_{\delta_e^{\tail}}}}{n\mathds{E}[d_v^{\tout}]-(\delta_e^{\tail}-1)}\\
    &\hspace{1cm} \times \frac{1}{\delta_e^{\head}!} \sideset{}{^*}\sum_{\vb*{w} \in V^{\delta_e^{\head}}} \delta_e^{\head}! \frac{d_{w_1}}{n\mathds{E}[d_v^{\tin}]} \frac{d_{w_2}}{n\mathds{E}[d_v^{\tin}]-1}\hdots \frac{d_{w_{\delta_e^{\head}}}}{n\mathds{E}[d_v^{\tin}]-(\delta_e^{\head}-1)}\\
    &= \frac{(n\mathds{E}[d_v^{\tout}]-\delta_e^{\tail})!(n\mathds{E}[d_v^{\tin}]-\delta_e^{\head})!}{(n\mathds{E}[d_v^{\tout}])!(n\mathds{E}[d_v^{\tin}])!} \sideset{}{^*}\sum_{\vb*{v} \in V^{\delta_e^{\tail}}} d_{v_1} d_{v_2} \hdots d_{v_{\delta_e^{\tail}}}  \sideset{}{^*}\sum_{\vb*{w} \in V^{\delta_e^{\head}}} d_{w_1} d_{w_2} \hdots d_{w_{\delta_e^{\head}}},
\end{align*}
where the $1/\delta_e^{\tail}!$ and $1/\delta_e^{\head}$ terms ensure that every distinct combination of $\delta_e^{\tail}$ tail vertices and every distinct combination of $\delta_e^{\head}$ head vertices is considered exactly once and the $\delta_e^{\tail}!$ and $\delta_e^{\head}!$ terms cover all possible orderings in which the vertices $v_1,v_2,\hdots,v_{\delta_e^{\tail}}$ and $w_1,w_2,\hdots,w_{\delta_e^{\head}}$ can be matched to hyperedge $e$. Although the sum over $\vb*{v}$ sums over all distinct vertex lists, thus excluding lists with duplicate vertices, we may apply Corollary \ref{cor:main_cor} by introducing a function $f^i(v)$ which is 0 for $i \neq 1$. In particular, we apply Corollary \ref{cor:main_cor} to the sum over $\vb*{v}$ with $w=0$, $f^i(v) = d_{v}^{\tout}$ if $i = 1$ and $0$ if $i\neq 1$. Similarly, we apply Corollary \ref{cor:main_cor} to the sum over $\vb*{w}$ by introducing a function $g^i(v)$ which is 0 for $i \neq 1$. In particular, with $w=0$ and $g^i(v) = d_v^{\tin}$ if $i=0$ and $0$ else. We obtain
\begin{align}
\label{eq:dir_E[DH]_incl_c,d}
    &\mathds{P}(e \textnormal{ is non-degenerate}) \nonumber\\
    &= \frac{(n\mathds{E} [d_v^{\tout}]-\delta_e^{\tail})!(n\mathds{E}[d_v^{\tin}]-\delta_e^{\head})!}{(n\mathds{E}[d_v^{\tout}])!(n\mathds{E}[d_v^{\tin}])!} \sum_{\substack{\vb*{a} \in \mathds{N}^{\delta_e^{\tail}}:\\ \sum_{i=1}^{\delta_e^{\tail}} i a_i = \delta_e^{\tail}}} \frac{\delta_e^{\tail}!}{\prod_{k=1}^{\delta_e^{\tail}} k!^{a_k} }\sum_{\alpha(\cdot) \in R(\vb*{a})} (-1)^{\sum_{\vb*{x} \in \mathds{N}^{\delta_e^{\tail}}} (\sum_{i=1}^{\delta_e^{\tail}} x_i - 1)\alpha(\vb*{x})} \nonumber\\
     &\quad \times
    \prod_{\vb*{y} \in \mathds{N}^{\delta_e^{\tail}}} \Bigg( \frac{1}{\alpha(\vb*{y})!} \Big(\frac{n\mathds{E}\big[\prod_{i=1}^{\delta_e^{\tail}} (f^{(i)}(v))^{y_i} \big]\big(\sum_{i=1}^{\delta_e^{\tail}} y_i-1 \big)!}{\prod_{i=1}^{\delta_e^{\tail}}y_i!} \Big)^{\alpha(\vb*{y})} \Bigg) \nonumber \\
    & \quad \times \sum_{\substack{\vb*{b} \in \mathds{N}^{\delta_e^{\head}}:\\ \sum_{i=1}^{\delta_e^{\head}} i b_i = \delta_e^{\head}}}  \frac{\delta_e^{\head}!}{\prod_{k=1}^{\delta_e^{\head}} k!^{b_k} }\sum_{\beta(\cdot) \in R(\vb*{b})} (-1)^{\sum_{\vb*{x} \in \mathds{N}^{\delta_e^{\head}}} (\sum_{i=1}^{\delta_e^{\head}} x_i - 1)\beta(\vb*{x})} \nonumber \\
     &\quad \times 
    \prod_{\vb*{z} \in \mathds{N}^{\delta_e^{\head}}} \Bigg( \frac{1}{\beta(\vb*{z})!} \Big(\frac{n\mathds{E}\big[\prod_{i=1}^{\delta_e^{\head}} (g^{(i)}(v))^{z_i} \big]\big(\sum_{i=1}^{\delta_e^{\head}} z_i-1 \big)!}{\prod_{i=1}^{\delta_e^{\head}} z_i!} \Big)^{\beta(\vb*{z})} \Bigg) .
\end{align}
Note that for any $\vb*{a}$ with $a_i > 0$ for some $i > 1$ we obtain $\forall \alpha(\cdot) \in R(\vb*{a}): \alpha(\vb*{y})>0$ for some $\vb*{y}$ with $y_i>0$ for $i>1$. The result of the products is then 0, since then $f^{(i)}(v) = 0$. Therefore, only $\vb*{a}=\delta_e^{\tail} \vb*{e}_1$ remains in the sum over $\vb*{a}$. Similarly for sum over $\vb*{b}$, only $\vb*{b}=\delta_e^{\head} \vb*{e}_1$ remains. We denote $\mathcal{S}^\head =  \{\vb*{0},\vb*{e}_1,\dots,\delta_e^{\head} \vb*{e}_1\}$, and $\mathcal{S}^\tail =  \{\vb*{0},\vb*{e}_1,\dots,\delta_e^{\tail} \vb*{e}_1\}$. We obtain
\begin{align*}
    &\mathds{P}(e \textnormal{ is non-degenerate})\\
    &= \frac{(n\mathds{E}[d_v^{\tout}]-\delta_e^{\tail})!(n\mathds{E}[d_v^{\tin}]-\delta_e^{\head})!}{(n\mathds{E}[d_v^{\tout}])!(n\mathds{E}[d_v^{\tin}])!}  \delta_e^{\tail}!\delta_e^{\head}!\\
     &\quad\times\sum_{\alpha(\cdot) \in R(\delta_e^{\tail}\vb*{e}_1)} (-1)^{\sum_{\vb*{x} \in \mathds{N}^{\delta_e^{\tail}}} (x_1 - 1)\alpha(\vb*{x})}
    \prod_{\vb*{y} \in\mathcal{S}^\tail} \Bigg( \frac{1}{\alpha(\vb*{y})!} \Big(\frac{n\mathds{E}[(f^{(1)}(v))^{y_1}](y_1-1)!}{y_1!} \Big)^{\alpha(\vb*{y})} \Bigg) \\
     &\quad \times\sum_{\beta(\cdot) \in R(\delta_e^{\head}\vb*{e}_1)} (-1)^{\sum_{\vb*{x} \in \mathds{N}^{\delta_e^{\head}}} (x_1 - 1)\beta(\vb*{x})}
    \prod_{\vb*{z}  \in\mathcal{S}^\head} \Bigg( \frac{1}{\beta(\vb*{z})!} \Big(\frac{n\mathds{E}[(g^{(1)}(v))^{z_1}](z_1-1)!}{z_1!} \Big)^{\beta(\vb*{z})} \Bigg)\\
    &= \frac{(n\mathds{E}[d_v^{\tout}]-\delta_e^{\tail})!(n\mathds{E}[d_v^{\tin}]-\delta_e^{\head})!}{(n\mathds{E}[d_v^{\tout}])!(n\mathds{E}[d_v^{\tin}])!}  \delta_e^{\tail}!\delta_e^{\head}!\\
     &\quad \times \sum_{\substack{\vb*{a} \in \mathds{N}^{\delta_e^{\tail}}:\\\sum_{i=1}^{\delta_e^{\tail}}ia_i = \delta_e^{\tail}}}  (-1)^{\sum_{i=1}^{\delta_e^{\tail}} (i-1)a_i} 
    \prod_{i=0}^{\delta_e^{\tail}} \Bigg( \frac{1}{a_i!} \Big(\frac{n\mathds{E}[(d_v^{\tout})^i]}{i} \Big)^{a_i} \Bigg) \\
     &\quad \times \sum_{\substack{\vb*{b} \in \mathds{N}^{\delta_e^{\tail}}:\\\sum_{i=1}^{\delta_e^{\tail}}ib_i = \delta_e^{\tail}}} (-1)^{\sum_{i=1}^{\delta_e^{\head}} (i-1)b_i} 
    \prod_{i=0}^{\delta_e^{\head}} \Bigg( \frac{1}{b_i!} \Big(\frac{n\mathds{E}[(d_v^{\tin} )^i]}{i} \Big)^{b_i} \Bigg).
\end{align*}
We obtain
\begin{align*}
    \mathds{E}[\#\textnormal{non-degenerate hyperedges}] &= \sum_{e \in E^*} \frac{(n\mathds{E}[d_v^{\tout}]-\delta_e^{\tail})!(n\mathds{E}[d_v^{\tin}]-\delta_e^{\head})!}{(n\mathds{E}[d_v^{\tout}])!(n\mathds{E}[d_v^{\tin}])!}  \delta_e^{\tail}!\delta_e^{\head}!\\
     &\quad \times \sum_{\substack{\vb*{a} \in \mathds{N}^{\delta_e^{\tail}}:\\\sum_{i=1}^{\delta_e^{\tail}}ia_i = \delta_e^{\tail}}}  (-1)^{\sum_{i=1}^{\delta_e^{\tail}} (i-1)a_i} 
    \prod_{i=0}^{\delta_e^{\tail}} \Bigg( \frac{1}{a_i!} \Big(\frac{n\mathds{E}[(d_v^{\tout})^i]}{i} \Big)^{a_i} \Bigg) \\
     &\quad\times \sum_{\substack{\vb*{b} \in \mathds{N}^{\delta_e^{\tail}}:\\\sum_{i=1}^{\delta_e^{\tail}}ib_i = \delta_e^{\tail}}} (-1)^{\sum_{i=1}^{\delta_e^{\head}} (i-1)b_i} 
    \prod_{i=0}^{\delta_e^{\head}} \Bigg( \frac{1}{b_i!} \Big(\frac{n\mathds{E}[(d_v^{\tin} )^i]}{i} \Big)^{b_i} \Bigg),
\end{align*}
which yields the result when combined with Equation (\ref{eq:dir_E[deg_edges]}).
\end{proof}

We prove Lemma \ref{lemma:dir_E[deg-edges]_asymp} by applying Corollary \ref{cor:general_order} separately to the sum over $\vb*{a}$ and the sum over $\vb*{b}$.

\begin{proof}[Proof of Lemma \ref{lemma:dir_E[deg-edges]_asymp}]
    We consider the expression for $\mathds{E}[DH_n]$ using Equation \eqref{eq:dir_E[DH]_incl_c,d}, as it includes both the sums over $\vb*{a}$ and $\vb*{b}$ and the sums over $\alpha(\cdot)$ and $\beta(\cdot)$:
    \begin{align*}
        \mathds{E}[DH_n] &= |E| - |E|\frac{(n\mathds{E}[d_v^{\tout}]-\delta^{\tail})!(n\mathds{E}[d_v^{\tin}]-\delta^{\head})!}{(n\mathds{E}[d_v^{\tout}])!(n\mathds{E}[d_v^{\tin}])!} \\
        &\sum_{\substack{\vb*{a} \in \mathds{N}^{\delta^{\tail}}:\\ \sum_{i=1}^{\delta^{\tail}} i a_i = \delta^{\tail}}} \sum_{\alpha(\cdot) \in R(\vb*{a})} H_1(\vb*{a}, \alpha(\cdot))
    \sum_{\substack{\vb*{b} \in \mathds{N}^{\delta^{\head}}:\\ \sum_{i=1}^{\delta^{\head}} i b_i = \delta^{\head}}}  \sum_{\beta(\cdot) \in R(\vb*{b})} H_2(\vb*{b},\beta(\cdot)),
    \end{align*}
    where $H_1(\cdot, \cdot)$ is as in \eqref{eq:Hterms}, with $w=0$ and $f^{(i)}(v)=d_v^{\tout}$ if $i=1$ and $0$ else, and $H_2(\cdot, \cdot)$ is as in \eqref{eq:Hterms}, with $w=0$ and $f^{(i)}(v)=d_v^{\tin}$ if $i=1$ and $0$ else. We treat the sums over $\vb*{a}, \alpha(\cdot)$ and over $\vb*{b}, \beta(\cdot)$ separately. 

    \textbf{Tails}. For the first sums, all requirements of Corollary \ref{cor:general_order} hold. For all  
\(\vb*{y} \in \mathds{N}^{\delta^{\tail}}\) with   \(\sum_{i=1}^{\delta^{\tail}} iy_i \le \delta^{\tail}\) and $\sum_{i=1}^{\delta^{\tail}} \mathds{1}_{\{y_i \geq 1\}} \geq 2$,
\begin{align*}
\mathds{E}\!\left[\prod_{i=1}^{\delta^{\tail}} (f^{(i)}(U))^{y_i}\right]
    &\leq \mathds{E}\!\left[
         (d_U^{\tout})^{y_1}
      \right] \leq \mathds{E}[(d_U^{\tout})^{\delta^{\tail}}]
     = o(n).
\end{align*}
Moreover, for all $k \in [\delta^{\tail}]$:
\begin{align*}
    \mathds{E}\Big[(f^{(k)}(U))^{\lfloor \frac{\delta^{\tail}}{k}\rfloor} \Big] \leq \mathds{E}\big[(d_U^{\tout})^{\delta^{\tail}}\big] = o(n).
\end{align*}
Lastly, only for $k=1$ holds $\exists v \in V: f^{(k)}(v) > 0$. Then,
\begin{align*}
    \mathds{P}(f^{(1)}(U) \geq 1) = \mathds{P}(d_U^{\tout} \geq 1) \geq c.
\end{align*}   

\textbf{Heads}. The same reasoning applies to the sums over $\vb*{b}, \beta(\cdot)$. For all  
\(\vb*{y} \in \mathds{N}^{\delta^{\head}}\) with   \(\sum_{i=1}^{\delta^{\head}} iy_i \le \delta^{\head}\) and $\sum_{i=1}^{\delta^{\head}} \mathds{1}_{\{y_i \geq 1\}} \geq 2$,
\begin{align*}
\mathds{E}\!\left[\prod_{i=1}^{\delta^{\head}} (f^{(i)}(U))^{y_i}\right]
    &\leq \mathds{E}\!\left[
         (d_U^{\tin})^{y_1}
      \right] \leq \mathds{E}[(d_U^{\tin})^{\delta^{\head}}]
     = o(n).
\end{align*}
Moreover, for all $k \in [\delta^{\head}]$:
\begin{align*}
    \mathds{E}\Big[(f^{(k)}(U))^{\lfloor \frac{\delta^{\head}}{k}\rfloor} \Big] \leq \mathds{E}\big[(d_U^{\tin})^{\delta^{\head}}\big] = o(n).
\end{align*}
Lastly, only for $k=1$ holds $\exists v \in V: f^{(k)}(v) > 0$. Then,
\begin{align*}
    \mathds{P}(f^{(1)}(U) \geq 1) = \mathds{P}(d_U^{\tin} \geq 1) \geq c.
\end{align*}   

By Corollary \ref{cor:general_order},
\begin{align*}
    &\sum_{\substack{\vb*{a} \in \mathds{N}^{\delta^{\tail}}:\\ \sum_{i=1}^{\delta^{\tail}} i a_i = \delta^{\tail}}} \sum_{\alpha(\cdot) \in R(\vb*{a})} H_1(\vb*{a}, \alpha(\cdot))
    \sum_{\substack{\vb*{b} \in \mathds{N}^{\delta^{\head}}:\\ \sum_{i=1}^{\delta^{\head}} i b_i = \delta^{\head}}}  \sum_{\beta(\cdot) \in B(\vb*{b})} H_2(\vb*{b}, \beta(\cdot)) \\
    &\hspace{0.5cm} = \Big( (n \mathds{E}[d_U^{\tout}])^{\delta^{\tail}} - \frac{1}{2} d_U^{\tout}(d_U^{\tout}-1) (n \mathds{E}[d_U^{\tout}])^{\delta^{\tail}-2} n \mathds{E}[(d_U^{\tout})^2] (1+o(1)) \Big)\\
    &\quad\times \Big( (n \mathds{E}[d_U^{\tin}])^{\delta^{\head}} - \frac{1}{2} d_U^{\tin}(d_U^{\tin}-1) (n \mathds{E}[d_U^{\tin}])^{\delta^{\head}-2} n \mathds{E}[(d_U^{\tin})^2] (1+o(1)) \Big).
\end{align*}

Furthermore, since $\mathds{E}[d_U^{\tout}],\mathds{E}[d_U^{\tin}] \geq c > 0$, we obtain
    \begin{align*}
        & \frac{(n\mathds{E}[d_v^{\tout}]-\delta^{\tail})!(n\mathds{E}[d_v^{\tin}]-\delta^{\head})!}{(n\mathds{E}[d_v^{\tout}])!(n\mathds{E}[d_v^{\tin}])!} \nonumber\\
        &= \frac{1}{(n\mathds{E}[d_U^{\tout}])^{\delta^{\tail}}\Big(1 - \frac{\delta^{\tail}-1}{n\mathds{E}[d_U^{\tout}]}(1+o(1))\Big)(n\mathds{E}[d_U^{\tin}])^{\delta^{\head}}\Big(1 -  \frac{\delta^{\head}-1}{n\mathds{E}[d_U^{\tin}]}(1+o(1))\Big)} \\
        &= \frac{1}{(n\mathds{E}[d_U^{\tout}])^{\delta^{\tail}}(n\mathds{E}[d_U^{\tin}])^{\delta^{\head}}}\Big(1+\frac{\delta^{\head}-1}{n\mathds{E}[d_U^{\tin}]}(1+o(1))+\frac{\delta^{\tail}-1}{n\mathds{E}[d_U^{\tout}]}(1+o(1))\Big).
    \end{align*}
    Then,
    \begin{align*}
        \mathds{E}[DH_n] &= |E| - |E|\frac{1}{(n\mathds{E}[d_U^{\tout}])^{\delta^{\tail}}(n\mathds{E}[d_U^{\tin}])^{\delta^{\head}}}\Big(1+\frac{\delta^{\head}-1}{n\mathds{E}[d_U^{\tin}]}(1+o(1))+\frac{\delta^{\tail}-1}{n\mathds{E}[d_U^{\tout}]}(1+o(1))\Big)\\
        &\quad\times\Big((n \mathds{E}[d_U^{\tout}])^{\delta^{\tail}} - \frac{\delta^{\tail}(\delta^{\tail}-1)}{2} (n \mathds{E}[d_U^{\tout}])^{\delta^{\tail}-2} n \mathds{E}[(d_U^{\tout})^2](1+o(1)) \Big)\\
        &\quad\times\Big((n \mathds{E}[d_U^{\tin}])^{\delta^{\head}} - \frac{\delta^{\head}(\delta^{\head}-1)}{2} (n \mathds{E}[d_U^{\tin}])^{\delta^{\head}-2} n \mathds{E}[(d_U^{\tin})^2](1+o(1)) \Big)\\
        &= |E|\Big(\frac{\delta^{\tail}(\delta^{\tail}-1)\mathds{E}[(d_U^{\tout})^2]}{2n \mathds{E}[d_U^{\tout}]^2}+ \frac{\delta^{\head}(\delta^{\head}-1)\mathds{E}[(d_U^{\tin})^2]}{2n \mathds{E}[d_U^{\tin}]^2} -\frac{\delta^{\head}-1}{n\mathds{E}[d_U^{\tin}]}-\frac{\delta^{\tail}-1}{n\mathds{E}[d_U^{\tout}]}\Big)(1+o(1)) \\
        &=\Big(\frac{(\delta^{\tail}-1)\mathds{E}[(d_U^{\tout})^2]}{2 \mathds{E}[d_U^{\tout}]}+ \frac{(\delta^{\head}-1)\mathds{E}[(d_U^{\tin})^2]}{2 \mathds{E}[d_U^{\tin}]} -\frac{\delta^{\head}-1}{\delta^{\head}}-\frac{\delta^{\tail}-1}{\delta^{\tail}}\Big)(1+o(1)),
    \end{align*}
    where we used $|E| = \frac{n \mathds{E}[d_U^{\tout}]}{\delta^{\tail}} = \frac{n \mathds{E}[d_U^{\tin}]}{\delta^{\head}}$, $\delta^{\head} \in O(1)$, $\mathds{E}[(d_U^{\tin})^2] \leq \mathds{E}[(d_U^{\tin})^{\delta^{\head}}] = o(n)$ and $\mathds{E}[d_U^{\tin}] \geq c$.
\end{proof}  

\subsubsection{Multi-hyperedge pairs}
\label{section:pf_dir_multi-edges}

We prove Theorem \ref{thm:dir_E[multi-edges]} by applying Corollary \ref{cor:main_cor} separately to the tail and head of a hyperedge. 

\begin{proof}[Proof of Theorem \ref{thm:dir_E[multi-edges]}]
First,
\begin{align}
\label{eq:dir_E[multi-edges]}
    \mathds{E}[M_n] &= \frac{1}{2} \sum_{e \in E} \sum_{\substack{e' \in E:\\\delta_{e'}=\delta_e}} \mathds{P}(e=e').
\end{align}
Now,
\begin{align*}
    &\mathds{P}(e=e') \\
    &= \sum_{\vb*{v} \in V^{\delta_e^{\tail}}} \sum_{\vb*{w} \in V^{\delta_e^{\head}}} \frac{1}{D(\vb*{v})D(\vb*{w})} \mathds{P}(e^{\tail} = \{v_i\}_{i=1}^{\delta_e^{\tail}} , e^{\head} = \{w_i\}_{i=1}^{\delta_e^{\head}} ) \nonumber\\
    & \quad \times\mathds{P}(e'^{\tail} =\{v_i\}_{i=1}^{\delta_e^{\tail}}, e'^{\head} =\{w_i\}_{i=1}^{\delta_e^{\head}} | e^{\tail} =\{v_i\}_{i=1}^{\delta_e^{\tail}} , e^{\head} = \{w_i\}_{i=1}^{\delta_e^{\head}}  ) \\
    &= \frac{(n\mathds{E}[d_U^{\tout}]-2\delta_e^{\tail})!(n\mathds{E}[d_U^{\tin}]-2\delta_e^{\head})!}{(n\mathds{E}[d_U^{\tout}])!(n\mathds{E}[d_U^{\tin}])!} \sum_{\vb*{v} \in V^{\delta_e^{\tail}}} D(\vb*{v})  \prod_{i=1}^{k(\vb*{v})} f^{(m_i(\vb*{v}))} (v'_i) \sum_{\vb*{w} \in V^{\delta_e^{\head}}} D(\vb*{w}) \prod_{i=1}^{k(\vb*{w})} g^{(m_i(\vb*{w}))} (w'_i),
\end{align*}
where $f^i(v) = d_v^{\tout}(d_v^{\tout}-1)\hdots(d_v^{\tout}-(2i-1)) = \mathds{1}_{\{d_v^{\tout} \geq 2i\}}\frac{d_v^{\tout}!}{(d_v^{\tout}-2i)!}$ and $g^i(v) = d_v^{\tin}(d_v^{\tin}-1)\hdots(d_v^{\tin}-(2i-1)) = \mathds{1}_{\{d_v^{\tin} \geq 2i\}}\frac{d_v^{\tin}!}{(d_v^{\tin}-2i)!}$. For both sums, we apply Corollary \ref{cor:main_cor} with $w=1$. We obtain
\begin{align*}
    &\mathds{P}(e=e')\\
    &= \frac{(n\mathds{E}[d_U^{\tout}]-2\delta_e^{\tail})!(n\mathds{E}[d_U^{\tin}]-2\delta_e^{\head})!}{(n\mathds{E}[d_U^{\tout}])!(n\mathds{E}[d_U^{\tin}])!} \sum_{\substack{\vb*{a} \in \mathds{N}^{\delta_e^{\tail}}:\\\sum_{i=1}^{\delta_e^{\tail}} ia_i= \delta_e^{\tail}}}  \frac{\delta_e^{\tail}!^2}{\prod_{k=1}^{\delta_e^{\tail}} k!^{2a_k} }\sum_{\alpha(\cdot) \in R(\vb*{a})} (-1)^{\sum_{\vb*{x} \in \mathds{N}^{\delta_e^{\tail}}} (\sum_{i=1}^{\delta_e^{\tail}} x_i - 1)\alpha(\vb*{x})}\\
     &\hspace{1cm}\times
    \prod_{\vb*{y} \in \mathds{N}^{\delta_e^{\tail}}} \Bigg( \frac{1}{\alpha(\vb*{y})!} \Big(\frac{n\mathds{E}\big[\prod_{i=1}^{\delta_e^{\tail}} (\mathds{1}_{\{d_U^{\tout} \geq 2i\}}\frac{d_U^{\tout}!}{(d_U^{\tout}-2i)!})^{y_i} \big]\big(\sum_{i=1}^{\delta_e^{\tail}} y_i-1 \big)!}{\prod_{i=1}^{\delta_e^{\head}}y_i!} \Big)^{\alpha(\vb*{y})} \Bigg)\\
    &\hspace{1cm}\times \sum_{\substack{\vb*{b} \in \mathds{N}^{\delta_e^{\head}}:\\\sum_{i=1}^{\delta_e^{\head}} ib_i= \delta_e^{\tail}}}  \frac{\delta_e^{\head}!^2}{\prod_{k=1}^{\delta_e^{\head}} k!^{2b_k} }\sum_{\beta(\cdot) \in R(\vb*{b})} (-1)^{\sum_{\vb*{x} \in \mathds{N}^{\delta_e^{\head}}} (\sum_{i=1}^{\delta_e^{\head}} x_i - 1)\beta(\vb*{x})}\\
     &\hspace{1cm}
    \times \prod_{\vb*{z} \in \mathds{N}^{\delta_e^{\head}}} \Bigg( \frac{1}{\beta(\vb*{z})!} \Big(\frac{n\mathds{E}\big[\prod_{i=1}^{\delta_e^{\head}} (\mathds{1}_{\{d_U^{\tin} \geq 2i\}}\frac{d_U^{\tin}!}{(d_U^{\tin}-2i)!})^{z_i}\big]\big(\sum_{i=1}^{\delta_e^{\head}} z_i-1 \big)!}{\prod_{i=1}^{\delta_e^{\tail}} z_i!} \Big)^{\beta(\vb*{z})} \Bigg),
\end{align*}
which yields the result when plugged into Equation (\ref{eq:dir_E[multi-edges]}).
\end{proof}

We prove Lemma \ref{lemma:dir_E[multi_edges]_asymp} by applying Corollary \ref{cor:general_order} separately to the sum over $\vb*{a}$ and the sum over $\vb*{b}$.

\begin{proof}[Proof of Lemma \ref{lemma:dir_E[multi_edges]_asymp}]
We consider
\begin{align*}
    \mathds{E}[M_n]
    &= \frac{1}{2}\,|E|(|E|-1)\,
       \frac{(n\mathds{E}[d_U^{\tout}] - 2\delta^{\tail})!\,(n\mathds{E}[d_U^{\tin}] - 2\delta^{\head})!}
            {(n\mathds{E}[d_U^{\tout}])!\,(n\mathds{E}[d_U^{\tin}])!} \\
    &\qquad \times
       \sum_{\substack{\vb*{a} \in \mathds{N}^{\delta_e^{\tail}} \\ \sum_{i=1}^{\delta_e^{\tail}} i a_i = \delta_e^{\tail}}}
       \ \sum_{\alpha(\cdot) \in R(\vb*{a})} H_1(\vb*{a}, \alpha(\cdot))
       \sum_{\substack{\vb*{b} \in \mathds{N}^{\delta_e^{\head}} \\ \sum_{i=1}^{\delta_e^{\head}} i b_i = \delta_e^{\head}}}
       \ \sum_{\beta(\cdot) \in R(\vb*{b})} H_2(\vb*{b}, \beta(\cdot)).
\end{align*}

Here $H_1(\cdot,\cdot)$ is as in \eqref{eq:Hterms} with  
\( w = 1 \) and \( f^{(i)}(v) = \mathds{1}_{\{d_v^{\tout} \geq 2i\}}\frac{d_v^{\tout}!}{(d_v^{\tout}-2i)!} \),  
and $H_2(\cdot,\cdot)$ is as in \eqref{eq:Hterms} with  
\( w = 1 \) and \( f^{(i)}(v) = \mathds{1}_{\{d_v^{\tin} \geq 2i\}}\frac{d_v^{\tin}!}{(d_v^{\tin}-2i)!} \geq 1\).

We treat the sums over \(\vb*{a}, \alpha(\cdot)\) and over \(\vb*{b}, \beta(\cdot)\) separately.

\textbf{Tails.}
For the first sums, all requirements of Corollary \ref{cor:general_order} hold. For all  
\(\vb*{y} \in \mathds{N}^{\delta^{\tail}}\) with   \(\sum_{i=1}^{\delta^{\tail}} iy_i \le \delta^{\tail}\) and $\sum_{i=1}^{\delta^{\tail}} \mathds{1}_{\{y_i \geq 1\}} \geq 2$,
\begin{align*}
\mathds{E}\!\left[\prod_{i=1}^{\delta^{\tail}} (f^{(i)}(U))^{y_i}\right]
    &= \mathds{E}\!\left[
        \prod_{i=1}^{\delta^{\tail}}
        \left( \mathds{1}_{\{d_U^{\tout} \geq 2i\}}\frac{d_U^{\tout}!}{(d_U^{\tout}-2i)!}\right)^{y_i}
      \right] \le \mathds{E}\!\left[
        (d_U^{\tout})^{2\sum_{i=1}^{\delta^{\tail}} iy_i}
      \right]
     \le \mathds{E}\big[(d_U^{\tout})^{2\delta^{\tail}} \big]
     = o(n).
\end{align*}
Moreover, for all $k \in [\delta^{\tail}]$:
\begin{align*}
    \mathds{E}\big[(f^{(k)}(U))^{\lfloor \frac{\delta^{\tail}}{k}\rfloor}\big] = \mathds{E}\Big[\big(\mathds{1}_{\{d_U^{\tout} \geq 2k\}}\frac{d_U^{\tout}!}{(d_U^{\tout}-2k)!}\big)^{\lfloor \frac{\delta^{\tail}}{k}\rfloor}\Big] \leq \mathds{E}\big[(d_U^{\tout})^{2k \lfloor \frac{\delta^{\tail}}{k} \rfloor} \big] \leq \mathds{E}\big[(d_U^{\tout})^{2\delta^{\tail}} \big] = o(n).
\end{align*}
Lastly, for all $k \in [\delta^{\tail}]:$
\begin{align*}
    \mathds{P}(f^{(k)}(U) \geq 1) = \mathds{P}\Big(\mathds{1}_{\{d_U^{\tout} \geq 2k\}}\frac{d_U^{\tout}!}{(d_U^{\tout}-2k)!} \geq 1 \Big) = \mathds{P}(d_U^{\tout} \geq 2k) \geq \mathds{P}(d_U^{\tout} \geq 2\delta^{\tail}) \geq c.
\end{align*}

\textbf{Heads.} The same reasoning applies to the sums over \(\vb*{b}, \beta(\cdot)\). For all  
\(\vb*{y} \in \mathds{N}^{\delta^{\head}}\) with   \(\sum_{i=1}^{\delta^{\head}} iy_i \le \delta^{\head}\) and $\sum_{i=1}^{\delta^{\head}} \mathds{1}_{\{y_i \geq 1\}} \geq 2$,
\begin{align*}
\mathds{E}\!\left[\prod_{i=1}^{\delta^{\head}} (f^{(i)}(U))^{y_i}\right]
    &= \mathds{E}\!\left[
        \prod_{i=1}^{\delta^{\head}}
        \left( \mathds{1}_{\{d_U^{\tin} \geq 2i\}}\frac{d_U^{\tin}!}{(d_U^{\tin}-2i)!}\right)^{y_i}
      \right] \le \mathds{E}\!\left[
        (d_U^{\tin})^{2\sum_{i=1}^{\delta^{\head}} iy_i}
      \right]
     \le \mathds{E}[(d_U^{\tin})^{2\delta^{\head}}]
     = o(n).
\end{align*}
Moreover, for all $k \in [\delta^{\head}]$:
\begin{align*}
    \mathds{E}\big[(f^{(k)}(U))^{\lfloor \frac{\delta^{\head}}{k}\rfloor} \big]= \mathds{E}\Big[\big(\mathds{1}_{\{d_U^{\tin} \geq 2k\}}\frac{d_U^{\tin}!}{(d_U^{\tin}-2k)!}\big)^{\lfloor \frac{\delta^{\head}}{k}\rfloor} \Big] \leq \mathds{E}\big[(d_U^{\tin})^{2k \lfloor \frac{\delta^{\head}}{k} \rfloor}\big] \leq \mathds{E}\big[(d_U^{\tin})^{2\delta^{\head}}\big] = o(n).
\end{align*}
Lastly, for all $k \in [\delta^{\head}]:$
\begin{align*}
    \mathds{P}(f^{(k)}(U) \geq 1) = \mathds{P}\Big(\mathds{1}_{\{d_U^{\tin} \geq 2k\}}\frac{d_U^{\tin}!}{(d_U^{\tin}-2k)!} \geq 1 \Big) = \mathds{P}(d_U^{\tin} \geq 2k) \geq \mathds{P}(d_U^{\tin} \geq 2\delta^{\head}) \geq c.
\end{align*}

By Corollary \ref{cor:general_order},
\begin{align*}
    \mathds{E}[M_n]
    &= \frac{1}{2}\,|E|(|E|-1)\,
       \frac{(n\mathds{E}[d_U^{\tout}] - 2\delta^{\tail})!\,(n\mathds{E}[d_U^{\tin}] - 2\delta^{\head})!}
            {(n\mathds{E}[d_U^{\tout}])!\,(n\mathds{E}[d_U^{\tin}])!} \\
    &\qquad \times
       \delta^{\tail}!\,(n\mathds{E}[d_U^{\tout}(d_U^{\tout}-1)])^{\delta^{\tail}}
       \delta^{\head}!\,(n\mathds{E}[d_U^{\tin}(d_U^{\tin}-1)])^{\delta^{\head}}
       (1+o(1)).
\end{align*}

Also,
\begin{align*}
    \frac{(n\mathds{E}[d_U^{\tout}] - 2\delta^{\tail})!\,(n\mathds{E}[d_U^{\tin}] - 2\delta^{\head})!}
         {(n\mathds{E}[d_U^{\tout}])!\,(n\mathds{E}[d_U^{\tin}])!}
    &= \frac{1}{
        (n\mathds{E}[d_U^{\tout}])^{2\delta^{\tail}}
        (n\mathds{E}[d_U^{\tin}])^{2\delta^{\head}}
        (1 - O(1/|E|))
       }.
\end{align*}

Thus,
\begin{align*}
    \mathds{E}[M_n]
    &= \frac{1}{2}\,|E|(|E|-1)\,
       \frac{
         \delta^{\tail}!\,(n\mathds{E}[d_U^{\tout}(d_U^{\tout}-1)])^{\delta^{\tail}}
         \delta^{\head}!\,(n\mathds{E}[d_U^{\tin}(d_U^{\tin}-1)])^{\delta^{\head}}
       }{
         (n\mathds{E}[d_U^{\tout}])^{2\delta^{\tail}}
         (n\mathds{E}[d_U^{\tin}])^{2\delta^{\head}}
         (1 - O(1/|E|))
       }
       (1+o(1)) \\
    &= \frac{
         (\delta^{\tail}-1)!\,\mathds{E}[d_U^{\tout}(d_U^{\tout}-1)]^{\delta^{\tail}}
         (\delta^{\head}-1)!\,\mathds{E}[d_U^{\tin}(d_U^{\tin}-1)]^{\delta^{\head}}
       }{
         2\,n^{\delta^{\tail}+\delta^{\head}-2}\,
         \mathds{E}[d_U^{\tout}]^{2\delta^{\tail}-1}\,
         \mathds{E}[d_U^{\tin}]^{2\delta^{\head}-1}
       }
       (1+o(1)),
\end{align*}
where we used

\[
    |E| = \frac{n\mathds{E}[d_U^{\tout}]}{\delta^{\tail}}
        = \frac{n\mathds{E}[d_U^{\tin}]}{\delta^{\head}}.
\]

\end{proof}

\subsubsection{Self-loops}
\label{section:pf_sl}
We prove Theorem \ref{thm:E[sl]} by applying Corollary \ref{cor:main_cor}.

\begin{proof}[Proof of Theorem \ref{thm:E[sl]}]
The number of self-loops is computed by summing over all hyperedges that can be such self-loops (hyperedges with $\delta_e^{\head} = \delta_e^{\tail}$, which are elements of $E^* = \{e \in E: \delta_e^{\head} = \delta_e^{\tail} \}$). Now,
\begin{align}
\label{eq:E[self-loops]}
    \mathds{E}[S_n] &= \sum_{e \in E^*} \mathds{P}(e^{\tail}=e^{\head}).
\end{align}
To simplify notation, let $\delta_e = \delta_e^{\tail} = \delta_e^{\head}$. Now,
\begin{align*}
    \mathds{P}(e^{\tail}=e^{\head}) = \sum_{\vb*{v} \in V^{\delta_e}}  \frac{1}{D(\vb*{v})} \mathds{P}(e^{\tail} = \{v_1,v_2,\hdots,v_{\delta_e}\}) \mathds{P}(e^{\head} = \{v_1,v_2,\hdots,v_{\delta_e}\}),
\end{align*}
where $D(\vb*{v})$ counts the number of ways that the vertex set $\{v_1,\hdots,v_{\delta_e}\}$ appears as tuple $\vb*{v} = (v_1,\hdots,v_{\delta_e})$. Let $v'_1,v'_2,\hdots,v'_{u(\vb*{v})}$ denote the unique vertices in $\vb*{v}$. Then, $D(\vb*{v}) = \frac{\delta_e!}{m_{v'_1}(\vb*{v})!m_{v'_2}(\vb*{v})!\hdots m_{v'_{u(\vb*{v})}}(\vb*{v})!}$. Now, if for some $i \in [u(\vb*{v})]$ holds $d_{v_i'}^{\tout} \leq m_{v'_i}(\vb*{v}) - 1$ or $d_{v'_i}^{\tin} \leq m_{v'_i}(\vb*{v}) - 1$ then $\mathds{P}(e^{\tail} = \{v_1,v_2,\hdots,v_{\delta_e}\}) \mathds{P}(e^{\head} = \{v_1,v_2,\hdots,v_{\delta_e}\}) = 0$. If $\forall i \in [u(\vb*{v})]$ holds $d_{v'_i}^{\tout}, d_{v'_i}^{\tin} \geq m_{v'_i}(\vb*{v})$ then
\begin{align*}
    \mathds{P}(e^{\tail} = \{v_1,v_2,\hdots,v_{\delta_e}\}) &= D(\vb*{v}) \prod_{i=1}^{\delta_e} \frac{d_{v_i}^{\tout} - \sum_{j=1}^{i-1}\mathds{1}_{\{v_j=v_i\}}}{n\mathds{E}[d_U^{\tout}]-(i-1)}
\end{align*}
and
\begin{align*}
    \mathds{P}(e^{\head} = \{v_1,v_2,\hdots,v_{\delta_e}\} | e^{\tail} = \{v_1,v_2,\hdots,v_{\delta_e}\}) &= D(\vb*{v}) \prod_{i=1}^{\delta_e} \frac{d_{v_i}^{\tin} - \sum_{j=1}^{i-1}\mathds{1}_{\{v_j=v_i\}}}{n\mathds{E}[d_U^{\tin}]-(i-1)}.
\end{align*}
This gives
\begin{align*}
    &\mathds{P}(e^{\tail} = \{v_1,v_2,\hdots,v_{\delta_e}\}) \mathds{P}(e^{\head} = \{v_1,v_2,\hdots,v_{\delta_e}\} | e^{\tail} = \{v_1,v_2,\hdots,v_{\delta_e}\}) \\
    &= (D(\vb*{v}))^2 \prod_{i=1}^{\delta_e} \frac{d_{v_i}^{\tout} - \sum_{j=1}^{i-1}\mathds{1}_{\{v_j=v_i\}}}{n\mathds{E}[d_U^{\tout}]-(i-1)} \cdot \frac{d_{v_i}^{\tin} - \sum_{j=1}^{i-1}\mathds{1}_{\{v_j=v_i\}}}{n\mathds{E}[d_U^{\tin}]-(i-1)} \\
    &= \frac{(n\mathds{E}[d_U^{\tout}]-\delta_e)!(n\mathds{E}[d_U^{\tin}]-\delta_e)!}{(n\mathds{E}[d_U^{\tout}])!(n\mathds{E}[d_U^{\tin}])!}(D(\vb*{v}))^2 \prod_{i=1}^{u(\vb*{v})}  \mathds{1}_{\{d_{v'_i}^{\tout},d_{v'_i}^{\tin} \geq m_{v'_i}(\vb*{v})\}}\frac{d_{v'_i}^{\tout}!d_{v'_i}^{\tin}!}{(d_{v'_i}^{\tout}-m_{v'_i}(\vb*{v}))!(d_{v'_i}^{\tin}-m_{v'_i}(\vb*{v}))!}
\end{align*}
and so
\begin{align*}
    &\mathds{P}(e^{\tail}=e^{\head})\\
    &= \frac{(n\mathds{E}[d_U^{\tout}]-\delta_e)!(n\mathds{E}[d_U^{\tin}]-\delta_e)!}{(n\mathds{E}[d_U^{\tout}])!(n\mathds{E}[d_U^{\tin}])!} \sum_{\vb*{v} \in V^{\delta_e}} D(\vb*{v}) \prod_{i=1}^{u(\vb*{v})} \mathds{1}_{\{d_{v'_i}^{\tout},d_{v'_i}^{\tin} \geq m_{v'_i}(\vb*{v})\}} \frac{d_{v'_i}^{\tout}!d_{v'_i}^{\tin}!}{(d_{v'_i}^{\tout}-m_{v'_i}(\vb*{v}))!(d_{v'_i}^{\tin}-m_{v'_i}(\vb*{v}))!}.
\end{align*}
Now we apply Corollary \ref{cor:main_cor} with $w=1$ and $f^i(v) = \mathds{1}_{\{d_v^{\tout},d_v^{\tin} \geq i\}}\frac{d_v^{\tout}!}{(d_v^{\tout}-i)!}\frac{d_v^{\tin}!}{(d_v^{\tin}-i)!}$. We obtain
\begin{align*}
    &\mathds{P}(e^{\tail}=e^{\head})\\
    &= \frac{(n\mathds{E}[d_U^{\tin}]-\delta_e)!(n\mathds{E}[d_U^{\tout}]-\delta_e)!}{(n\mathds{E}[d_U^{\tin}])!(n\mathds{E}[d_U^{\tout}])!} \sum_{\substack{\vb*{a} \in \mathds{N}^{\delta_e}:\\ \sum_{i=1}^{\delta_e} ia_i= \delta_e}}  \frac{\delta_e!^2}{\prod_{k=1}^{\delta_e} k!^{2a_k} }\sum_{\alpha(\cdot) \in R(\vb*{a})} (-1)^{\sum_{\vb*{x} \in \mathds{N}^{\delta_e}} (\sum_{i=1}^{\delta_e} x_i - 1)\alpha(\vb*{x})}\\
     &\hspace{1cm}\times 
    \prod_{\vb*{y} \in \mathds{N}^{\delta_e}} \Bigg( \frac{1}{\alpha(\vb*{y})!} \Big(\frac{n\mathds{E}\big[\prod_{i=1}^{\delta_e} (\mathds{1}_{\{d_U^{\tout},d_U^{\tin}\geq i\}}\frac{d_U^{\tout}!}{(d_U^{\tout}-i)!}\frac{d_U^{\tin}!}{(d_U^{\tin}-i)!})^{y_i}\big]\big(\sum_{i=1}^{\delta_e} y_i-1 \big)!}{\prod_{i=1}^{\delta_e} y_i!} \Big)^{\alpha(\vb*{y})} \Bigg),
\end{align*}
which yields the result when plugged into Equation (\ref{eq:E[self-loops]}).
\end{proof}

We prove Lemma \ref{lemma:E[sl]_asymp} by applying Corollary \ref{cor:general_order}.

\begin{proof}[Proof of Lemma \ref{lemma:E[sl]_asymp}]
We consider
\begin{align*}
    \mathds{E}[S_n]
    &= |E|\,
       \frac{(n\mathds{E}[d_U^{\tin}] - \delta)!\,(n\mathds{E}[d_U^{\tout}] - \delta)!}
            {(n\mathds{E}[d_U^{\tin}])!\,(n\mathds{E}[d_U^{\tout}])!}
       \sum_{\substack{\vb*{a} \in \mathds{N}^{\delta} \\ \sum_{i=1}^{\delta} i a_i = \delta}}
       \ \sum_{\alpha(\cdot) \in R(\vb*{a})} H(\vb*{a}, \alpha(\cdot)),
\end{align*}
where $H(\cdot,\cdot)$ is as in \eqref{eq:Hterms}, with $w=1$ and

\[
    f^{(i)}(v)
    = \mathds{1}_{\{d_v^{\tout},d_v^{\tin} \geq i\}}\frac{d_v^{\tout}!}{(d_v^{\tout}-i)!}\,
      \frac{d_v^{\tin}!}{(d_v^{\tin}-i)!}.
\]

All requirements for Corollary \ref{cor:general_order} are met.  
Indeed, for all $\vb*{y} \in \mathds{N}^{\delta}$ with $\sum_{i=1}^{\delta} i y_i \leq \delta$ where $\sum_{i=1}^{\delta} \mathds{1}_{\{y_i \geq 1\}} \geq 2$, we have
\begin{align*}
    \mathds{E}\!\left[\prod_{i=1}^{\delta} (f^{(i)}(U))^{y_i} \right]
    &= \mathds{E}\!\left[
        \prod_{i=1}^{\delta}
        \left(\mathds{1}_{\{d_U^{\tout},d_U^{\tin} \geq i\}}\frac{d_U^{\tout}!}{(d_U^{\tout} - i)!} \frac{d_U^{\tin}!}{(d_U^{\tin} - i)!}\right)^{y_i} 
      \right] \le \mathds{E}\!\left[
        \prod_{i=1}^{\delta} (d_U^{\tout}d_U^{\tin})^{iy_i} 
      \right] \\
      &= \mathds{E}\!\left[(d_U^{\tout}d_U^{\tin})^{\sum_{i=1}^{\delta} iy_i} \right] \leq \mathds{E}\!\left[(d_U^{\tout}d_U^{\tin})^{\delta} \right]
     = o(n).
\end{align*}
Moreover, for all $k \in [\delta]$:
\begin{align*}
\mathds{E}[(f^{(k)}(U))^{\lfloor \frac{\delta}{k} \rfloor}] &= \mathds{E}\Big[\big(\mathds{1}_{\{d_U^{\tout}, d_U^{\tin} \geq k\}}\frac{d_U^{\tout}!}{(d_U^{\tout} - k)!} \frac{d_U^{\tin}!}{(d_U^{\tin} - k)!} \big)^{\lfloor \frac{\delta}{k} \rfloor}\Big] \leq \mathds{E}[(d_U^{\tout} d_U^{\tin})^{\delta}] = o(n).
\end{align*}
Lastly, for all $k \in [\delta]$:
\begin{align*}
    \mathds{P}(f^{(k)}(U) \geq 1) = \mathds{P}\Big(\mathds{1}_{\{d_U^{\tout},d_U^{\tin} \geq k\}}\frac{d_U^{\tout}!}{(d_U^{\tout} - k)!} \frac{d_U^{\tin}!}{(d_U^{\tin} - k)!} \geq 1 \Big) = \mathds{P}(d_U^{\tout} , d_U^{\tin} \geq k) \geq \mathds{P}(d_U^{\tout}, d_U^{\tin} \geq \delta) \geq c.
\end{align*}

By Corollary \ref{cor:general_order},
\begin{align*}
    \mathds{E}[S_n]
    &= |E|\,
       \frac{(n\mathds{E}[d_U^{\tin}] - \delta)!\,(n\mathds{E}[d_U^{\tout}] - \delta)!}
            {(n\mathds{E}[d_U^{\tin}])!\,(n\mathds{E}[d_U^{\tout}])!}
       \, \delta!\,
       \bigl(n\,\mathds{E}[d_U^{\tin}d_U^{\tout}]\bigr)^{\delta}
       (1+o(1)).
\end{align*}

Now,
\begin{align*}
    \frac{(n\mathds{E}[d_U^{\tin}] - \delta)!\,(n\mathds{E}[d_U^{\tout}] - \delta)!}
         {(n\mathds{E}[d_U^{\tin}])!\,(n\mathds{E}[d_U^{\tout}])!}
    &= \frac{1}{
        (n^2 \mathds{E}[d_U^{\tin}] \mathds{E}[d_U^{\tout}])^{\delta}
        \bigl(1 - O(\delta^2/(n\mathds{E}[d_U^{\tin}]))\bigr)
        \bigl(1 - O(\delta^2/(n\mathds{E}[d_U^{\tout}]))\bigr)
       } \\
    &= \frac{1}{
        (n^2 \mathds{E}[d_U^{\tin}] \mathds{E}[d_U^{\tout}])^{\delta}
        (1 - O(1/|E|))
       }.
\end{align*}

Thus,
\begin{align*}
    \mathds{E}[S_n]
    &= |E|\,
       \frac{
         \delta!\,
         \bigl(n\,\mathds{E}[d_U^{\tin}d_U^{\tout}]\bigr)^{\delta}
       }{
         (n^2\mathds{E}[d_U^{\tin}]\mathds{E}[d_U^{\tout}])^{\delta}
         (1 - O(1/|E|))
       }
       (1+o(1)) \\
    &= (\delta-1)!\,
       \frac{
         \mathds{E}[d_U^{\tin}d_U^{\tout}]^{\delta}
       }{
         (n\mathds{E}[d_U^{\tin}])^{\delta-1}\,
         \mathds{E}[d_U^{\tout}]^{\delta}
       }
       (1+o(1)) \\
    &= (\delta-1)!\,
       \frac{
         \mathds{E}[d_U^{\tin}d_U^{\tout}]^{\delta}
       }{
         n^{\delta-1}\,\mathds{E}[d_U^{\tin}]^{2\delta-1}
       }
       (1+o(1)),
\end{align*}
where we used $|E| = \frac{n\mathds{E}[d_U^{\tin}]}{\delta}=\frac{n\mathds{E}[d_U^{\tout}]}{\delta}$ and
$\mathds{E}[d_U^{\tout}] = \mathds{E}[d_U^{\tin}]$.
\end{proof}

\subsubsection{Weak self-loops}
\label{section:pf_weak_sl}
We prove Theorem \ref{thm:E[weak-sl]} using Lemma \ref{lemma:main_lemma}.

\begin{proof}[Proof of Theorem \ref{thm:E[weak-sl]}]
First,
\begin{align}
\label{eq:E[weak_sl]}
    \mathds{E}[WS_n]
    &= |E|
       - \mathds{E}[\#\textnormal{hyperedges with } e^{\tail} \cap e^{\head} = \emptyset].
\end{align}

We now analyze the latter expectation. We have
\begin{align}
\label{eq:E[non-weak_sl]}
    \mathds{E}[\#\textnormal{hyperedges with } e^{\tail} \cap e^{\head} = \emptyset]
    &= \sum_{e \in E}
       \mathds{P}(e^{\tail} \cap e^{\head} = \emptyset).
\end{align}

Now,
\begin{align*}
    \mathds{P}(e^{\tail}\cap e^{\head} = \emptyset)
    &=
    \sum_{\vb*{v} \in V^{\delta_e^{\tail}}}
    \sum_{\vb*{w} \in (V \setminus \{v_1,\ldots,v_{\delta_e^{\tail}}\})^{\delta_e^{\head}}}
        \frac{1}{D(\vb*{v}) D(\vb*{w})}
        \\
    &\qquad\qquad \times
        \mathds{P}(e^{\tail} = \{v_1,\ldots,v_{d_{e^{\tail}}}\})
        \mathds{P}(e^{\head} = \{w_1,\ldots,w_{d_{e^{\head}}}\}),
\end{align*}
where $D(\vb*{v})$ counts the number of ways the vertex multiset
$\{v_1,\ldots,v_{d_{e^{\tail}}}\}$
appears as an ordered tuple
$\vb*{v} = (v_1,\ldots,v_{d_{e^{\tail}}})$. Let $v'_1,\ldots,v'_{u(\vb*{v})}$ denote the unique vertices in $\vb*{v}$. Then $ D(\vb*{v})
    = \frac{
        \delta_e!
    }{
        m_{v'_1}(\vb*{v})!\,
        m_{v'_2}(\vb*{v})!\,
        \hdots\,
        m_{v'_{u(\vb*{v})}}(\vb*{v})!
    }$. Similarly, $D(\vb*{w})$ counts the number of ways the multiset  
$\{w_1,\ldots,w_{d_{e^{\head}}}\}$  
appears as tuple \\
$\vb*{w} = (w_1,\ldots,w_{d_{e^{\head}}})$. Now, if for some $i \in [u(\vb*{v})]$ we have $
    d_{v_i}^{\tout} \le m_{v'_i}(\vb*{v}) - 1,
$
or for some $j \in [u(\vb*{w})]$ we have $d_{w_j}^{\tin} \le m_{w'_j}(\vb*{w}) - 1$,
then
\[
    \mathds{P}(e^{\tail}=\{v_1,\ldots,v_{\delta_e^{\tail}}\})
    \mathds{P}(e^{\head}=\{w_1,\ldots,w_{\delta_e^{\head}}\})
    = 0.
\]

If instead $\forall i \in [u(\vb*{v})]:\ d_{v'_i}^{\tout} \ge m_{v'_i}(\vb*{v})$ and $
    \forall j \in [u(\vb*{w})]:\ d_{w'_j}^{\tin} \ge m_{w'_j}(\vb*{w})$, then
\begin{align*}
    \mathds{P}(e^{\tail}=\{v_1,\ldots,v_{d_{e^{\tail}}}\})
    &=
    D(\vb*{v})
    \prod_{i=1}^{\delta_e^{\tail}}
        \mathds{P}(v_i \to e^{\tail}
            \mid
            \forall j < i: v_j \to e^{\tail})
    \\
    &=
    D(\vb*{v})
    \prod_{i=1}^{\delta_e^{\tail}}
        \frac{
            d_{v_i}^{\tout}
            - \sum_{j=1}^{i-1} \mathds{1}_{\{v_j = v_i\}}
        }{
            n\mathds{E}[d_U^{\tout}] - (i-1)
        }
    \\
    &=
    \frac{
        (n\mathds{E}[d_U^{\tout}] - \delta_e^{\tail})!
    }{
        (n\mathds{E}[d_U^{\tout}])!
    }
    D(\vb*{v})
    \prod_{i=1}^{u(\vb*{v})}
        \frac{
            d_{v'_i}^{\tout}! 
        }{
            (d_{v'_i}^{\tout} - m_{v'_i}(\vb*{v}))!
        }.
\end{align*}
And similarly,
\begin{align*}
    \mathds{P}(e^{\head}=\{w_1,w_2,\hdots,w_{d_{e^{\head}}}\}) &= D(\vb*{w}) \prod_{i=1}^{\delta_e^{\head}} \mathds{P}(w_i \rightarrow e^{\head} | \forall j \in [i-1]: w_j \rightarrow e^{\head})\\
    &= D(\vb*{w}) \prod_{i=1}^{\delta_e^{\head}} \frac{d_{w_i}^{\tin} - \sum_{j=1}^{i-1} \mathds{1}_{\{w_j=w_i\}}}{n\mathds{E}[d_U^{\tin}] -(i-1)}\\
    &= \frac{(n\mathds{E}[d_U^{\tin}]-\delta_e^{\head})!}{(n\mathds{E}[d_U^{\tin}])!}D(\vb*{w}) \prod_{i=1}^{u(\vb*{w})}  \frac{d_{w'_i}^{\tin}!}{(d_{w'_i}^{\tin}-m_{w'_i}(\vb*{w}))!}.
\end{align*}
Therefore,
\begin{align*}
    &\mathds{P}(e^{\tail}\cap e^{\head} = \emptyset) \\
    &= \frac{(n\mathds{E}[d_U^{\tout}]-\delta_e^{\tail})!(n\mathds{E}[d_U^{\tin}]-\delta_e^{\head})!}{(n\mathds{E}[d_U^{\tout}])!(n\mathds{E}[d_U^{\tin}])!}\sum_{\vb*{v} \in V^{\delta_e^{\tail}}} \sum_{\vb*{w} \in (V \backslash \{v_1,\hdots,v_{\delta_e^{\tail}}\})^{\delta_e^{\head}}} \prod_{i=1}^{u(\vb*{v})} \Big( \mathds{1}_{\{d_{v'_i}^{\tout} \geq m_{v'_i}(\vb*{v})\}}\frac{d_{v'_i}^{\tout}!}{(d_{v'_i}^{\tout}-m_{v'_i}(\vb*{v}))!} \Big) \\
    &\hspace{1cm} \times \prod_{i=1}^{u(\vb*{w})}  \Big( \mathds{1}_{\{d_{w'_i}^{\tin} \geq m_{w'_i}(\vb*{w})\}}\frac{d_{w'_i}^{\tin}!}{(d_{w'_i}^{\tin}-m_{w'_i}(\vb*{w}))!} \Big).
\end{align*}
Now we apply Lemma \ref{lemma:main_lemma}, with $w=0$, $f_1^i(v) = \mathds{1}_{\{d_v^{\tout} \geq i\}} \frac{d_v^{\tout}!}{(d_v^{\tout}-i)!}$ and $f_2^i(v) = \mathds{1}_{\{d_v^{\tin} \geq i\}} \frac{d_v^{\tin}!}{(d_v^{\tin}-i)!}$. We obtain
\begin{align*}
    &\mathds{P}(e^{\tail} \cap e^{\head} = \emptyset)\\
    &= \frac{(n\mathds{E}[d_U^{\tout}]-\delta_e^{\tail})!(n\mathds{E}[d_U^{\tin}]-\delta_e^{\head})!}{(n\mathds{E}[d_U^{\tout}])!(n\mathds{E}[d_U^{\tin}])!} \sum_{\substack{\vb*{a} \in \mathds{N}^{\delta_e^{\tail}}:\\ \sum_{i=1}^{\delta_e^{\tail}} ia_i = \delta_e^{\tail}}}  \frac{\delta_e^{\tail}!}{\prod_{k=1}^{\delta_e^{\tail}} (k!^{a_k} a_k!)} \sum_{\substack{\vb*{b} \in \mathds{N}^{\delta_e^{\head}}:\\ \sum_{i=1}^{\delta_e^{\head}} ib_i = \delta_e^{\head}}}  \frac{\delta_e^{\head}!}{\prod_{k=1}^{\delta_e^{\head}} (k!^{b_k} b_k!)} \nonumber \\
    &\hspace{0.5cm} \times \sum_{\gamma(\cdot) \in \hat{R}(\vb*{a},\vb*{b})} (-1)^{\sum{\vb*{g} \in \mathds{N}^{\delta_e^{\tail}}, \vb*{h} \in \mathds{N}^{\delta_e^{\head}}}( \sum_{i=1}^{\delta_e^{\head}} g_i + \sum_{j=1}^{\delta_e^{\head}} h_j -1) \gamma(\vb*{g},\vb*{h})} \nonumber \\
    &\hspace{0.5cm} \times \prod_{\substack{\vb*{y} \in \mathds{N}^{\delta_e^{\tail}}\\ \vb*{z} \in \mathds{N}^{\delta_e^{\head}}}} \Bigg( \frac{1}{\gamma(\vb*{y},\vb*{z})!} \Big(\frac{n\mathds{E}\big[\prod_{i=1}^{\delta_e^{\tail}} (f_1^i(U))^{y_i} \prod_{j=1}^{\delta_e^{\head}} (f_2^j(U))^{z_j}\big]\big(\sum_{i=1}^{\delta_e^{\tail}} y_i-1\big)!\big(\sum_{j=1}^{\delta_e^{\head}} z_j-1 \big)!}{\prod_{i=1}^{\delta_e^{\tail}} y_i! \prod_{j=1}^{\delta_e^{\head}} z_j!} \Big)^{\gamma(\vb*{y},\vb*{z})} \Bigg),
\end{align*}
which yields the result when plugged into Equation (\ref{eq:E[weak_sl]}), combined with Equation \eqref{eq:E[non-weak_sl]}.
\end{proof}

Now we prove Lemma \ref{lemma:E[weak-sl]_asymp} by applying Lemma \ref{lemma:general_order}.

\begin{proof}[Proof of Lemma \ref{lemma:E[weak-sl]_asymp}]
    We consider 
    \begin{align*}
        \mathds{E}[WS_n] &= |E| - |E| \frac{(n\mathds{E}[d_U^{\tout}]-\delta^{\tail})!(n\mathds{E}[d_U^{\tin}]-\delta^{\head})!}{(n\mathds{E}[d_U^{\tout}])!(n\mathds{E}[d_U^{\tin}])!}\sum_{\substack{\vb*{a} \in \mathds{N}^{\delta^{\tail}}:\\ \sum_{i=1}^{\delta^{\tail}} ia_i = \delta^{\tail}}}  \sum_{\substack{\vb*{b} \in \mathds{N}^{\delta^{\head}}:\\ \sum_{i=1}^{\delta^{\head}} ib_i = \delta^{\head}}} \sum_{\gamma(\cdot) \in \hat{R}(\vb*{a},\vb*{b})} H(\vb*{a},\vb*{b},\gamma(\cdot)),
    \end{align*}
    where $H(\cdot, \cdot)$ is as in \eqref{eq:Hterms}, with $w=0$, $f_1^i(v) = \mathds{1}_{\{d_v^{\tout} \geq i\}} \frac{d_v^{\tout}!}{(d_v^{\tout}-i)!}$ and $f_2^i(v) = \mathds{1}_{\{d_v^{\tin} \geq i\}} \frac{d_v^{\tin}!}{(d_v^{\tin}-i)!}$. 
     All requirements for Lemma \ref{lemma:general_order} are met.  
Indeed, for all $\vb*{y} \in \mathds{N}^{\delta^{\tail}}$ with $\sum_{i=1}^{\delta^{\tail}} i y_i \leq \delta^{\tail}$ and for all $\vb*{z} \in \mathds{N}^{\delta^{\head}}$ with $\sum_{j=1}^{\delta^{\head}} j z_j \leq \delta^{\head}$ where $\sum_{i=1}^{\delta^{\tail}} \sum_{j=1}^{\delta^{\head}} \mathds{1}_{\{y_i \geq 1\}} + \mathds{1}_{\{z_j \geq 1\}} \geq 2$, we have
\begin{align*}
    \mathds{E}\!\left[\prod_{i=1}^{\delta^{\tail}} (f_1^{(i)}(U))^{y_i} \prod_{j=1}^{\delta^{\head}} (f_2^{(j)}(U))^{z_j} \right]
    &= \mathds{E}\!\left[
        \prod_{i=1}^{\delta^{\tail}}
        \left(\mathds{1}_{\{d_U^{\tout} \geq i\}}\frac{d_U^{\tout}!}{(d_U^{\tout} - i)!} \right)^{y_i} \prod_{j=1}^{\delta^{\head}}
        \left(\mathds{1}_{\{d_U^{\tin} \geq j\}}\frac{d_U^{\tin}!}{(d_U^{\tin} - j)!} \right)^{z_j} 
      \right] \\
      &\le \mathds{E}\!\left[
        \prod_{i=1}^{\delta^{\tail}}
        (d_U^{\tout})^{iy_i} \prod_{j=1}^{\delta^{\head}}
        (d_U^{\tin})^{jz_j} 
      \right] \\
      &= \mathds{E}\!\left[
        (d_U^{\tout})^{\sum_{i=1}^{\delta^{\tail}}iy_i} 
        (d_U^{\tin})^{\sum_{j=1}^{\delta^{\head}}jz_j} 
      \right] \\
      &\leq \mathds{E}\!\left[
        (d_U^{\tout})^{\delta^{\tail}}
        (d_U^{\tin})^{\delta^{\head}} 
      \right]
     = o(n).
\end{align*}
Moreover, for all $k \in [\delta^{\tail}]$:
\begin{align*}
\mathds{E}\big[(f_1^{(k)}(U))^{\lfloor \frac{\delta^{\tail}}{k} \rfloor} \big] &= \mathds{E}\Big[\Big(\mathds{1}_{\{d_U^{\tout} \geq k\}}\frac{d_U^{\tout}!}{(d_U^{\tout} - k)!}\Big)^{\lfloor \frac{\delta^{\tail}}{k} \rfloor}\Big] \leq \mathds{E}[(d_U^{\tout})^{\delta^{\tail}}] = o(n)
\end{align*}
and for all $l \in [\delta^{\head}]$:
\begin{align*}
\mathds{E}\big[(f_2^{(l)}(U))^{\lfloor \frac{\delta^{\head}}{l} \rfloor} \big] &= \mathds{E}\Big[\Big(\mathds{1}_{\{d_U^{\tin} \geq l\}}\frac{d_U^{\tin}!}{(d_U^{\tin} - l)!}\Big)^{\lfloor \frac{\delta^{\head}}{l} \rfloor}\Big] \leq \mathds{E}[(d_U^{\tin})^{\delta^{\head}}] = o(n)
\end{align*}
Lastly, for all $k \in [\delta^{\tail}]$ and for $n$ large enough:
\begin{align*}
    \mathds{P}(f_1^{(k)}(U) \geq 1) = \mathds{P}\Big(\mathds{1}_{\{d_U^{\tout} \geq k\}} \frac{d_U^{\tout}!}{(d_U^{\tout} - k)!} \geq 1\Big) = \mathds{P}(d_U^{\tout} \geq k) \geq \mathds{P}(d_U^{\tout} \geq \delta^{\tail}) \geq c
\end{align*}
and for all $l \in [\delta^{\head}]$ and for $n$ large enough:
\begin{align*}
    \mathds{P}(f_2^{(l)}(U) \geq 1) = \mathds{P}\Big(\mathds{1}_{\{d_U^{\tin} \geq l\}} \frac{d_U^{\tin}!}{(d_U^{\tin} - l)!} \geq 1\Big) = \mathds{P}(d_U^{\tin} \geq l) \geq \mathds{P}(d_U^{\tin} \geq \delta^{\head}) \geq c.
\end{align*}
     Therefore, by Lemma \ref{lemma:general_order},
    \begin{align*}
        \mathds{E}[WS_n] &= |E| - |E|\frac{(n\mathds{E}[d_U^{\tout}]-\delta^{\tail})!(n\mathds{E}[d_U^{\tin}]-\delta^{\head})!}{(n\mathds{E}[d_U^{\tout}])!(n\mathds{E}[d_U^{\tin}])!} \Bigg( (n\mathds{E}[d_U^{\tout} ])^{\delta^{\tail}} (n\mathds{E}[d_U^{\tin} ])^{\delta^{\head}} \nonumber\\
 &\hspace{0.5cm}+ \Big( -\frac{1}{2} \frac{\delta^{\tail}!}{(\delta^{\tail}-2)!}(n\mathds{E}[d_U^{\tout}])^{\delta^{\tail}-2} n\mathds{E}[(d_U^{\tout})^2] (n\mathds{E}[d_U^{\tin}])^{\delta^{\head}} \nonumber\\
 &\hspace{1.15cm} -\frac{1}{2} \frac{ \delta^{\head}!}{(\delta^{\head}-2)!} (n\mathds{E}[d_U^{\tout}])^{\delta^{\tail}} (n\mathds{E}[d_U^{\tin}])^{\delta^{\head}-2} n\mathds{E}[(d_U^{\tin})^2] \nonumber\\
 &\hspace{1.15cm} - \delta^{\tail}\delta^{\head}  (n\mathds{E}[d_U^{\tout}])^{\delta^{\tail}-1} (n\mathds{E}[d_U^{\tin}])^{\delta^{\head}-1} n\mathds{E}[d_U^{\tout}d_U^{\tin}] \nonumber\\
 &\hspace{1.15cm} + \frac{1}{2} \frac{\delta^{\tail}!}{(\delta^{\tail}-2)!} (n\mathds{E}[d_U^{\tout}])^{\delta^{\tail}-2} n\mathds{E}[d_U^{\tout}(d_U^{\tout}-1)] (n\mathds{E}[d_U^{\tin}])^{\delta^{\head}} \nonumber\\
 &\hspace{1.15cm}+ \frac{1}{2} \frac{ \delta^{\head}!}{(\delta^{\head}-2)!} (n\mathds{E}[d_U^{\tout}])^{\delta^{\tail}} (n\mathds{E}[d_U^{\tin}])^{\delta^{\head}-2} n\mathds{E}[d_U^{\tin}(d_U^{\tin}-1)] \Big) (1+o(1)) \Bigg).
    \end{align*}
    Since $\mathds{E}[d_U^{\tout}],\mathds{E}[\delta_e^{\tin}] \geq c > 0$, we obtain
   \begin{align*}
    &\frac{(n\mathds{E}[d_U^{\tin}] - \delta^{\tail})!\,(n\mathds{E}[d_U^{\tout}] - \delta^{\head})!}
         {(n\mathds{E}[d_U^{\tin}])!\,(n\mathds{E}[d_U^{\tout}])!}\nonumber\\
    &= \frac{1}{
        (n \mathds{E}[d_U^{\tout}])^{\delta^{\tail}}(n \mathds{E}[d_U^{\tin}])^{\delta^{\head}}
        \bigl(1 - O(d^2/(n\mathds{E}[d_U^{\tin}]))\bigr)
        \bigl(1 - O(d^2/(n\mathds{E}[d_U^{\tout}]))\bigr)
       } \\
    &= \frac{1}{
        (n \mathds{E}[d_U^{\tout}])^{\delta^{\tail}}(n \mathds{E}[d_U^{\tin}])^{\delta^{\head}}
        (1 - O(1/|E|))
       }.
\end{align*}
    
    Then,
    \begin{align*}
        \mathds{E}[WS_n] &= |E| - |E|\frac{1}{
        (n \mathds{E}[d_U^{\tout}])^{\delta^{\tail}}(n \mathds{E}[d_U^{\tin}])^{\delta^{\head}}
        (1 - O(1/|E|))} \Bigg( (n\mathds{E}[d_U^{\tout} ])^{\delta^{\tail}} (n\mathds{E}[d_U^{\tin} ])^{\delta^{\head}} \nonumber\\
 &\hspace{0.5cm}+ \Big( -\frac{1}{2} \delta^{\tail}(\delta^{\tail}-1)(n\mathds{E}[d_U^{\tout}])^{\delta^{\tail}-2} n\mathds{E}[(d_U^{\tout})^2] (n\mathds{E}[d_U^{\tin}])^{\delta^{\head}} \nonumber\\
 &\hspace{1.15cm} -\frac{1}{2} \delta^{\head}(\delta^{\head}-1) (n\mathds{E}[d_U^{\tout}])^{\delta^{\tail}} (n\mathds{E}[d_U^{\tin}])^{\delta^{\head}-2} n\mathds{E}[(d_U^{\tin})^2] \nonumber\\
 &\hspace{1.15cm} - \delta^{\tail}\delta^{\head}  (n\mathds{E}[d_U^{\tout}])^{\delta^{\tail}-1} (n\mathds{E}[d_U^{\tin}])^{\delta^{\head}-1} n\mathds{E}[d_U^{\tout}d_U^{\tin}] \nonumber\\
 &\hspace{1.15cm} + \frac{1}{2}\delta^{\tail}(\delta^{\tail}-1) (n\mathds{E}[d_U^{\tout}])^{\delta^{\tail}-2} n\mathds{E}[d_U^{\tout}(d_U^{\tout}-1)] (n\mathds{E}[d_U^{\tin}])^{\delta^{\head}} \nonumber\\
 &\hspace{1.15cm}+ \frac{1}{2} \delta^{\head}(\delta^{\head}-1) (n\mathds{E}[d_U^{\tout}])^{\delta^{\tail}} (n\mathds{E}[d_U^{\tin}])^{\delta^{\head}-2} n\mathds{E}[d_U^{\tin}(d_U^{\tin}-1)] \Big) (1+o(1)) \Bigg)\\
 &= -|E|\Big( -\frac{1}{2} \delta^{\tail}(\delta^{\tail}-1)(n\mathds{E}[d_U^{\tout}])^{-2} n\mathds{E}[(d_U^{\tout})^2] \nonumber\\
 &\hspace{1.4cm} -\frac{1}{2} \delta^{\head}(\delta^{\head}-1) (n\mathds{E}[d_U^{\tin}])^{-2} n\mathds{E}[(d_U^{\tin})^2] \nonumber\\
 &\hspace{1.4cm} - \delta^{\tail}\delta^{\head}  (n\mathds{E}[d_U^{\tout}])^{-1} (n\mathds{E}[d_U^{\tin}])^{-1} n\mathds{E}[d_U^{\tout}d_U^{\tin}] \nonumber\\
 &\hspace{1.4cm} + \frac{1}{2}\delta^{\tail}(\delta^{\tail}-1) (n\mathds{E}[d_U^{\tout}])^{-2} n\mathds{E}[d_U^{\tout}(d_U^{\tout}-1)]  \nonumber\\
 &\hspace{1.4cm}+ \frac{1}{2} \delta^{\head}(\delta^{\head}-1)  (n\mathds{E}[d_U^{\tin}])^{-2} n\mathds{E}[d_U^{\tin}(d_U^{\tin}-1)] \Big) (1+o(1))\\
 &= \Big( \frac{(\delta^{\tail}-1)(\mathds{E}[(d_U^{\tout})^2]- \mathds{E}[d_U^{\tout}(d_U^{\tout}-1)])}{2\mathds{E}[d_U^{\tout}]} + \frac{(\delta^{\head}-1)(\mathds{E}[(d_U^{\tin})^2]- \mathds{E}[d_U^{\tin}(d_U^{\tin}-1)])}{2\mathds{E}[d_U^{\tin}]}\\
 &\hspace{0.5cm} + \frac{\delta^{\head} \mathds{E}[d_U^{\tout}d_U^{\tin}]}{\mathds{E}[d_U^{\tin}]} \Big)(1+o(1))\\
 &= \Big( \frac{\delta^{\tail} + \delta^{\head}-2}{2} + \frac{\delta^{\head} \mathds{E}[d_U^{\tout}d_U^{\tin}]}{\mathds{E}[d_U^{\tin}]} \Big)(1+o(1))
    \end{align*}
    where we have used that $|E| = \frac{n \mathds{E}[d_U^{\tout}]}{\delta^{\tail}}=\frac{n \mathds{E}[d_U^{\tin}]}{\delta^{\head}}$.
\end{proof}

\appendix
\section{Proof of Lemma \ref{lemma:main_lemma}}
\label{app:pf_main_lemma}

Here, we prove Lemma \ref{lemma:main_lemma}.
\begin{proof}[Proof of Lemma \ref{lemma:main_lemma}]
For a fixed $v \in V^{\delta_1}$, let us describe the vertices in $\vb*{v}$ by the multiplicity vector $\vb*{a} \in \mathds{N}^{\delta_1}$, where $a_i = |\{v \in V: m_v(\vb*{v}) = i\}|$ denotes the number of distinct vertices that appear in $\vb*{v}$ with multiplicity exactly $i$. Note that the tuple $\vb*{a}$ does not uniquely specify which vertices in $\vb*{v}$ have which multiplicity: for $\delta_1=3$, the vertices $v_1$, $v_2$ and $v_3$ can meet $\vb*{a}=\vb*{e}_1+\vb*{e}_2$ if $v_1=v_2$ and $v_3$ is unique, or if $v_1$ is unique and $v_2=v_3$ or if $v_1=v_3$ and $v_2$ is unique. More specifically, we can think of each such case as a partition of the vertices $\{v_1,v_2,v_3\}$ into sets of equal vertices, i.e., $\{v_1,v_2\},\{v_3\}$ in the first, $\{v_1\},\{v_2,v_3\}$ in the second and $\{v_1,v_3\},\{v_2\}$ in the third case. 

We now count the number of partitions of $v_1,\hdots,v_{\delta_1}$ that yield some fixed multiplicity tuple $\vb*{a}$. To that end, we first count the number of ways the $a_1$ vertices of multiplicity 1 can be picked, followed by the number of ways the $a_2$ vertices of multiplicity 2 can be picked from the remaining vertices, etc. More generally, we count the number of ways that $a_k$ vertex sets of size $k$ can be constructed, given that $a_r$ vertex sets of size $r=1,2,\hdots,k-1$ are already constructed. To construct the $a_k$ vertex sets, $\delta_1-\sum_{i=1}^{k-1} ia_i$ vertices can be used, as they are not part of a previously constructed vertex set. Generating $a_k$ sets of $k$ vertices out of $\delta_1-\sum_{i=1}^{k-1} ia_i$ vertices can be done in
\begin{align*}
    \binom{\delta_1-\sum_{i=1}^{k-1} ia_1}{k}\binom{\delta_1-\sum_{i=1}^{k-1} ia_i-k}{k}\hdots \binom{\delta_1-\sum_{i=1}^{k-1} ia_i-k(a_k-1)}{k}\frac{1}{a_k!} = \frac{(\delta_1-\sum_{i=1}^{k-1} ia_i)!}{k!^{a_k} a_k! (\delta_1-\sum_{i=1}^{k} ia_i)! }
\end{align*}
different ways. Combining over all $k$, we obtain
\begin{align*}
    \textnormal{\#}\textnormal{partitions of }\vb*{v} \textnormal{ that yield }\vb*{a} &= \prod_{k=1}^{\delta_1} \frac{(\delta_1-\sum_{i=1}^{k-1} ia_i)!}{k!^{a_k} a_k! (\delta_1-\sum_{i=1}^{k} ia_i)! } = \frac{\delta_1!}{\prod_{k=1}^{\delta_1} (k!^{a_k} a_k!)}.
\end{align*}

Similarly, we consider all possible partitions of $\vb*{w}$. For a fixed choice of $w_1,\hdots,w_{\delta_2}$, let $b_k=|\{w \in V: m_w(\vb*{w}) = k\}|$ describe the number of vertices with multiplicity $k$, for any $k \in [\delta_2]$. Similarly as for $\vb*{v}$,
\begin{align*}
    \textnormal{\#}\textnormal{partitions of }\vb*{w} \textnormal{ that yield }\vb*{b} = \frac{\delta_2!}{\prod_{k=1}^{\delta_2} (k!^{b_k} b_k!)}.
\end{align*}
Now we can rewrite the sum over all vertices as a sum over all possible partitions into sets. We obtain
\begin{align*}
    \sum_{v \in V^{\delta_1}} \sum_{\vb*{w} \in (V \backslash \{v_1,\hdots,v_{\delta_1}\})^{\delta_2}} 
    &= \sum_{\substack{\vb*{a} \in \mathds{N}^{\delta_1}:\\ \sum_{i=1}^{\delta_1} ia_i = \delta_1}}  \frac{\delta_1!}{\prod_{k=1}^{\delta_1} (k!^{a_k} a_k!)} \sideset{}{^*}\sum_{\substack{s_1^{(1)} \in V^{a_1}\\s_1^{(2)} \in V^{a_2}\\ \vdots \\ s_1^{(\delta_1)} \in V^{a_{\delta_1}}}} \sum_{\substack{\vb*{b} \in \mathds{N}^{\delta_2}:\\ \sum_{i=1}^{\delta_2} ib_i = \delta_2}}  \frac{\delta_2!}{\prod_{k=1}^{\delta_2} (k!^{b_k} b_k!)} \sideset{}{^*}\sum_{\substack{s_2^{(1)} \in V^{b_1}\\s_2^{(2)} \in V^{b_2}\\ \vdots \\ s_2^{(\delta_2)} \in V^{b_{\delta_2}}}}\\
    &= \sum_{\substack{\vb*{a} \in \mathds{N}^{\delta_1}:\\ \sum_{i=1}^{\delta_1} ia_i = \delta_1}}  \frac{\delta_1!}{\prod_{k=1}^{\delta_1} (k!^{a_k} a_k!)} \sum_{\substack{\vb*{b} \in \mathds{N}^{\delta_2}:\\ \sum_{i=1}^{\delta_2} ib_i = \delta_2}}  \frac{\delta_2!}{\prod_{k=1}^{\delta_2} (k!^{b_k} b_k!)}  \sideset{}{^*}\sum_{\vb*{s} \in V^{h(\vb*{a})+h(\vb*{b})}}, 
    \end{align*}
where $s^{(i)}_{1,j}$ is the $j$'th vertex with multiplicity $i$ in $\vb*{v}$, with $h(\vb*{a}) = \sum_{i=1}^{\delta_1} a_i = u(\vb*{v})$ the total number of distinct vertices in $\vb*{v}$. Similarly, $s^{(i)}_{2,j}$ is the $j$'th vertex with multiplicity $i$ in $\vb*{w}$, with $h(\vb*{b}) = \sum_{i=1}^{\delta_2} b_i = u(\vb*{w})$ the total number of distinct vertices in $\vb*{w}$. Note that, since no elements of $\vb*{w}$ are in $\{v_1,\hdots,v_{\delta_1}\}$, the vertices in $\vb*{v}$ are indeed distinct from the vertices in $\vb*{w}$. Also, observe that
\begin{align*}
    D(\vb*{v}) = \frac{\delta_1!}{1!^{a_1}2!^{a_2}\hdots \delta_1!^{a_{\delta_1}}} = \frac{\delta_1!}{\prod_{k=1}^{\delta_1} k!^{a_k}}
\end{align*}
and similarly
\begin{align*}
    D(\vb*{w}) = \frac{\delta_2!}{1!^{b_1}2!^{b_2}\hdots \delta_2!^{b_{\delta_2}}} = \frac{\delta_2!}{\prod_{k=1}^{\delta_2} k!^{b_k}}.
\end{align*}
We obtain
\begin{align}
\label{eq:sum_v_to_sum_a_2}
    &\sum_{\vb*{v} \in V^{\delta_1}}  \sum_{\vb*{w} \in (V\backslash\{v_1,\hdots,v_{\delta_1}\})^{\delta_2}} (D(\vb*{v})D(\vb*{w}))^w \prod_{i=1}^{u(\vb*{v})} f_1^{(m_{v'_i}(\vb*{v}))}(v'_i) \prod_{j=1}^{u(\vb*{w})} f_2^{(m_{w'_j}(\vb*{w}))}(w'_j) \nonumber\\
    &\hspace{1cm}= \sum_{\substack{\vb*{a} \in \mathds{N}^{\delta_1}:\\ \sum_{i=1}^{\delta_1} ia_i = \delta_1}}  \frac{\delta_1!}{\prod_{k=1}^{\delta_1} (k!^{a_k} a_k!)} \Bigg(\frac{\delta_1!}{\prod_{k=1}^{\delta_1}k!^{a_k}}\Bigg)^w \sum_{\substack{\vb*{b} \in \mathds{N}^{\delta_2}:\\ \sum_{i=1}^{\delta_2} ib_i = \delta_2}}  \frac{\delta_2!}{\prod_{k=1}^{\delta_2} (k!^{b_k} b_k!)} \Bigg(\frac{\delta_2!}{\prod_{k=1}^{\delta_2}k!^{b_k}}\Bigg)^w \nonumber \\
    &\hspace{2cm} \times \sideset{}{^*}\sum_{\vb*{s} \in V^{h(\vb*{a})+h(\vb*{b})}}    \prod_{i=1}^{\delta_1} \prod_{j=1}^{a_i}   
    f_1^{(i)}(s^{(i)}_{1,j}) \prod_{k=1}^{\delta_2} \prod_{l=1}^{b_k}   
    f_2^{(k)}(s^{(k)}_{2,l}) \nonumber\\
    &\hspace{1cm}=\sum_{\substack{\vb*{a} \in \mathds{N}^{\delta_1}:\\ \sum_{i=1}^{\delta_1} ia_i = \delta_1}}  \frac{\delta_1!^{1+w}}{\prod_{k=1}^{\delta_1} (k!^{(1+w)a_k} a_k!)} \sum_{\substack{\vb*{b} \in \mathds{N}^{\delta_2}:\\ \sum_{i=1}^{\delta_2} ib_i = \delta_2}}  \frac{\delta_2!^{1+w}}{\prod_{k=1}^{\delta_2} (k!^{(1+w)b_k} b_k!)} \nonumber \\
    &\hspace{2cm} \times \sideset{}{^*}\sum_{\vb*{s} \in V^{h(\vb*{a})+h(\vb*{b})}}    \prod_{i=1}^{\delta_1} \prod_{j=1}^{a_i}   
    f_1^{(i)}(s^{(i)}_{1,j}) \prod_{k=1}^{\delta_2} \prod_{l=1}^{b_k}   
    f_2^{(k)}(s^{(k)}_{2,l}).
\end{align}

Now we analyze this last sum of products. The summations range over all vertex lists $\vb*{s} \in V^{h(\vb*{a})+h(\vb*{b})}$ that consist of distinct vertices, i.e., without duplicates. Equivalently, these vertex lists can be described by taking all possible vertex lists and subtracting those that contain repeated vertices. We count duplicates by counting vertices that replicate the value of earlier vertices: for instance, if $s^{(1)}_{1,2} = s^{(1)}_{1,1}$ and $s^{(3)}_{2,2} = s^{(2)}_{1,2}$, then two vertices copy existing values. Importantly, due to the structure of the summation, a vertex can only copy values from earlier vertices—never from later ones. Let $\sum_{\vb*{s} \in V^{h(\vb*{a})+h(\vb*{b})}: x \textnormal{ vertices copy}}$ denote the sum over all vertex lists $\vb*{s} \in V^{h(\vb*{a})+h(\vb*{b})}$ where $x$ vertices take the value of the vertex that they copy and $h(\vb*{a})+h(\vb*{b})-x$ vertices can take any value, including replications of other vertices. Then we obtain
\begin{align*}
    \sideset{}{^*}\sum_{\vb*{s} \in V^{h(\vb*{a})+h(\vb*{b})}} &= \sum_{\substack{\vb*{s} \in V^{h(\vb*{a})+h(\vb*{b})}: \\ 0 \textnormal{ vertices copy}}} - \sum_{\substack{\vb*{s} \in V^{h(\vb*{a})+h(\vb*{b})}: \\ \geq 1 \textnormal{ vertex copies}}}.
\end{align*}
Moreover, for any $x \in \mathds{N}$,
\begin{align*}
    \sum_{\substack{\vb*{s} \in V^{h(\vb*{a})+h(\vb*{b})}: \\ \geq x \textnormal{ vertices copy}}} = \sum_{\substack{\vb*{s} \in V^{h(\vb*{a})+h(\vb*{b})}: \\ x \textnormal{ vertices copy}}} - \sum_{\substack{\vb*{s} \in V^{h(\vb*{a})+h(\vb*{b})}: \\ \geq x+1 \textnormal{ vertices copy}}}.
\end{align*}
Let
\begin{align*}
    A(\vb*{s}) = \prod_{i=1}^{\delta_1} \prod_{j=1}^{a_i}  
    f_1^{(i)}(s^{(i)}_{1,j}) \prod_{k=1}^{\delta_2} \prod_{l=1}^{b_k}  
    f_2^{(k)}(s^{(k)}_{2,l}).
\end{align*}
Then, we obtain
\begin{align}
\label{eq:sum_g_split_2}
    \sideset{}{^*}\sum_{\vb*{s} \in V^{h(\vb*{a})+h(\vb*{b})}} A(\vb*{s}) = \sum_{\substack{\vb*{s} \in V^{h(\vb*{a})+h(\vb*{b})}: \\ 0 \textnormal{ vertices copy}}} A(\vb*{s}) - \sum_{\substack{\vb*{s} \in V^{h(\vb*{a})+h(\vb*{b})}: \\ 1 \textnormal{ vertex copies}}} A(\vb*{s}) + \hdots + (-1)^{h(\vb*{a})+h(\vb*{b})-1} \hspace{-1cm}\sum_{\substack{\vb*{s} \in V^{h(\vb*{a})+h(\vb*{b})}: \\ h(\vb*{a})+h(\vb*{b})-1 \textnormal{ vertices copy}}} \hspace{-1cm} A(\vb*{s}),
\end{align}
with
\begin{align}
\label{eq:sum_g_one_x_2}
    \sum_{\substack{\vb*{s} \in V^{h(\vb*{a})+h(\vb*{b})}: \\ x \textnormal{ vertices copy}}} &= \sum_{\substack{i_1,i_2,\hdots,i_x \in [2,h(\vb*{a})+h(\vb*{b})]: \\ i_1<i_2<\hdots<i_x}} \sum_{\substack{j_1,j_2,\hdots,j_x \in [h(\vb*{a})+h(\vb*{b})-1]: \\ \forall k: j_k < i_k}} \sum_{\substack{\vb*{s} \in V^{h(\vb*{a})+h(\vb*{b})}: \\ \forall k: s_{i_k} = s_{j_k}}},
\end{align}
which denotes that vertex $s_{i_k}$ copies vertices $s_{j_k}$, for all $k \in [x]$. Here, the vertex $s_r$ denotes the $r$'th entry in the list of $\vb*{s}$ values, i.e., 
\begin{align*}
    s_r = 
    \begin{cases}
        s^{(t(r))}_{1,b(r)} & r \leq h(\vb*{a})\\
        s^{(t(r))}_{2,b(r)} & r > h(\vb*{a}),
    \end{cases}
\end{align*}
for $t(r) = \max_{w \geq 1} w: r > \sum_{d=1}^{w-1} a_d$ and $b(r) = r-\sum_{d=1}^{t(r)-1} a_d$ if $r \leq h(\vb*{a})$ and $t(r) = \max_{w \geq 1} w: r > h(\vb*{a}) +  \sum_{d=1}^{w-1} b_d$ and $b(r) = r-h(\vb*{a}) - \sum_{d=1}^{t(r)-1} b_d$ if $r > h(\vb*{a})$.

Next, we analyze what each sum in the form of Equation (\ref{eq:sum_g_one_x_2}) contributes to Equation (\ref{eq:sum_g_split_2}). To that end, let us fix $\vb*{i}=(i_1,i_2,\hdots,i_x)$ and $\vb*{j}=(j_1,j_2,\hdots,j_x)$. Let $\sum_{s \in V \textnormal{ if } t \notin i_1,i_2,\hdots,i_x}$ denote a sum that only appears if $t \notin i_1,i_2,\hdots,i_x$, for some given $t$. Then,
\begin{align*}
     &\sum_{\substack{\vb*{s} \in V^{h(\vb*{a})+h(\vb*{b})}: \\ \forall k: s_{i_k} = s_{j_k}}} A(\vb*{s})\\
    &= \sum_{\substack{s_1 \in V\\\textnormal{if }1 \notin i_1,i_2,\hdots,i_x}}  \sum_{\substack{s_2 \in V\\\textnormal{if }2 \notin i_1,i_2,\hdots,i_x}} \hdots  \sum_{\substack{s_{h(\vb*{a})+h(\vb*{b})} \in V\\\textnormal{if }h(\vb*{a})+h(\vb*{b}) \notin i_1,i_2,\hdots,i_x}}  \prod_{i=1}^{\delta_1}\prod_{r_1=1}^{h(\vb*{a})}  \big(f_1^{(i)}(s_{r_1}) \big)^{g_1^{(i)}(r_1)} \prod_{j=1}^{\delta_2}\prod_{r_2=1}^{h(\vb*{b})}  \big(f_2^{(j)}(s_{r_2}) \big)^{g_2^{(j)}(r_2)},
\end{align*}
where $g_1^{(i)}(r)$ denotes how often the $r$'th vertex appears under $f_1^{(i)}$ after copying:
\begin{align*}
    g_1^{(i)}(r) &= 
    \begin{cases}
        0 & \textnormal{if } r \in i_1,i_2,\hdots,i_x \lor r > h(\vb*{a})\\
        \mathds{1}_{\{t(r) = i\}} + \sum_{k=1}^{x} \mathds{1}_{\{t(i_k) = i \land i_k \rightarrow r \land i_k \leq h(\vb*{a})\}} & \textnormal{else},
    \end{cases}\\
    g_2^{(i)}(r) &= 
    \begin{cases}
        0 & \textnormal{if } r \in i_1,i_2,\hdots,i_x\\
        \mathds{1}_{\{t(r) = i \land r > h(\vb*{a})\}} + \sum_{k=1}^{x} \mathds{1}_{\{t(i_k) = i \land i_k \rightarrow r \land i_k > h(\vb*{a})\}} & \textnormal{else},
    \end{cases}
\end{align*}
where $i_k \rightarrow r$ denotes that vertex $s_{i_k}$ eventually copies vertex $s_r$, i.e., vertex $s_{i_k}$ copies vertex $s_r$ or vertex $s_{i_k}$ copies a vertex that copies vertex $s_r$, etc. Thus, $g_j^{(i)}(r)$ denotes that $f_j^{(i)}(s_r)$ appears $g_j^{(i)}(r)$ times, for $j \in \{1,2\}$. 
 
Since there are $x$ vertices that copy, there are $h(\vb*{a}) + h(\vb*{b})-x$ values of $r$ with $g_j^{(i)}(r) \geq 1$ for some $i,j$. Let $U \in V$ be a uniformly random picked vertex, then,
\begin{align*}
    \sum_{\substack{\vb*{s} \in V^{h(\vb*{a})+h(\vb*{b})}: \\ \forall k: s_{i_k} = s_{j_k}}} A(\vb*{s}) &= \prod_{\substack{r=1\\r \notin \vb*{i}}}^{h(\vb*{a})+h(\vb*{b})} n \mathds{E}\Big[\prod_{i=1}^{\delta_1} f_1^{(i)}(U)^{g_1^{(i)}(r)} \prod_{j=1}^{\delta_2} f_2^{(j)}(U)^{g_2^{(j)}(r)}\Big] \\
    &=n^{h(\vb*{a})+h(\vb*{b})-x} \prod_{r=1}^{h(\vb*{a})+h(\vb*{b})} 
     \mathds{E}\Big[\prod_{i=1}^{\delta_1} (f_1^{(i)}(U))^{g_1^{(i)}(r)} \prod_{j=1}^{\delta_2} (f_2^{(j)}(U))^{g_2^{(j)}(r)}\Big].
\end{align*}
To interpret the tuple $((g_1^{(i)}(r))_{i \in [\delta_1]},(g_2^{(j)}(r))_{j \in [\delta_2]})$, we distinguish the cases $r \leq h(\vb*{a})$ and $r > h(\vb*{a})$. In the first case, let $k$ be the smallest value of $i$ for which $g_1^{(i)}(r)$ has a nonzero value. Then, we know that vertex $r$ represents a vertex with multiplicity $k$. The tuple $((g_1^{(i)}(r))_{i \in [\delta_1]},(g_2^{(j)}(r))_{j \in [\delta_2]})$ indicates that $g_1^{(k)}(r)-1$ vertices with multiplicity $k$ in the first $h(\vb*{a})$ elements of $\vb*{s}$ eventually copy vertex $r$ and that also for all $j>k$: $g_1^{(j)}(r)$ vertices representing sets of size $j$ in the first $h(\vb*{a})$ elements of $\vb*{s}$ copy vertex $r$. In addition, for all $j \in [\delta_2]$, $g_2^{(j)}(r)$ vertices with multiplicity $j$ in the last $h(\vb*{b})$ elements of $\vb*{s}$ eventually copy vertex $r$. If $r > h(\vb*{a})$ then $(g_1^{(i)}(r))_{i \in [\delta_1]} = \vb*{0}$. Let $k$ be the smallest value of $i$ for which $g_2^{(i)}(r)$ has a nonzero value. Then, the tuple $(\vb*{0},(g_2^{(j)}(r))_{j \in [\delta_2]})$ indicates that $g_2^{(k)}(r)-1$ vertices with multiplicity $k$ in the last $h(\vb*{b})$ elements of $\vb*{s}$ eventually copy vertex $r$ and that also for all $j>k$: $g_2^{(j)}(r)$ vertices with multiplicity $j$ in the last $h(\vb*{b})$ elements of $\vb*{s}$ copy vertex $r$.

Now, notice that this expression contains duplicate elements if not all tuples $((g_1^{(i)}(r))_{i \in [\delta_1]},(g_2^{(j)}(r))_{j \in [\delta_2]})$ are unique for all $s_r$. Let $\gamma(\vb*{y},\vb*{z}) = |r \in [h(\vb*{a})+h(\vb*{b})]: ((g_1^{(i)}(r))_{i \in [\delta_1]},(g_2^{(j)}(r))_{j \in [\delta_2]}) = (\vb*{y},\vb*{z})|$ count the number of occurrences of the exponents $((g_1^{(i)}(r))_{i \in [\delta_1]},(g_2^{(j)}(r))_{j \in [\delta_2]})$. Thus, $\gamma(\vb*{y},\vb*{z})$ counts how often the term $\mathds{E}[\prod_{i=1}^{\delta_1} (f_1^{(i)}(U))^{y_i} \prod_{j=1}^{\delta_2} (f_2^{(j)}(U))^{z_j}]$ appears. Then,
\begin{align}
\label{eq:sum_g_given_i_j_2}
    &\sum_{\substack{\vb*{s} \in V^{h(\vb*{a})+h(\vb*{b})}: \\ \forall k: s_{i_k} = s_{j_k}}} A(\vb*{s}) = n^{h(\vb*{a})+h(\vb*{b})-x} \prod_{\substack{\vb*{y} \in \mathds{N}^{\delta_1}\\ \vb*{z} \in \mathds{N}^{\delta_2}}} \mathds{E}\Big[\prod_{i=1}^{\delta_1} (f_1^{(i)}(U))^{y_i} \prod_{j=1}^{\delta_2} (f_2^{(j)}(U))^{z_j}\Big]^{\gamma(\vb*{y},\vb*{z})}.
\end{align}
For every choice of $(\vb*{i},\vb*{j})$, the sum (\ref{eq:sum_g_given_i_j_2}) is characterized by $(\gamma(\vb*{y},\vb*{z}))_{\vb*{y} \in \mathds{N}^{\delta_1}, \vb*{z} \in \mathds{N}^{\delta_2}}$. Moreover, \\$(\gamma(\vb*{y},\vb*{z}))_{\vb*{y} \in \mathds{N}^{\delta_1}, \vb*{z} \in \mathds{N}^{\delta_2}}$ are related to $a_i$, $c_j$ and $x$ as
\begin{align}
    \forall i \in [\delta_1]: a_i &= \sum_{\vb*{y} \in \mathds{N}^{\delta_1}} \sum_{\vb*{z} \in \mathds{N}^{\delta_2}} y_i \gamma(\vb*{y},\vb*{z}) \label{eq:a_i_vs_a_2}\\
    \forall i \in [\delta_2]: b_i &= \sum_{\vb*{y} \in \mathds{N}^{\delta_1}} \sum_{\vb*{z} \in \mathds{N}^{\delta_2}} z_i \gamma(\vb*{y},\vb*{z}) \label{eq:b_i_vs_a_2}\\
    x &= \sum_{\vb*{y} \in \mathds{N}^{\delta_1}} \sum_{\vb*{z} \in \mathds{N}^{\delta_2}} \Big(\sum_{i=1}^{\delta_1} y_i + \sum_{j=1}^{\delta_2} z_j -1\Big) \gamma(\vb*{y},\vb*{z}). \label{eq:x_vs_a_2}
\end{align}
A given list of values $(\gamma(\vb*{y},\vb*{z}))_{\vb*{y} \in \mathds{N}^{\delta_1}, \vb*{z} \in \mathds{N}^{\delta_2}}$ is not unique to a single choice of $(\vb*{i}, \vb*{j})$. For each $(\gamma(\vb*{y},\vb*{z}))_{\vb*{y} \in \mathds{N}^{\delta_1}, \vb*{z} \in \mathds{N}^{\delta_2}}$, we aim to count the number of corresponding $(\vb*{i}, \vb*{j})$ pairs that yield it. Remember that $(\vb*{y},\vb*{z})$ describes which types of vertices copy which types of vertices. Let us refer to the vertex being copied and the vertices that eventually copy it as a vertex group of size $(\vb*{y},\vb*{z})$. To count the number of ways that $(\gamma(\vb*{y},\vb*{z}))_{\vb*{y} \in \mathds{N}^{\delta_1}, \vb*{z} \in \mathds{N}^{\delta_2}}$ groups can be constructed, we first count the number of ways that $\gamma(\vb*{0},\vb*{e}_{\delta_2})$ groups of size $(\vb*{0},\vb*{e}_{\delta_2})$ can be constructed, followed by $\gamma(\vb*{0},2\vb*{e}_{\delta_2})$ groups of size $(\vb*{0},2\vb*{e}_{\delta_2})$, considering every group size and ending at $\gamma(\sum_{k=1}^{\delta_1} \delta_1 \vb*{e}_k, \sum_{k=1}^{\delta_2} \delta_2 \vb*{e}_l)$ groups of size $(\sum_{k=1}^{\delta_1} \delta_1 \vb*{e}_k, \sum_{k=1}^{\delta_2} \delta_2 \vb*{e}_l)$. The number of ways that $\gamma(\vb*{y},\vb*{z})$ vertex groups of size $(\vb*{y},\vb*{z})$ can be constructed, given that $\gamma(\hat{\vb*{y}},\hat{\vb*{z}})$ vertex groups of size $(\hat{\vb*{y}},\hat{\vb*{z}})$ with, for $\hat{\vb*{y}} \neq \vb*{y}$, $\hat{y}_i < y_i$ for $i=\min k: \hat{y}_k \neq y_k$ and, for $\hat{\vb*{y}}=\vb*{y}, \hat{\vb*{z}}\neq \vb*{z}$, $\hat{z}_i < z_i$ for $i=\min k: \hat{z}_k \neq z_k$ are already constructed. Let $\vb*{e}_i$ be the standard unit vector. To construct $\gamma(\vb*{y},\vb*{z})$ vertex groups,
\begin{align*}
 \lambda_{(\vb*{y},\vb*{z})}^{(i)} &=a_i - \sum_{\hat{\vb*{z}} \in \mathds{N}^{\delta_2}}\sum_{f=1}^{\delta_1} \Bigg( \sum_{\hat{y}_f=0}^{y_f-1} \sum_{\hat{y}_{f+1}=0}^{\delta_1} \hdots \sum_{\hat{y}_{\delta_1}=0}^{\delta_1} \Big(\sum_{k=1}^{f-1} y_k \vb*{e}_k+ \sum_{l=f}^{\delta_1} \hat{y}_l \vb*{e}_l\Big)_i \gamma \Big(\sum_{k=1}^{f-1} y_k \vb*{e}_k+ \sum_{l=f}^{\delta_1} \hat{y}_l \vb*{e}_l, \hat{\vb*{z}}\Big) \Bigg) \\
 &\hspace{1cm} - \sum_{f=1}^{\delta_2} \Bigg(\sum_{\hat{z}_f=0}^{z_f-1} \sum_{\hat{z}_{f+1}=0}^{\delta_2} \hdots \sum_{\hat{z}_{\delta_2}=0}^{\delta_2} y_i \gamma \Big(\vb*{y},\sum_{k=1}^{f-1} z_k \vb*{e}_k+ \sum_{l=f}^{\delta_2} \hat{z}_l \vb*{e}_l \Big) \Bigg)
\end{align*}
vertices with index $r \leq h(\vb*{a})$ such that $t(r)=i$ can be used, as they are not part of a previously constructed vertex group. The first block of summations counts all $(\hat{\vb*{y}},\hat{\vb*{z}})$ for which $\hat{\vb*{y}} \neq \vb*{y}$ and $\hat{y}_i < y_i$ for $i=\min k: \hat{y}_k \neq y_k$, where we represent $i$ with $f$ in the summation. The second block of summations counts all $(\hat{\vb*{y}},\hat{\vb*{z}})$ for which $\hat{\vb*{y}}=\vb*{y}$ and $\hat{\vb*{z}}\neq \vb*{z}$, $\hat{z}_i < z_i$ for $i=\min k: \hat{z}_k \neq z_k$. In addition,
\begin{align*}
 \mu_{(\vb*{y},\vb*{z})}^{(i)} &=b_i - \sum_{\hat{\vb*{z}} \in \mathds{N}^{\delta_2}}\sum_{f=1}^{\delta_1} \Bigg( \sum_{\hat{y}_f=0}^{y_f-1} \sum_{\hat{y}_{f+1}=0}^{\delta_1} \hdots \sum_{\hat{y}_{\delta_1}=0}^{\delta_1} \hat{z}_i \gamma \Big(\sum_{k=1}^{f-1} y_k \vb*{e}_k+ \sum_{l=f}^{\delta_1} \hat{y}_l \vb*{e}_l, \hat{\vb*{z}} \Big) \Bigg)\\
 &\hspace{1cm} - \sum_{f=1}^{\delta_2} \Bigg( \sum_{\hat{z}_f=0}^{z_f-1} \sum_{\hat{z}_{f+1}=0}^{\delta_2} \hdots \sum_{\hat{z}_{\delta_2}=0}^{\delta_2} \Big(\sum_{k=1}^{f-1} z_k \vb*{e}_k+ \sum_{l=f}^{\delta_2} \hat{z}_l \vb*{e}_l \Big)_i \gamma \Big(\vb*{y},\sum_{k=1}^{f-1} z_k \vb*{e}_k+ \sum_{l=f}^{\delta_2} \hat{z}_l \vb*{e}_l \Big) \Bigg)
\end{align*}
vertices with index $r > h(\vb*{a})$ such that $t(r)=i$ can be used, as they are not part of a previously constructed vertex group. Generating $\gamma(\vb*{y},\vb*{z})$ sets of $\sum_{i=1}^{\delta_1}y_i$ vertices with index $r \leq h(\vb*{a})$ and $\gamma(\vb*{y},\vb*{z})$ sets of $\sum_{i=1}^{\delta_2}z_i$ vertices with index $r > h(\vb*{a})$ out of those vertices can be done in
\begin{align*}
    &\frac{1}{\gamma(\vb*{y},\vb*{z})!} 
    \prod_{i=1}^{\gamma(\vb*{y},\vb*{z})} \Bigg( \prod_{j=1}^{\delta_1}  \binom{\lambda_{(\vb*{y},\vb*{z})}^{(j)}-y_j(i-1)}{y_j} \prod_{k=1}^{\delta_2}  \binom{\mu_{(\vb*{y},\vb*{z})}^{(k)}-z_k(i-1)}{z_k} \Bigg)\\
    &= \frac{\prod_{j=1}^{\delta_1}\lambda_{(\vb*{y},\vb*{z})}^{(j)}! \prod_{k=1}^{\delta_2}\mu_{(\vb*{y},\vb*{z})}^{(k)}!}{\gamma(\vb*{y},\vb*{z})!(\prod_{j=1}^{\delta_1} y_j! \prod_{k=1}^{\delta_2} z_k!)^{\gamma(\vb*{y},\vb*{z})} \prod_{j=1}^{\delta_1} (\lambda_{(\vb*{y},\vb*{z})}^{(j)} -y_j\gamma(\vb*{y},\vb*{z}))! \prod_{k=1}^{\delta_2} (\mu_{(\vb*{y},\vb*{z})}^{(k)} -z_k \gamma(\vb*{y},\vb*{z}))!}
\end{align*}
different ways. The set of $\sum_{i=1}^{\delta_1} y_i + \sum_{j=1}^{\delta_2} z_j$ vertices can be put in order of increasing vertex index $(v_{p_1},v_{p_2},\hdots,v_{p_{\delta_1+\delta_2}})$, with $p_i < p_j$ when $i<j$. vertex $v_{p_1}$ must be the vertex that does not copy any other vertex, since a vertex can only copy a vertex with a smaller index than its own. Moreover, for $i>1$, vertex $v_{p_i}$ copies one of the vertices $(v_{p_1},v_{p_2},\hdots,v_{p_{i-1}})$. Therefore, there are $(\sum_{i=1}^{\delta_1} y_i + \sum_{j=1}^{\delta_2} z_j - 1)!$ ways to let the vertices $v_{p_2},\hdots,v_{p_{\delta_1+\delta_2}}$ copy previously chosen vertices. Thus, there are
\begin{align}
\label{eq:nr_ways_to_choose_a_groups_2}
    \frac{\prod_{j=1}^{\delta_1}\lambda_{(\vb*{y},\vb*{z})}^{(j)}! \prod_{k=1}^{\delta_2}\mu_{(\vb*{y},\vb*{z})}^{(k)}! \big(\sum_{j=1}^{\delta_1} y_j + \sum_{k=1}^{\delta_2} z_k-1\big)!^{\gamma(\vb*{y},\vb*{z})}}{\gamma(\vb*{y},\vb*{z})!\big(\prod_{j=1}^{\delta_1} y_j! \prod_{k=1}^{\delta_2} z_k! \big)^{\gamma(\vb*{y},\vb*{z})} \prod_{j=1}^{\delta_1} (\lambda_{(\vb*{y},\vb*{z})}^{(j)} -y_j\gamma(\vb*{y},\vb*{z}))! \prod_{k=1}^{\delta_2} (\mu_{(\vb*{y},\vb*{z})}^{(k)} -z_k \gamma(\vb*{y},\vb*{z}))!}
\end{align}
ways to choose $\gamma(\vb*{y},\vb*{z})$ groups of size $(\vb*{y},\vb*{z})$ and the vertices that they copy (if any), given that previous groups where already chosen. We will make use of the following claims to simplify this expression. The proofs of these claims follow after this proof.

\begin{claim}
\label{claim:gamma_simplification_1} 
    $\lambda_{(\vb*{y},\vb*{z})}^{(i)} -y_i\gamma(\vb*{y},\vb*{z}) =  \lambda_{(\vb*{y}, \vb*{z}+\vb*{e}_{\delta_2})}^{(i)}$.
\end{claim}

\begin{claim}
\label{claim:delta_simplification_1}
    $\mu_{(\vb*{y},\vb*{z})}^{(i)} -z_i \gamma(\vb*{y},\vb*{z}) = \mu_{(\vb*{y}, \vb*{z}+\vb*{e}_{\delta_2})}^{(i)}$.
\end{claim}

Using Claims \ref{claim:gamma_simplification_1} and \ref{claim:delta_simplification_1}, Equation (\ref{eq:nr_ways_to_choose_a_groups_2}) equals
\begin{align*}
    \frac{\prod_{j=1}^{\delta_1}\lambda_{(\vb*{y},\vb*{z})}^{(j)}! \prod_{k=1}^{\delta_2}\mu_{(\vb*{y},\vb*{z})}^{(k)}! \big(\sum_{j=1}^{\delta_1} y_j + \sum_{k=1}^{\delta_2} z_k-1 \big)!^{\gamma(\vb*{y},\vb*{z})}}{\gamma(\vb*{y},\vb*{z})! \big(\prod_{j=1}^{\delta_1} y_j! \prod_{k=1}^{\delta_2} z_k! \big)^{\gamma(\vb*{y},\vb*{z})} \prod_{j=1}^{\delta_1} \lambda_{(\vb*{y},\vb*{z}+\vb*{e}_{\delta_2})}^{(j)}! \prod_{k=1}^{\delta_2} \mu_{(\vb*{y},\vb*{z} + \vb*{e}_{\delta_2})}^{(k)}!}.
\end{align*}
Combining over all $\vb*{y}$ and $\vb*{z}$, we obtain
\begin{align*}
    &\textnormal{\#}(\vb*{i},\vb*{j}) \textnormal{ that yield } (\gamma(\vb*{y},\vb*{z}))_{\vb*{y} \in \mathds{N}^{\delta_1}, \vb*{z} \in \mathds{N}^{\delta_2}}\\
    &= \prod_{\substack{\vb*{y} \in \mathds{N}^{\delta_1}\\ \vb*{z} \in \mathds{N}^{\delta_2}}}\frac{\prod_{j=1}^{\delta_1}\lambda_{(\vb*{y},\vb*{z})}^{(j)}! \prod_{k=1}^{\delta_2}\mu_{(\vb*{y},\vb*{z})}^{(k)}! \big(\sum_{j=1}^{\delta_1} y_j + \sum_{k=1}^{\delta_2} z_k-1 \big)!^{\gamma(\vb*{y},\vb*{z})}}{\gamma(\vb*{y},\vb*{z})! \big(\prod_{j=1}^{\delta_1} y_j! \prod_{k=1}^{\delta_2} z_k! \big)^{\gamma(\vb*{y},\vb*{z})} \prod_{j=1}^{\delta_1} \lambda_{(\vb*{y},\vb*{z}+\vb*{e}_{\delta_2})}^{(j)}! \prod_{k=1}^{\delta_2} \mu_{(\vb*{y},\vb*{z} + \vb*{e}_{\delta_2})}^{(k)}!}\\
    &= \prod_{\substack{\vb*{y} \in \mathds{N}^{\delta_1}\\ \vb*{z} \in \mathds{N}^{\delta_2}}} \frac{\big(\sum_{j=1}^{\delta_1} y_j + \sum_{k=1}^{\delta_2} z_k-1 \big)!^{\gamma(\vb*{y},\vb*{z})}}{\gamma(\vb*{y},\vb*{z})!\big(\prod_{j=1}^{\delta_1} y_j! \prod_{k=1}^{\delta_2} z_k! \big)^{\gamma(\vb*{y},\vb*{z})}} \prod_{\substack{\vb*{y} \in \mathds{N}^{\delta_1}\\ \vb*{z} \in \mathds{N}^{\delta_2}}} \frac{\prod_{j=1}^{\delta_1}\lambda_{(\vb*{y},\vb*{z})}^{(j)}! \prod_{k=1}^{\delta_2}\mu_{(\vb*{y},\vb*{z})}^{(k)}!}{\prod_{j=1}^{\delta_1} \lambda_{(\vb*{y},\vb*{z}+\vb*{e}_{\delta_2})}^{(j)}! \prod_{k=1}^{\delta_2} \mu_{(\vb*{y},\vb*{z} + \vb*{e}_{\delta_2})}^{(k)}!}\\
    &= \prod_{\substack{\vb*{y} \in \mathds{N}^{\delta_1}\\ \vb*{z} \in \mathds{N}^{\delta_2}}} \frac{\big(\sum_{j=1}^{\delta_1} y_j + \sum_{k=1}^{\delta_2} z_k-1 \big)!^{\gamma(\vb*{y},\vb*{z})}}{\gamma(\vb*{y},\vb*{z})!\big(\prod_{j=1}^{\delta_1} y_j! \prod_{k=1}^{\delta_2} z_k! \big)^{\gamma(\vb*{y},\vb*{z})}} \prod_{\substack{\vb*{y} \in \mathds{N}^{\delta_1}\\ \vb*{z} \in \mathds{N}^{\delta_2-1}}} \frac{\prod_{j=1}^{\delta_1}\lambda_{(\vb*{y},\vb*{z})}^{(j)}! \prod_{k=1}^{\delta_2}\mu_{(\vb*{y},\vb*{z})}^{(k)}!}{\prod_{j=1}^{\delta_1} \lambda_{(\vb*{y},\vb*{z}+(\delta_2+1)\vb*{e}_{\delta_2})}^{(j)}! \prod_{k=1}^{\delta_2} \mu_{(\vb*{y},\vb*{z} + (\delta_2+1)\vb*{e}_{\delta_2})}^{(k)}!},
\end{align*} 
where we use the telescoping product in the final equality. Again, we make use of two claims to simplify this expression. The proofs of these claims follow after this proof.

\begin{claim}
    \label{claim:simplification_gamma_2}
    If $z_{\delta_2}=0$ then $\lambda_{(\vb*{y}, \vb*{z} + (\delta_2+1)\vb*{e}_{\delta_2})}^{(i)} = \lambda_{(\vb*{y}, \vb*{z} + \vb*{e}_{\delta_2-1})}^{(i)}$.
\end{claim}

\begin{claim}
    \label{claim:simplification_delta_2}
    If $z_{\delta_2}=0$ then $\mu_{(\vb*{y}, \vb*{z} + (\delta_2+1)\vb*{e}_{\delta_2})}^{(i)} = \mu_{(\vb*{y}, \vb*{z} + \vb*{e}_{\delta_2-1})}^{(i)}$.
\end{claim}

By Claims \ref{claim:simplification_gamma_2} and \ref{claim:simplification_delta_2}, we obtain
\begin{align*}
    &\textnormal{\#}(\vb*{i},\vb*{j}) \textnormal{ that yield } (\gamma(\vb*{y},\vb*{z}))_{\vb*{y} \in \mathds{N}^{\delta_1}, \vb*{z} \in \mathds{N}^{\delta_2}}\\
    &= \prod_{\substack{\vb*{y} \in \mathds{N}^{\delta_1}\\ \vb*{z} \in \mathds{N}^{\delta_2}}} \frac{\big(\sum_{j=1}^{\delta_1} y_j + \sum_{k=1}^{\delta_2} z_k-1 \big)!^{\gamma(\vb*{y},\vb*{z})}}{\gamma(\vb*{y},\vb*{z})! \big(\prod_{j=1}^{\delta_1} y_j! \prod_{k=1}^{\delta_2} z_k! \big)^{\gamma(\vb*{y},\vb*{z})}}  \prod_{\substack{\vb*{y} \in \mathds{N}^{\delta_1}\\ \vb*{z} \in \mathds{N}^{\delta_2-1}}} \frac{\prod_{j=1}^{\delta_1}\lambda_{(\vb*{y},\vb*{z})}^{(j)}! \prod_{k=1}^{\delta_2}\mu_{(\vb*{y},\vb*{z})}^{(k)}!}{\prod_{j=1}^{\delta_1} \lambda_{(\vb*{y},\vb*{z}+\vb*{e}_{\delta_2-1})}^{(j)}! \prod_{k=1}^{\delta_2} \mu_{(\vb*{y},\vb*{z} + \vb*{e}_{\delta_2-1})}^{(k)}!}\\
    &= \prod_{\substack{\vb*{y} \in \mathds{N}^{\delta_1}\\ \vb*{z} \in \mathds{N}^{\delta_2}}} \frac{\big(\sum_{j=1}^{\delta_1} y_j + \sum_{k=1}^{\delta_2} z_k-1 \big)!^{\gamma(\vb*{y},\vb*{z})}}{\gamma(\vb*{y},\vb*{z})!\big(\prod_{j=1}^{\delta_1} y_j! \prod_{k=1}^{\delta_2} z_k! \big)^{\gamma(\vb*{y},\vb*{z})}}  \prod_{\vb*{y} \in \mathds{N}^{\delta_1}} \frac{\prod_{j=1}^{\delta_1} \lambda_{(\vb*{y},\vb*{0})}^{(j)}! \prod_{k=1}^{\delta_2}\mu_{(\vb*{y},\vb*{0})}^{(k)}!}{\prod_{j=1}^{\delta_1} \lambda_{(\vb*{y},(\delta_2+1) \vb*{e}_1)}^{(j)}! \prod_{k=1}^{\delta_2} \mu_{(\vb*{y},(\delta_2+1) \vb*{e}_1)}^{(k)}!},
\end{align*}
where we repeated the telescoping product for the second equality. We again use two claims to simplify this expression. Their proofs follow after this proof.

\begin{claim}
    \label{claim:simpliciation_gamma_3}
    $\lambda_{(\vb*{y}, (\delta_2+1)\vb*{e}_1)}^{(i)} = \lambda_{(\vb*{y} + \vb*{e}_{\delta_1}, \vb*{0})}^{(i)}$.
\end{claim}

\begin{claim}
    \label{claim:simplification_delta_3}
    $\mu_{(\vb*{y}, (\delta_2+1)\vb*{e}_1)}^{(i)} = \mu_{(\vb*{y} + \vb*{e}_{\delta_1}, \vb*{0})}^{(i)}$.
\end{claim}

By Claims \ref{claim:simpliciation_gamma_3} and \ref{claim:simplification_delta_3}, we obtain
\begin{align*}
    &\textnormal{\#}(\vb*{i},\vb*{j}) \textnormal{ that yield } (\gamma(\vb*{y},\vb*{z}))_{\vb*{y} \in \mathds{N}^{\delta_1}, \vb*{z} \in \mathds{N}^{\delta_2}}\\
    &= \prod_{\substack{\vb*{y} \in \mathds{N}^{\delta_1}\\ \vb*{z} \in \mathds{N}^{\delta_2}}} \frac{\big(\sum_{j=1}^{\delta_1} y_j + \sum_{k=1}^{\delta_2} z_k-1 \big)!^{\gamma(\vb*{y},\vb*{z})}}{\gamma(\vb*{y},\vb*{z})!\big(\prod_{j=1}^{\delta_1} y_j! \prod_{k=1}^{\delta_2} z_k! \big)^{\gamma(\vb*{y},\vb*{z})}} \frac{\prod_{j=1}^{\delta_1}\lambda_{(\vb*{0},\vb*{0})}^{(j)}! \prod_{k=1}^{\delta_2}\mu_{(\vb*{0},\vb*{0})}^{(k)}!}{\prod_{j=1}^{\delta_1} \lambda_{((\delta_1+1) \vb*{e}_1, \vb*{0})}^{(j)}! \prod_{k=1}^{\delta_2} \mu_{((\delta_1+1) \vb*{e}_1, \vb*{0})}^{(k)}!},
\end{align*}
by the telescoping product. Now,
\begin{align*}
    \lambda_{(\vb*{0},\vb*{0})}^{(i)} &= a_i\\
    \mu_{(\vb*{0},\vb*{0})}^{(i)} &= b_i\\
    \lambda_{((\delta_1+1)\vb*{e},\vb*{0})}^{(i)} &= 0\\
    \mu_{((\delta_1+1)\vb*{e},\vb*{0})}^{(i)} &= 0.
\end{align*}
Therefore,
\begin{align*}
    &\textnormal{\#}(\vb*{i},\vb*{j}) \textnormal{ that yield } (\gamma(\vb*{y},\vb*{z}))_{\vb*{y} \in \mathds{N}^{\delta_1}, \vb*{z} \in \mathds{N}^{\delta_2}} = \prod_{\substack{\vb*{y} \in \mathds{N}^{\delta_1}\\ \vb*{z} \in \mathds{N}^{\delta_2}}} \frac{\big(\sum_{j=1}^{\delta_1} y_j + \sum_{k=1}^{\delta_2} z_k-1 \big)!^{\gamma(\vb*{y},\vb*{z})}}{\gamma(\vb*{y},\vb*{z})!\big(\prod_{j=1}^{\delta_1} y_j! \prod_{k=1}^{\delta_2} z_k! \big)^{\gamma(\vb*{y},\vb*{z})}} \prod_{j=1}^{\delta_1}a_j! \prod_{k=1}^{\delta_2}b_k!.
\end{align*}

Using Equations (\ref{eq:a_i_vs_a_2}), (\ref{eq:b_i_vs_a_2}) and (\ref{eq:x_vs_a_2}), we conclude
\begin{align}
\label{eq:sum_g_2}
    &\sum_{\substack{\vb*{s} \in V^{h(\vb*{a}) + h(\vb*{b})}: \\ x \textnormal{ vertices copy}}} A(\vb*{s}) \nonumber\\
    &= \sum_{\substack{(\gamma(\vb*{y},\vb*{z}))_{\vb*{y} \in \mathds{N}^{\delta_1}, \vb*{z} \in \mathds{N}^{\delta_2}}:\\  \forall i \in [\delta_1]: a_i = \sum_{\vb*{y} \in \mathds{N}^{\delta_1}} \sum_{\vb*{z} \in \mathds{N}^{\delta_2}} y_i \gamma(\vb*{y},\vb*{z}) \\
    \forall i \in [\delta_2]: b_i = \sum_{\vb*{y} \in \mathds{N}^{\delta_1}} \sum_{\vb*{z} \in \mathds{N}^{\delta_2}} z_i \gamma(\vb*{y},\vb*{z}) \\
    x = \sum_{\vb*{y} \in \mathds{N}^{\delta_1}} \sum_{\vb*{z} \in \mathds{N}^{\delta_2}} (\sum_{i=1}^{\delta_1} y_i + \sum_{j=1}^{\delta_2} z_i -1) \gamma(\vb*{y},\vb*{z}).}} \prod_{\substack{\vb*{y} \in \mathds{N}^{\delta_1}\\ \vb*{z} \in \mathds{N}^{\delta_2}}} \frac{\big(\sum_{j=1}^{\delta_1} y_j + \sum_{k=1}^{\delta_2} z_k-1 \big)!^{\gamma(\vb*{y},\vb*{z})}}{\gamma(\vb*{y},\vb*{z})!\big(\prod_{j=1}^{\delta_1} y_j! \prod_{k=1}^{\delta_2} z_k! \big)^{\gamma(\vb*{y},\vb*{z})}} \prod_{j=1}^{\delta_1}a_j! \prod_{k=1}^{\delta_2}b_k! \nonumber\\
    &\hspace{1cm} \times n^{h(\vb*{a})+h(\vb*{b})-x} \prod_{\substack{\vb*{y} \in \mathds{N}^{\delta_1}\\ \vb*{z} \in \mathds{N}^{\delta_2}}} \mathds{E}\Big[\prod_{i=1}^{\delta_1} (f_1^{(i)}(U))^{y_i} \prod_{j=1}^{\delta_2} (f_2^{(j)}(U))^{z_j} \Big]^{\gamma(\vb*{y},\vb*{z})} \nonumber\\
    &= \sum_{\substack{(\gamma(\vb*{y},\vb*{z}))_{\vb*{y} \in \mathds{N}^{\delta_1}, \vb*{z} \in \mathds{N}^{\delta_2}}:\\  \forall i \in [\delta_1]: a_i = \sum_{\vb*{y} \in \mathds{N}^{\delta_1}} \sum_{\vb*{z} \in \mathds{N}^{\delta_2}} y_i \gamma(\vb*{y},\vb*{z}) \\
    \forall i \in [\delta_2]: b_i = \sum_{\vb*{y} \in \mathds{N}^{\delta_1}} \sum_{\vb*{z} \in \mathds{N}^{\delta_2}} z_i \gamma(\vb*{y},\vb*{z}) \\
    x = \sum_{\vb*{y} \in \mathds{N}^{\delta_1}} \sum_{\vb*{z} \in \mathds{N}^{\delta_2}} (\sum_{i=1}^{\delta_1} y_i + \sum_{j=1}^{\delta_2} z_i -1) \gamma(\vb*{y},\vb*{z}).}} \prod_{j=1}^{\delta_1}a_j! \prod_{k=1}^{\delta_2}b_k! \nonumber \\
    &\hspace{1cm} \times \prod_{\substack{\vb*{y} \in \mathds{N}^{\delta_1}\\ \vb*{z} \in \mathds{N}^{\delta_2}}} \Bigg( \frac{1}{\gamma(\vb*{y},\vb*{z})!} \Big(\frac{n\mathds{E}\big[\prod_{i=1}^{\delta_1} (f_1^{(i)}(U))^{y_i} \prod_{j=1}^{\delta_2} (f_2^{(j)}(U))^{z_j} \big]\big(\sum_{j=1}^{\delta_1} y_j + \sum_{k=1}^{\delta_2} z_k-1 \big)!}{\prod_{j=1}^{\delta_1} y_j! \prod_{k=1}^{\delta_2} z_k!} \Big)^{\gamma(\vb*{y},\vb*{z})} \Bigg).
\end{align}
Combining Equation (\ref{eq:sum_g_2}) with Equation (\ref{eq:sum_g_split_2}) gives
\begin{align}
\label{eq:sum_g_to_sum_a_2}
    & \sideset{}{^*}\sum_{\vb*{s} \in V^{h(\vb*{a})+h(\vb*{b})}}  A(\vb*{s}) \nonumber\\
    &= \sum_{\substack{(\gamma(\vb*{y},\vb*{z}))_{\vb*{y} \in \mathds{N}^{\delta_1}, \vb*{z} \in \mathds{N}^{\delta_2}}:\\  \forall i \in [\delta_1]: a_i = \sum_{\vb*{y} \in \mathds{N}^{\delta_1}} \sum_{\vb*{z} \in \mathds{N}^{\delta_2}} y_i \gamma(\vb*{y},\vb*{z}) \\
    \forall i \in [\delta_2]: b_i = \sum_{\vb*{y} \in \mathds{N}^{\delta_1}} \sum_{\vb*{z} \in \mathds{N}^{\delta_2}} z_i \gamma(\vb*{y},\vb*{z})}} \prod_{j=1}^{\delta_1}a_j! \prod_{k=1}^{\delta_2}b_k!  \hspace{0.3cm}(-1)^{\sum_{\vb*{y} \in \mathds{N}^{\delta_1}} \sum_{\vb*{z} \in \mathds{N}^{\delta_2}} (\sum_{i=1}^{\delta_1} y_i + \sum_{j=1}^{\delta_2} z_i -1) \gamma(\vb*{y},\vb*{z})} \nonumber \\
    &\hspace{1cm}  \times \prod_{\substack{\vb*{y} \in \mathds{N}^{\delta_1}\\ \vb*{z} \in \mathds{N}^{\delta_2}}} \Bigg( \frac{1}{\gamma(\vb*{y},\vb*{z})!} \Big(\frac{n\mathds{E}\big[\prod_{i=1}^{\delta_1} (f_1^{(i)}(U))^{y_i} \prod_{j=1}^{\delta_2} (f_2^{(j)}(U))^{z_j} \big]\big(\sum_{j=1}^{\delta_1} y_j + \sum_{k=1}^{\delta_2} z_k-1 \big)!}{\prod_{j=1}^{\delta_1} y_j! \prod_{k=1}^{\delta_2} z_k!} \Big)^{\gamma(\vb*{y},\vb*{z})} \Bigg).
\end{align}

Combining Equation (\ref{eq:sum_g_to_sum_a_2}) with Equation (\ref{eq:sum_v_to_sum_a_2}) yields the result.
\end{proof}

\begin{proof}[Proof of Claim \ref{claim:gamma_simplification_1}]
    \begin{align*}
    \lambda_{(\vb*{y},\vb*{z})}^{(i)} -y_i\gamma(\vb*{y},\vb*{z})
    &=  a_i - \sum_{\hat{\vb*{z}} \in \mathds{N}^{\delta_2}}\sum_{f=1}^{\delta_1} \Bigg( \sum_{\hat{y}_f=0}^{y_f-1} \sum_{\hat{y}_{f+1}=0}^{\delta_1} \hdots \sum_{\hat{y}_{\delta_1}=0}^{\delta_1} \Big(\sum_{k=1}^{f-1} y_k \vb*{e}_k+ \sum_{l=f}^{\delta_1} \hat{y}_l \vb*{e}_l \Big)_i \gamma \Big(\sum_{k=1}^{f-1} y_k \vb*{e}_k+ \sum_{l=f}^{\delta_1} \hat{y}_l \vb*{e}_l, \hat{\vb*{z}} \Big) \Bigg)\\
 &\hspace{1cm} - \sum_{f=1}^{\delta_2-1} \Bigg( \sum_{\hat{z}_f=0}^{z_f-1} \sum_{\hat{z}_{f+1}=0}^{\delta_2} \hdots \sum_{\hat{z}_{\delta_2}=0}^{\delta_2} y_i \gamma \Big(\vb*{y},\sum_{k=1}^{f-1} z_k \vb*{e}_k+ \sum_{l=f}^{\delta_2} \hat{z}_l \vb*{e}_l \Big) \Bigg) \\
 &\hspace{1cm} - \sum_{\hat{z}_{\delta_2} = 0}^{z_{\delta_2}-1} y_i \gamma \Big(\vb*{y},\sum_{k=1}^{\delta_2-1} z_k \vb*{e}_k+ \hat{z}_{\delta_2} \vb*{e}_{\delta_2} \Big) - \Big(\sum_{k=1}^{\delta_1} y_k \vb*{e}_k \Big)_i\gamma\Big(\sum_{k=1}^{\delta_1} y_k \vb*{e}_k, \vb*{z} \Big) \\
    &= a_i - \sum_{\hat{\vb*{z}} \in \mathds{N}^{\delta_2}}\sum_{f=1}^{\delta_1} \Bigg( \sum_{\hat{y}_f=0}^{y_f-1} \sum_{\hat{y}_{f+1}=0}^{\delta_1} \hdots \sum_{\hat{y}_{\delta_1}=0}^{\delta_1} \Big(\sum_{k=1}^{f-1} y_k \vb*{e}_k+ \sum_{l=f}^{\delta_1} \hat{y}_l \vb*{e}_l \Big)_i \gamma \Big(\sum_{k=1}^{f-1} y_k \vb*{e}_k+ \sum_{l=f}^{\delta_1} \hat{y}_l \vb*{e}_l, \hat{\vb*{z}} \Big) \Bigg)\\
 &\hspace{1cm} - \sum_{f=1}^{\delta_2-1} \Bigg( \sum_{\hat{z}_f=0}^{z_f-1} \sum_{\hat{z}_{f+1}=0}^{\delta_2} \hdots \sum_{\hat{z}_{\delta_2}=0}^{\delta_2} y_i \gamma \Big(\vb*{y},\sum_{k=1}^{f-1} z_k \vb*{e}_k+ \sum_{l=f}^{\delta_2} \hat{z}_l \vb*{e}_l \Big) \Bigg) \\
 &\hspace{1cm} - \sum_{\hat{z}_{\delta_2}=0}^{z_{\delta_2}} y_i \gamma \Big(\vb*{y},\sum_{k=1}^{\delta_2-1} z_k \vb*{e}_k + \hat{z}_{\delta_2} \vb*{e}_{\delta_2} \Big)  \\
    &= \lambda_{(\vb*{y}, \vb*{z}+\vb*{e}_{\delta_2})}^{(i)}.
\end{align*}
\end{proof}

\begin{proof}[Proof of Claim \ref{claim:delta_simplification_1}]
    \begin{align*}
    \mu_{(\vb*{y},\vb*{z})}^{(i)} -z_i \gamma(\vb*{y},\vb*{z}) 
    &=  b_i - \sum_{\hat{\vb*{z}} \in \mathds{N}^{\delta_2}}\sum_{f=1}^{\delta_1} \Bigg( \sum_{\hat{y}_f=0}^{y_f-1} \sum_{\hat{y}_{f+1}=0}^{\delta_1} \hdots \sum_{\hat{y}_{\delta_1}=0}^{\delta_1} \hat{z}_i \gamma \Big(\sum_{k=1}^{f-1} y_k \vb*{e}_k+ \sum_{l=f}^{\delta_1} \hat{y}_l \vb*{e}_l, \hat{\vb*{z}} \Big) \Bigg)\\
 &\hspace{1cm} - \sum_{f=1}^{\delta_2-1} \Bigg( \sum_{\hat{z}_f=0}^{z_f-1} \sum_{\hat{z}_{f+1}=0}^{\delta_2} \hdots \sum_{\hat{z}_{\delta_2}=0}^{\delta_2} \Big(\sum_{k=1}^{f-1} z_k\vb*{e}_k + \sum_{l=f}^{\delta_2} \hat{z}_l\vb*{e}_l \Big)_i \gamma \Big(\vb*{y},\sum_{k=1}^{f-1} z_k \vb*{e}_k+ \sum_{l=f}^{\delta_2} \hat{z}_l \vb*{e}_l \Big) \Bigg) \\
 &\hspace{1cm} - \sum_{\hat{z}_{\delta_2} = 0}^{z_{\delta_2}-1} \Big(\sum_{k=1}^{\delta_2-1} z_k\vb*{e}_k + \hat{z}_{\delta_2}\vb*{e}_{\delta_2} \Big)_i \gamma \Big(\vb*{y},\sum_{k=1}^{\delta_2-1} z_k\vb*{e}_k + \hat{z}_{\delta_2}\vb*{e}_{\delta_2} \Big) - \Big(\sum_{k=1}^{\delta_2} z_k \vb*{e}_k \Big)_i\gamma\Big(\sum_{k=1}^{\delta_1} y_k \vb*{e}_k, \vb*{z} \Big) \\
    &= b_i - \sum_{\hat{\vb*{z}} \in \mathds{N}^{\delta_2}}\sum_{f=1}^{\delta_1} \Bigg( \sum_{\hat{y}_f=0}^{y_f-1} \sum_{\hat{y}_{f+1}=0}^{\delta_1} \hdots \sum_{\hat{y}_{\delta_1}=0}^{\delta_1} \hat{z}_i \gamma \Big(\sum_{k=1}^{f-1} y_k \vb*{e}_k+ \sum_{l=f}^{\delta_1} \hat{y}_l \vb*{e}_l, \hat{\vb*{z}} \Big) \Bigg)\\
 &\hspace{1cm} - \sum_{f=1}^{\delta_2-1} \Bigg( \sum_{\hat{z}_f=0}^{z_f-1} \sum_{\hat{z}_{f+1}=0}^{\delta_2} \hdots \sum_{\hat{z}_{\delta_2}=0}^{\delta_2} \Big(\sum_{k=1}^{f-1} z_k\vb*{e}_k + \sum_{l=f}^{\delta_2} \hat{z}_l\vb*{e}_l \Big)_i \gamma \Big(\vb*{y},\sum_{k=1}^{f-1} z_k \vb*{e}_k+ \sum_{l=f}^{\delta_2} \hat{z}_l \vb*{e}_l \Big) \Bigg) \\
 &\hspace{1cm} - \sum_{\hat{z}_{\delta_2}=0}^{z_{\delta_2}} \Big(\sum_{k=1}^{\delta_2} z_k \vb*{e}_k + \hat{z}_{\delta_2} \vb*{e}_{\delta_2} \Big)_i \gamma \Big(\vb*{y},\sum_{k=1}^{\delta_2-1} z_k \vb*{e}_k + \hat{z}_{\delta_2} \vb*{e}_{\delta_2} \Big)  \\
    &= \mu_{(\vb*{y}, \vb*{z}+\vb*{e}_{\delta_2})}^{(i)}.
\end{align*}
\end{proof}

\begin{proof}[Proof of Claim \ref{claim:simplification_gamma_2}]
    \begin{align*}
    \lambda_{(\vb*{y}, \vb*{z} + (\delta_2+1)\vb*{e}_{\delta_2})}^{(i)} &= a_i - \sum_{\hat{\vb*{z}} \in \mathds{N}^{\delta_2}}\sum_{f=1}^{\delta_1} \Bigg( \sum_{\hat{y}_f=0}^{y_f-1} \sum_{\hat{y}_{f+1}=0}^{\delta_1} \hdots \sum_{\hat{y}_{\delta_1}=0}^{\delta_1} \Big(\sum_{k=1}^{f-1} y_k \vb*{e}_k+ \sum_{l=f}^{\delta_1} \hat{y}_l \vb*{e}_l \Big)_i \gamma \Big(\sum_{k=1}^{f-1} y_k \vb*{e}_k+ \sum_{l=f}^{\delta_1} \hat{y}_l \vb*{e}_l, \hat{\vb*{z}} \Big) \Bigg)\\
 &\hspace{1cm} - \sum_{f=1}^{\delta_2-2} \Bigg( \sum_{\hat{z}_f=0}^{z_f-1} \sum_{\hat{z}_{f+1}=0}^{\delta_2} \hdots \sum_{\hat{z}_{\delta_2}=0}^{\delta_2} y_i \gamma \Big(\vb*{y},\sum_{k=1}^{f-1} z_k \vb*{e}_k+ \sum_{l=f}^{\delta_2} \hat{z}_l \vb*{e}_l \Big) \Bigg) \\
    &\hspace{1cm} - \sum_{\hat{z}_{\delta_2-1}=0}^{z_{\delta_2-1}-1} \sum_{\hat{z}_{\delta_2}=0}^{\delta_2} y_i \gamma \Big(\vb*{y},\sum_{k=1}^{\delta_2-2} z_k \vb*{e}_k + \sum_{l=\delta_2-1}^{\delta_2} \hat{z}_l\vb*{e}_l \Big) - \sum_{\hat{z}_{\delta_2}=0}^{\delta_2} y_i \gamma \Big(\vb*{y},\sum_{k=1}^{\delta_2-1} z_k \vb*{e}_k + \hat{z}_{\delta_2}\vb*{e}_{\delta_2} \Big) \\
    &= a_i - \sum_{\hat{\vb*{z}} \in \mathds{N}^{\delta_2}}\sum_{f=1}^{\delta_1} \Bigg( \sum_{\hat{y}_f=0}^{y_f-1} \sum_{\hat{y}_{f+1}=0}^{\delta_1} \hdots \sum_{\hat{y}_{\delta_1}=0}^{\delta_1} \Big(\sum_{k=1}^{f-1} y_k \vb*{e}_k+ \sum_{l=f}^{\delta_1} \hat{y}_l \vb*{e}_l \Big)_i \gamma \Big(\sum_{k=1}^{f-1} y_k \vb*{e}_k+ \sum_{l=f}^{\delta_1} \hat{y}_l \vb*{e}_l, \hat{\vb*{z}} \Big) \Bigg)\\
 &\hspace{1cm} - \sum_{f=1}^{\delta_2-2} \Bigg( \sum_{\hat{z}_f=0}^{z_f-1} \sum_{\hat{z}_{f+1}=0}^{\delta_2} \hdots \sum_{\hat{z}_{\delta_2}=0}^{\delta_2} y_i \gamma \Big(\vb*{y},\sum_{k=1}^{f-1} z_k \vb*{e}_k+ \sum_{l=f}^{\delta_2} \hat{z}_l \vb*{e}_l \Big) \Bigg) \\
    &\hspace{1cm} - \sum_{\hat{z}_{\delta_2-1}=0}^{z_{\delta_2-1}} \sum_{\hat{z}_{\delta_2}=0}^{\delta_2} y_i \gamma \Big(\vb*{y},\sum_{k=1}^{\delta_2-2} z_k \vb*{e}_k + \sum_{l=\delta_2-1}^{\delta_2} \hat{z}_l\vb*{e}_l \Big) \\
    &= \lambda_{(\vb*{y}, \vb*{z} + \vb*{e}_{\delta_2-1})}^{(i)}.
\end{align*}
\end{proof}

\begin{proof}[Proof of Claim \ref{claim:simplification_delta_2}]
    \begin{align*}
    \mu_{(\vb*{y}, \vb*{z} + (\delta_2+1)\vb*{e}_{\delta_2})}^{(i)} &= b_i - \sum_{\hat{\vb*{z}} \in \mathds{N}^{\delta_2}}\sum_{f=1}^{\delta_1} \Bigg( \sum_{\hat{y}_f=0}^{y_f-1} \sum_{\hat{y}_{f+1}=0}^{\delta_1} \hdots \sum_{\hat{y}_{\delta_1}=0}^{\delta_1} z_i \gamma \Big(\sum_{k=1}^{f-1} y_k \vb*{e}_k+ \sum_{l=f}^{\delta_1} \hat{y}_l \vb*{e}_l, \hat{\vb*{z}} \Big) \Bigg)\\
 &\hspace{1cm} - \sum_{f=1}^{\delta_2-2} \Bigg( \sum_{\hat{z}_f=0}^{z_f-1} \sum_{\hat{z}_{f+1}=0}^{\delta_2} \hdots \sum_{\hat{z}_{\delta_2}=0}^{\delta_2} \Big(\sum_{k=1}^{f-1} z_k \vb*{e}_k + \sum_{l=f}^{\delta_2} z_l \vb*{e}_l \Big)_i \gamma \Big(\vb*{y},\sum_{k=1}^{f-1} z_k \vb*{e}_k+ \sum_{l=f}^{\delta_2} \hat{z}_l \vb*{e}_l \Big) \Bigg) \\
    &\hspace{1cm} - \sum_{\hat{z}_{\delta_2-1}=0}^{z_{\delta_2-1}-1} \sum_{\hat{z}_{\delta_2}=0}^{\delta_2} \Big(\sum_{k=1}^{\delta_2-2} z_k \vb*{e}_k + \sum_{l=\delta_2-1}^{\delta_2} \hat{z}_l \vb*{e}_l \Big)_i \gamma \Big(\vb*{y},\sum_{k=1}^{\delta_2-2} z_k \vb*{e}_k + \sum_{l=\delta_2-1}^{\delta_2} \hat{z}_l\vb*{e}_l \Big) \\
    &\hspace{1cm} - \sum_{\hat{z}_{\delta_2}=0}^{\delta_2} \Big(\sum_{k=1}^{\delta_2-1} z_k \vb*{e}_k + \hat{z}_{\delta_2} \vb*{e}_{\delta_2} \Big)_i \gamma \Big(\vb*{y},\sum_{k=1}^{\delta_2-1} z_k \vb*{e}_k + \hat{z}_{\delta_2}\vb*{e}_{\delta_2} \Big)\\
    &= b_i - \sum_{\hat{\vb*{z}} \in \mathds{N}^{\delta_2}}\sum_{f=1}^{\delta_1} \Bigg( \sum_{\hat{y}_f=0}^{y_f-1} \sum_{\hat{y}_{f+1}=0}^{\delta_1} \hdots \sum_{\hat{y}_{\delta_1}=0}^{\delta_1} z_i \gamma \Big(\sum_{k=1}^{f-1} y_k \vb*{e}_k+ \sum_{l=f}^{\delta_1} \hat{y}_l \vb*{e}_l, \hat{\vb*{z}} \Big) \Bigg)\\
 &\hspace{1cm} - \sum_{f=1}^{\delta_2-2} \Bigg( \sum_{\hat{z}_f=0}^{z_f-1} \sum_{\hat{z}_{f+1}=0}^{\delta_2} \hdots \sum_{\hat{z}_{\delta_2}=0}^{\delta_2} \Big(\sum_{k=1}^{f-1} z_k \vb*{e}_k + \sum_{l=f}^{\delta_2} z_l \vb*{e}_l \Big)_i \gamma \Big(\vb*{y},\sum_{k=1}^{f-1} z_k \vb*{e}_k+ \sum_{l=f}^{\delta_2} \hat{z}_l \vb*{e}_l \Big) \Bigg) \\
    &\hspace{1cm} - \sum_{\hat{z}_{\delta_2-1}=0}^{z_{\delta_2-1}} \sum_{\hat{z}_{\delta_2}=0}^{\delta_2} \Big(\sum_{k=1}^{\delta_2-2} z_k \vb*{e}_k + \sum_{l=\delta_2-1}^{\delta_2} z_l \vb*{e}_l \Big)_i \gamma \Big(\vb*{y},\sum_{k=1}^{\delta_2-2} z_k \vb*{e}_k + \sum_{l=\delta_2-1}^{\delta_2} \hat{z}_l\vb*{e}_l \Big) \\
    &= \mu_{(\vb*{y}, \vb*{z} + \vb*{e}_{\delta_2-1})}^{(i)}.
\end{align*}
\end{proof}

\begin{proof}[Proof of Claim \ref{claim:simpliciation_gamma_3}]
    \begin{align*}
    \lambda_{(\vb*{y}, (\delta_2+1)\vb*{e}_1)}^{(i)}&= a_i - \sum_{\hat{\vb*{z}} \in \mathds{N}^{\delta_2}}\sum_{f=1}^{\delta_1} \Bigg( \sum_{\hat{y}_f=0}^{y_f-1} \sum_{\hat{y}_{f+1}=0}^{\delta_1} \hdots \sum_{\hat{y}_{\delta_1}=0}^{\delta_1} \Big(\sum_{k=1}^{f-1} y_k \vb*{e}_k+ \sum_{l=f}^{\delta_1} \hat{y}_l \vb*{e}_l \Big)_i \gamma \Big(\sum_{k=1}^{f-1} y_k \vb*{e}_k+ \sum_{l=f}^{\delta_1} \hat{y}_l \vb*{e}_l, \hat{\vb*{z}} \Big) \Bigg)\\
 &\hspace{1cm} -  \sum_{\hat{z}_1=0}^{\delta_2} \sum_{\hat{z}_2=0}^{\delta_2} \hdots \sum_{\hat{z}_{\delta_2}=0}^{\delta_2} y_i \gamma \Big(\vb*{y}, \sum_{l=1}^{\delta_2} \hat{z}_l \vb*{e}_l \Big) \\
 &= a_i - \sum_{\hat{\vb*{z}} \in \mathds{N}^{\delta_2}}\sum_{f=1}^{\delta_1} \Bigg( \sum_{\hat{y}_f=0}^{y_f-1} \sum_{\hat{y}_{f+1}=0}^{\delta_1} \hdots \sum_{\hat{y}_{\delta_1}=0}^{\delta_1} \Big(\sum_{k=1}^{f-1} y_k \vb*{e}_k+ \sum_{l=f}^{\delta_1} \hat{y}_l \vb*{e}_l \Big)_i \gamma \Big(\sum_{k=1}^{f-1} y_k \vb*{e}_k+ \sum_{l=f}^{\delta_1} \hat{y}_l \vb*{e}_l, \hat{\vb*{z}} \Big) \Bigg)\\
 &\hspace{1cm} - \sum_{\hat{\vb*{z}} \in \mathds{N}^{\delta_2}} y_i \gamma \Big(\vb*{y}, \sum_{l=1}^{\delta_2} \hat{z}_l \vb*{e}_l \Big)\\
    &= a_i - \sum_{\hat{\vb*{z}} \in \mathds{N}^{\delta_2}}\sum_{f=1}^{\delta_1-1} \Bigg( \sum_{\hat{y}_f=0}^{y_f-1} \sum_{\hat{y}_{f+1}=0}^{\delta_1} \hdots \sum_{\hat{y}_{\delta_1}=0}^{\delta_1} \Big(\sum_{k=1}^{f-1} y_k \vb*{e}_k+ \sum_{l=f}^{\delta_1} \hat{y}_l \vb*{e}_l \Big)_i \gamma \Big(\sum_{k=1}^{f-1} y_k \vb*{e}_k+ \sum_{l=f}^{\delta_1} \hat{y}_l \vb*{e}_l, \hat{\vb*{z}} \Big) \Bigg)\\
    &\hspace{1cm}- \sum_{\hat{\vb*{z}} \in \mathds{N}^{\delta_2}}\Bigg( \sum_{\hat{y}_{\delta_1}=0}^{y_{\delta_1}} \Big(\sum_{k=1}^{\delta_1-1} y_k \vb*{e}_k+ \hat{y}_{\delta_1} \vb*{e}_{\delta_1} \Big)_i \gamma\Big(\sum_{k=1}^{\delta_1-1} y_k \vb*{e}_k+ \hat{y}_{\delta_1} \vb*{e}_{\delta_1}, \hat{\vb*{z}} \Big) \Bigg)\\
    &= \lambda_{(\vb*{y} + \vb*{e}_{\delta_1}, \vb*{0})}^{(i)}.
\end{align*}
\end{proof}

\begin{proof}[Proof of Claim \ref{claim:simplification_delta_3}]
    \begin{align*}
    \mu_{(\vb*{y}, (\delta_2+1)\vb*{e}_1)}^{(i)}&= b_i - \sum_{\hat{\vb*{z}} \in \mathds{N}^{\delta_2}}\sum_{f=1}^{\delta_1} \Bigg( \sum_{\hat{y}_f=0}^{y_f-1} \sum_{\hat{y}_{f+1}=0}^{\delta_1} \hdots \sum_{\hat{y}_{\delta_1}=0}^{\delta_1} z_i \gamma \Big(\sum_{k=1}^{f-1} y_k \vb*{e}_k+ \sum_{l=f}^{\delta_1} \hat{y}_l \vb*{e}_l, \hat{\vb*{z}} \Big) \Bigg)\\
 &\hspace{1cm} - \sum_{\hat{z}_1=0}^{\delta_2} \sum_{\hat{z}_2=0}^{\delta_2} \hdots \sum_{\hat{z}_{\delta_2}=0}^{\delta_2} \Big(\sum_{l=1}^{\delta_2} \hat{z}_l \vb*{e}_l \Big)_i \gamma \Big(\vb*{y}, \sum_{l=1}^{\delta_2} \hat{z}_l \vb*{e}_l \Big)  \\
    &= b_i - \sum_{\hat{\vb*{z}} \in \mathds{N}^{\delta_2}}\sum_{f=1}^{\delta_1} \Bigg( \sum_{\hat{y}_f=0}^{y_f-1} \sum_{\hat{y}_{f+1}=0}^{\delta_1} \hdots \sum_{\hat{y}_{\delta_1}=0}^{\delta_1} z_i \gamma \Big(\sum_{k=1}^{f-1} y_k \vb*{e}_k+ \sum_{l=f}^{\delta_1} \hat{y}_l \vb*{e}_l, \hat{\vb*{z}} \Big) \Bigg)\\
 &\hspace{1cm} - \sum_{\hat{\vb*{z}} \in \mathds{N}^{\delta_2}} \Big(\sum_{l=1}^{\delta_2}  \hat{z}_l \vb*{e}_l\Big)_i \gamma \Big(\vb*{y},\sum_{l=1}^{\delta_2}  \hat{z}_l \vb*{e}_l \Big) \\
 &= b_i - \sum_{\hat{\vb*{z}} \in \mathds{N}^{\delta_2}}\sum_{f=1}^{\delta_1-1} \Bigg( \sum_{\hat{y}_f=0}^{y_f-1} \sum_{\hat{y}_{f+1}=0}^{\delta_1} \hdots \sum_{\hat{y}_{\delta_1}=0}^{\delta_1} z_i \gamma \Big(\sum_{k=1}^{f-1} y_k \vb*{e}_k+ \sum_{l=f}^{\delta_1} \hat{y}_l \vb*{e}_l, \hat{\vb*{z}} \Big) \Bigg)\\
 &\hspace{1cm} - \sum_{\hat{\vb*{z}} \in \mathds{N}^{\delta_2}} \sum_{\hat{y}_{\delta_1}=0}^{y_{\delta_1}}\Big(\sum_{l=1}^{\delta_2}  \hat{z}_l \vb*{e}_l\Big)_i \gamma \Big(\vb*{y},\sum_{l=1}^{\delta_2}  \hat{z}_l \vb*{e}_l \Big) \\
    &= \mu_{(\vb*{y} + \vb*{e}_{\delta_1}, \vb*{0})}^{(i)}.
\end{align*}
\end{proof}

\section{Proof of Lemma \ref{lemma:general_order}}
\label{app:pf_lemma_general_order}
To prove Lemma \ref{lemma:general_order}, which identifies the asymptotically dominant term in Lemma~\ref{lemma:main_lemma}, we first analyze the individual summands appearing in Equation~\eqref{eq:main_lemma}. Recall the definition of $\hat{H}(\vb*{a},\vb*{b},\gamma(\cdot))$ from \eqref{eq:Hterms}. We now determine which triples $(\vb*{a},\vb*{b},\gamma(\cdot))$ yield the largest contributions to the sum. The argument proceeds in three steps. First, Claim~\ref{claim:b->only_e} shows that, for any fixed $\vb*{a}$ and $\vb*{b}$, all terms $\hat{H}(\vb*{a}, \vb*{b},\gamma(\cdot))$ in which $\gamma(\vb*{y},\vb*{z})>0$ for some multi-index $(\vb*{y},\vb*{z})$ with at least two nonzero coordinates are asymptotically negligible. Thus, only those $\gamma(\cdot)$ supported on scaled standard basis vectors can contribute at the leading order. Among these remaining terms, Claims~\ref{claim:b->only_pure_e_1} and \ref{claim:b->only_pure_e_2} show that the largest contribution arises when $\gamma(\vb*{e}_k,\vb*{0})=a_k$ and $\gamma(\vb*{0},\vb*{e}_l) = b_l$ for all $k,l$. Finally, Claims~\ref{claim:a->only_first_index_1} and \ref{claim:a->only_first_index_2} show that, over all admissible $\vb*{a}$ and $\vb*{b}$, the asymptotically dominant term is $\vb*{a} = \delta_1 \vb*{e}_1$, $\vb*{b} = \delta_2 \vb*{e}_2$.
\\

\begin{claim}
\label{claim:b->only_e}
Let $\vb*{e}_i$ denote the $i$th standard basis vector.  
Let $\vb*{a} \in \mathbb{N}^{\delta_1}$ and $\vb*{b} \in \mathbb{N}^{\delta_2}$ satisfy $
\sum_{i=1}^{\delta_1} i a_i = \delta_1,
$ and $
\sum_{j=1}^{\delta_2} j b_j = \delta_2.$ 
Let $\gamma \in \hat{R}(\vb*{a},\vb*{b})$ be such that  
$
\gamma(\vb*{y},\vb*{z}) \ge 1$
for some $
(\vb*{y},\vb*{z}) \in \mathbb{N}^{\delta_1} \times \mathbb{N}^{\delta_2}$,
where
\[
\|\vb*{y}\|_0 + \|\vb*{z}\|_0
\;\ge\; 2.
\]

Assume $\delta_1,\delta_2 = O(1)$ and
\[
\mathds{E}\!\left[
    \prod_{i=1}^{\delta_1} (f_1^{(i)}(U))^{y_i}
    \prod_{j=1}^{\delta_2} (f_2^{(j)}(U))^{z_j}
\right]
\in o(n).
\]

\begin{enumerate}\item[(i)]
If there exist indices $i \in [\delta_1]$, $j \in [\delta_2]$ with  
$
y_i > 0 $ or $\ z_j > 0$,
and
$f_1^{(i)}(v)\, f_2^{(j)}(v) = 0$ for all $v \in V$,
then
\[
\hat{H}(\vb*{a},\vb*{b},\gamma) = 0,
\qquad
\hat{H}(\vb*{a},\vb*{b},\gamma') = 0.
\]

\item[(ii)]
If instead there exist constants $c_1,c_2>0$ such that  
\[
\lim_{n\to\infty} \mathds{P}(f_1^{(i)}(U)\ge 1) \ge c_1
\quad\text{for all } i \text{ with } y_i \ge 1,
\]
and
\[
\lim_{n\to\infty} \mathds{P}(f_2^{(j)}(U)\ge 1) \ge c_2
\quad\text{for all } j \text{ with } z_j \ge 1,
\]
then
\[
\hat{H}(\vb*{a},\vb*{b},\gamma) = o\!\left(\hat{H}(\vb*{a},\vb*{b},\gamma')\right),
\]
where $\gamma'(\vb*{y},\vb*{z}) = \gamma(\vb*{y},\vb*{z}) - 1$,
and 

\[
\begin{cases}
\gamma'(\vb*{y}-\vb*{e}_l,\vb*{z}) = \gamma(\vb*{y}-\vb*{e}_l,\vb*{z}) + 1,\\[2mm]
\gamma'(\vb*{e}_l,\vb*{z}) = \gamma(\vb*{e}_l,\vb*{z}) + 1,
\end{cases}
\qquad\text{if } y_l \ge 1 \text{ for some } l,
\]

or

\[
\begin{cases}
\gamma'(\vb*{y},\vb*{z}-\vb*{e}_k) = \gamma(\vb*{y},\vb*{z}-\vb*{e}_k) + 1,\\[2mm]
\gamma'(\vb*{y},\vb*{e}_k) = \gamma(\vb*{y},\vb*{e}_k) + 1,
\end{cases}
\qquad\text{if } z_k \ge 1 \text{ for some } k,
\]
while for all other $(\tilde{\vb*{y}},\tilde{\vb*{z}})$,
$
\gamma'(\tilde{\vb*{y}},\tilde{\vb*{z}}) = \gamma(\tilde{\vb*{y}},\tilde{\vb*{z}}).
$
\end{enumerate}
\end{claim}

\begin{proof}
W.l.o.g.\ let $y_l \ge 1$ (hence $a_l \ge 1$) for some $l \in [\delta_1]$. By symmetry, the same argument applies when $z_k \ge 1$ for some $k \in [\delta_2]$.

Using~\eqref{eq:Hterms}, we have
 \begin{align*}
        &|\hat{H}(\vb*{a},\vb*{b},\gamma(\cdot))|\\
        &= \frac{\delta_1!^{1+w}}{\prod_{k=1}^{\delta_1} k!^{(1+w)a_k}}  \frac{\delta_2!^{1+w}}{\prod_{k=1}^{\delta_2} k!^{(1+w)b_k} }  \\
        &\hspace{0.5cm}\times \hspace{-0.25cm} \prod_{\substack{\tilde{\vb*{y}} \in \mathds{N}^{\delta_1} \\ \tilde{\vb*{z}} \in \mathds{N}^{\delta_2}: \\ (\tilde{\vb*{y}},\tilde{\vb*{z}}) \notin \{(\vb*{y},\vb*{z}),(\vb*{y}-\vb*{e}_l,\vb*{z}), (\vb*{e}_l,\vb*{z})\} }} \hspace{-1cm}\Bigg( \frac{1}{\gamma(\tilde{\vb*{y}},\tilde{\vb*{z}})!} \Big(\frac{n\mathds{E}\big[\prod_{i=1}^{\delta_1} (f_1^{(i)}(U))^{\tilde{y}_i} \prod_{j=1}^{\delta_2} (f_2^{(j)}(U))^{\tilde{z}_j} \big] \big(\sum_{i=1}^{\delta_1} \tilde{y}_i + \sum_{j=1}^{\delta_2} \tilde{z}_j-1 \big)!}{\prod_{i=1}^{\delta_1} \tilde{y}_i! \prod_{j=1}^{\delta_2} \tilde{z}_j!} \Big)^{\gamma(\tilde{\vb*{y}},\tilde{\vb*{z}})} \Bigg)\\
        &\hspace{0.5cm} \times \frac{1}{\gamma(\vb*{y},\vb*{z})!} \Big(\frac{n\mathds{E}\big[\prod_{i=1}^{\delta_1} (f_1^{(i)}(U))^{y_i} \prod_{j=1}^{\delta_2} (f_2^{(j)}(U))^{z_j} \big] \big(\sum_{i=1}^{\delta_1} y_i! + \sum_{j=1}^{\delta_2} z_j-1 \big)!}{\prod_{\substack{i=1\\i \neq l}}^{\delta_1} y_i!y_l! \prod_{j=1}^{\delta_2} z_j!} \Big)^{\gamma(\vb*{y},\vb*{z})}\\
        &\hspace{0.5cm} \times \frac{1}{\gamma(\vb*{y} - \vb*{e}_l,\vb*{z})!} \Big(\frac{n\mathds{E}\big[\prod_{i=1}^{\delta_1} (f_1^{(i)}(U))^{y_i - \mathds{1}_{\{i=l\}}} \prod_{j=1}^{\delta_2} (f_2^{(j)}(U))^{z_j} \big] \big(\sum_{i=1}^{\delta_1} y_i + \sum_{j=1}^{\delta_2} z_j-2 \big)!}{\prod_{\substack{i=1\\i \neq l}}^{d} y_i! (y_l-1)! \prod_{j=1}^{\delta_1} z_j!} \Big)^{\gamma(\vb*{y} - \vb*{e}_l,\vb*{z})} \\
        &\hspace{0.5cm} \times \frac{1}{\gamma(\vb*{e}_l,\vb*{z})!} \Big( \frac{n\mathds{E}\big[f_1^{(l)}(U) \prod_{j=1}^{\delta_2} (f_2^{(j)}(U))^{z_j}\big]\big(\sum_{j=1}^{\delta_2} z_j \big)!}{\prod_{j=1}^{\delta_2} z_j!} \Big)^{\gamma(\vb*{e}_l,\vb*{z})}
    \end{align*}
   and
    \begin{align*}
        &|\hat{H}(\vb*{a},\vb*{b},\gamma'(\cdot))|\\
        &= \frac{\delta_1!^{1+w}}{\prod_{k=1}^{\delta_1} k!^{(1+w)a_k}}  \frac{\delta_2!^{1+w}}{\prod_{k=1}^{\delta_2} k!^{(1+w)b_k} }  \\
        &\hspace{0.5cm} \times \hspace{-0.25cm} \prod_{\substack{\tilde{\vb*{y}} \in \mathds{N}^{\delta_1} \\ \tilde{\vb*{z}} \in \mathds{N}^{\delta_2}: \\ (\tilde{\vb*{y}},\tilde{\vb*{z}}) \notin \{(\vb*{y},\vb*{z}),(\vb*{y}-\vb*{e}_l,\vb*{z}), (\vb*{e}_l,\vb*{z})\} }} \hspace{-1cm}\Bigg( \frac{1}{\gamma(\tilde{\vb*{y}},\tilde{\vb*{z}})!} \Big(\frac{n\mathds{E}\big[\prod_{i=1}^{\delta_1} (f_1^{(i)}(U))^{\tilde{y}_i} \prod_{j=1}^{\delta_2} (f_2^{(j)}(U))^{\tilde{z}_j} \big] \big(\sum_{i=1}^{\delta_1} \tilde{y}_i + \sum_{j=1}^{\delta_2} \tilde{z}_j-1 \big)!}{\prod_{i=1}^{\delta_1} \tilde{y}_i! \prod_{j=1}^{\delta_2} \tilde{z}_j!} \Big)^{\gamma(\tilde{\vb*{y}},\tilde{\vb*{z}})} \Bigg)\\
        &\hspace{0.5cm} \times \frac{1}{(\gamma(\vb*{y},\vb*{z})-1)!} \Big(\frac{n\mathds{E}\big[\prod_{i=1}^{\delta_1} (f_1^{(i)}(U))^{y_i} \prod_{j=1}^{\delta_2} (f_2^{(j)}(U))^{z_j} \big]\big(\sum_{i=1}^{\delta_1} y_i + \sum_{j=1}^{\delta_2} z_j-1 \big)!}{\prod_{\substack{i=1\\i \neq l}}^{\delta_1} y_i! y_l! \prod_{j=1}^{\delta_2} z_j!} \Big)^{\gamma(\vb*{y},\vb*{z})-1}\\
        &\hspace{0.5cm} \times \frac{1}{(\gamma(\vb*{y} - \vb*{e}_l,\vb*{z})+1)!} \Big(\frac{n\mathds{E}\big[\prod_{i=1}^{\delta_1} (f_1^{(i)}(U))^{y_i - \mathds{1}_{\{i=l\}}} \prod_{j=1}^{\delta_2} (f_2^{(j)}(U))^{z_j} \big]\big(\sum_{i=1}^{\delta_1} y_i + \sum_{j=1}^{\delta_2} z_j-2 \big)!}{\prod_{\substack{i=1\\i \neq l}}^{d} y_i! (y_l-1)! \prod_{j=1}^{\delta_2} z_j!} \Big)^{\gamma(\vb*{y} - \vb*{e}_l,\vb*{z})+1} \\
        &\hspace{0.5cm} \times \frac{1}{(\gamma(\vb*{e}_l,\vb*{z})+1)!} \Big( \frac{n\mathds{E}\big[f_1^{(l)}(U) \prod_{j=1}^{\delta_2} (f_2^{(j)}(U))^{z_j}\big](\sum_{j=1}^{\delta_2} z_j)!}{\prod_{j=1}^{\delta_2} z_j!} \Big)^{\gamma(\vb*{e}_l,\vb*{z})+1}.
    \end{align*}

If there exist $i \in [\delta_1]$ or $j \in [\delta_2]$ with $y_i > 0$ or $z_j>0$ such that
\[
    \forall v\in V:\ 
    f_1^{(i)}(v) f_2^{(j)}(v) = 0,
\]
we obtain
\[
    \hat{H}(\vb*{a},\vb*{c},\gamma(\cdot))
    =
    \hat{H}(\vb*{a},\vb*{c},\gamma'(\cdot))
    =
    0.
\]

For the second case, using  
$\gamma(\vb*{y}-\vb*{e}_l,\vb*{z}), 
 \gamma(\vb*{e}_l,\vb*{z}) \le \delta_1$,  
\(
 \sum_{i=1}^{\delta_1} y_i 
 + \sum_{j=1}^{\delta_2} z_j
 \le \delta_1 + \delta_2
\),  
\(
 \prod_{j=1}^{\delta_2} z_j! \le \delta_2!
\),  
and  
$\gamma(\vb*{y},\vb*{z}),y_l \ge 1$,  
we get

\begin{align*}
&\frac{
    |\hat{H}(\vb*{a},\vb*{c},\gamma(\cdot))|
}{
    |\hat{H}(\vb*{a},\vb*{c},\gamma'(\cdot))|
}
\nonumber\\
&
= \frac{
    (\gamma(\vb*{y}-\vb*{e}_l,\vb*{z})+1)
    (\gamma(\vb*{e}_l,\vb*{z})+1)
    \mathds{E}\!\left[
        \prod_{i=1}^{\delta_1}(f_1^{(i)}(U))^{y_i}
        \prod_{j=1}^{\delta_2}(f_2^{(j)}(U))^{z_j}
    \right]
    \big(\sum_{i=1}^{\delta_1}y_i + \sum_{j=1}^{\delta_2}z_j -1 \big)
    \prod_{j=1}^{\delta_2} z_j!
}{
    \gamma(\vb*{y},\vb*{z})
    y_l
    n\mathds{E}\!\left[
        \prod_{i=1}^{\delta_1}(f_1^{(i)}(U))^{y_i - \mathds{1}_{\{i=l\}}}
        \prod_{j=1}^{\delta_2}(f_2^{(j)}(U))^{z_j}
    \right]
    \mathds{E}\!\left[
        f_1^{(l)}(U)
        \prod_{j=1}^{\delta_2}(f_2^{(j)}(U))^{z_j}
    \right]
    \big(\sum_{j=1}^{\delta_2}z_j \big)!
}
\\
&\le
\frac{
    (\delta_1+1)^2
    \mathds{E}\!\left[
        \prod_{i=1}^{\delta_1}(f_1^{(i)}(U))^{y_i}
        \prod_{j=1}^{\delta_2}(f_2^{(j)}(U))^{z_j}
    \right]
    (\delta_1 + \delta_2 - 1)
    \delta_2!
}{
    n
    \mathds{E}\!\left[
        \prod_{i=1}^{\delta_1}(f_1^{(i)}(U))^{y_i - \mathds{1}_{\{i=l\}}}
        \prod_{j=1}^{\delta_2}(f_2^{(j)}(U))^{z_j}
    \right]
    \mathds{E}\!\left[
        f_1^{(l)}(U)
        \prod_{j=1}^{\delta_2}(f_2^{(j)}(U))^{z_j}
    \right]
}.
\end{align*}

Since $\delta_1,\delta_2 = O(1)$ and  
\[
\mathds{E}\!\left[
    \prod_{i=1}^{\delta_1}(f_1^{(i)}(U))^{y_i}
    \prod_{j=1}^{\delta_2}(f_2^{(j)}(U))^{z_j}
\right]
\in o(n),
\]
and for $n$ sufficiently large
\begin{align*}
\mathds{E}\!\left[
    \prod_{i=1}^{\delta_1}(f_1^{(i)}(U))^{y_i - \mathds{1}_{\{i=l\}}}
    \prod_{j=1}^{\delta_2}(f_2^{(j)}(U))^{z_j}
\right]
&\ge
\prod_{\substack{i=1\\ y_i \ge 1}}^{\delta_1}
\mathds{P}(f_1^{(i)}(U)\ge 1)
\prod_{\substack{j=1\\ z_j \ge 1}}^{\delta_2}
\mathds{P}(f_2^{(j)}(U)\ge 1)\nonumber\\
&\ge\ 
\min(c_1,c_1^{\delta_1}) \min(c_2,c_2^{\delta_2}) > 0
\end{align*}
and
\begin{align*}
\mathds{E}\!\left[
    f_1^{(l)}(U)
    \prod_{j=1}^{\delta_2}(f_2^{(j)}(U))^{z_j}
\right]
&\ge
c_1 \min(c_2,c_2^{\delta_2}) > 0,
\end{align*}
we conclude
\[
\frac{
    |\hat{H}(\vb*{a},\vb*{c},\gamma(\cdot))|
}{
    |\hat{H}(\vb*{a},\vb*{c},\gamma'(\cdot))|
}
\le
\frac{o(n)}{n}
=
o(1).
\]

\end{proof}

\begin{claim}
\label{claim:b->only_pure_e_1}
    Let $\vb*{e}_i$ be the $i$th standard basis vector. Let $\vb*{a} \in \mathds{N}^{\delta_1}: \sum_{i=1}^{\delta_1} ia_i=\delta_1$ and $\vb*{b} \in \mathds{N}^{\delta_2}: \sum_{j=1}^{\delta_1} jb_i=\delta_2$. Let $\gamma(\cdot) \in \hat{R}(\vb*{a},\vb*{b})$ with $\gamma(y_l \vb*{e}_l, \vb*{0}) \geq 1$ for some $y_l \geq 2$. If
    \begin{enumerate}
        \item $\delta_1 \in O(1)$
        \item $\mathds{E}[(f_1^{(l)}(U))^{y_l}] \in o(n)$
        \item $\exists c>0 :  \lim_{n \rightarrow \infty} \mathds{P}(f_1^{(l)}(U) \geq 1) \geq c$
    \end{enumerate}
    then 
      \begin{equation*}
    \hat{H}(\vb*{a},\vb*{b},\gamma(\cdot)) = o\Big(\hat{H}(\vb*{a},\vb*{b},\gamma'(\cdot))\Big),
    \end{equation*}
    where 
    \begin{align*}
        \begin{cases}
            \gamma'(y_l \vb*{e}_l,\vb*{0})=\gamma(y_l \vb*{e}_l,\vb*{0})-1\\
            \gamma'((y_l-1) \vb*{e}_l,\vb*{0})=\gamma((y_l-1)\vb*{e}_l,\vb*{0}) + 1\\
            \gamma'(\vb*{e}_l,\vb*{0})=\gamma(\vb*{e}_l,\vb*{0}) + 1
        \end{cases}
    \end{align*}
    and $\gamma'(\tilde{\vb*{y}},\tilde{\vb*{z}}) = \gamma(\tilde{\vb*{y}},\tilde{\vb*{z}})$ for all other $\tilde{\vb*{y}},\tilde{\vb*{z}}$.
\end{claim}

\begin{proof}
    Equation~\eqref{eq:Hterms} gives
    \begin{align*}
        &| \hat{H}(\vb*{a},\vb*{b},\gamma(\cdot))| \\
        &=\frac{\delta_1!^{1+w}}{\prod_{k=1}^{\delta_1} k!^{(1+w)a_k}}  \frac{\delta_2!^{1+w}}{\prod_{k=1}^{\delta_2} k!^{(1+w)b_k} }  \\
        &\hspace{0.5cm} \times \hspace{-0.25cm}\prod_{\substack{\tilde{\vb*{y}} \in \mathds{N}^{\delta_1} \\ \tilde{\vb*{z}} \in \mathds{N}^{\delta_2}: \\ (\tilde{\vb*{y}},\tilde{\vb*{z}}) \notin \{(y_l \vb*{e}_l, \vb*{0}), ((y_l-1)\vb*{e}_l,\vb*{0}), (\vb*{e}_l,\vb*{0})\} }}\hspace{-1.5cm} \Bigg( \frac{1}{\gamma(\tilde{\vb*{y}},\tilde{\vb*{z}})!}  \Big(\frac{n\mathds{E}\big[\prod_{i=1}^{\delta_1} (f_1^{(i)}(U))^{\tilde{y}_i} \prod_{j=1}^{\delta_2} (f_2^{(j)}(U))^{\tilde{z}_j} \big] \big(\sum_{i=1}^{\delta_1} \tilde{y}_i + \sum_{j=1}^{\delta_2} \tilde{z}_j-1 \big)!}{\prod_{i=1}^{\delta_1} \tilde{y}_i! \prod_{j=1}^{\delta_2} \tilde{z}_j!} \Big)^{\gamma(\tilde{\vb*{y}},\tilde{\vb*{z}})} \Bigg)\\
        &\hspace{0.5cm} \times \frac{1}{\gamma(y_l\vb*{e}_l,\vb*{0})!} \Big(\frac{n\mathds{E}[(f_1^{(l)}(U))^{y_l}](y_l-1)!}{y_l!} \Big)^{\gamma(y_l\vb*{e}_l,\vb*{0})}\\
        &\hspace{0.5cm}\times \frac{1}{\gamma((y_l-1)\vb*{e}_l,\vb*{0})!} \Big(\frac{n\mathds{E}[(f_1^{(l)}(U))^{y_l-1}](y_l-2)!}{(y_l-1)!} \Big)^{\gamma((y_l-1)\vb*{e}_l,\vb*{0})} \\
        &\hspace{0.5cm} \times \frac{1}{\gamma(\vb*{e}_l,\vb*{0})!} \Big(n\mathds{E}[f_1^{(l)}(U)] \Big)^{\gamma(\vb*{e}_l,\vb*{0})}
    \end{align*}
    and
    \begin{align*}
        &| \hat{H}(\vb*{a},\vb*{b},\gamma'(\cdot))| \\
        &=\frac{\delta_1!^{1+w}}{\prod_{k=1}^{\delta_1} k!^{(1+w)a_k}}  \frac{\delta_2!^{1+w}}{\prod_{k=1}^{\delta_2} k!^{(1+w)b_k} }  \\
        &\hspace{0.5cm} \times \hspace{-0.25cm} \prod_{\substack{\tilde{\vb*{y}} \in \mathds{N}^{\delta_1} \\ \tilde{\vb*{z}} \in \mathds{N}^{\delta_2}: \\ (\tilde{\vb*{y}},\tilde{\vb*{z}}) \notin \{(y_l \vb*{e}_l, \vb*{0}), ((y_l-1)\vb*{e}_l,\vb*{0}), (\vb*{e}_l,\vb*{0})\} }} \hspace{-1.5cm}\Bigg( \frac{1}{\gamma(\tilde{\vb*{y}},\tilde{\vb*{z}})!} \Big(\frac{n\mathds{E}\big[\prod_{i=1}^{\delta_1} (f_1^{(i)}(U))^{\tilde{y}_i} \prod_{j=1}^{\delta_2} (f_2^{(j)}(U))^{\tilde{z}_j} \big] \big(\sum_{i=1}^{\delta_1} \tilde{y}_i + \sum_{j=1}^{\delta_2} \tilde{z}_j-1 \big)!}{\prod_{i=1}^{\delta_1} \tilde{y}_i! \prod_{j=1}^{\delta_2} \tilde{z}_j!} \Big)^{\gamma(\tilde{\vb*{y}},\tilde{\vb*{z}})} \Bigg)\\
        &\hspace{0.5cm} \times\frac{1}{(\gamma(y_l\vb*{e}_l,\vb*{0})-1)!} \Big(\frac{n\mathds{E}[(f_1^{(l)}(U))^{y_l}](y_l-1)!}{y_l!} \Big)^{\gamma(y_l\vb*{e}_l,\vb*{0})-1}\\
        &\hspace{0.5cm} \times\frac{1}{(\gamma((y_l-1)\vb*{e}_l,\vb*{0})+1)!} \Big(\frac{n\mathds{E}[(f_1^{(l)}(U))^{y_l-1}](y_l-2)!}{(y_l-1)!} \Big)^{\gamma((y_l-1)\vb*{e}_l,\vb*{0})+1} \\
        &\hspace{0.5cm} \times\frac{1}{(\gamma(\vb*{e}_l,\vb*{0})+1)!} \Big(n\mathds{E}[f_1^{(l)}(U)] \Big)^{\gamma(\vb*{e}_l,\vb*{0})+1}.
    \end{align*}
    Using $\gamma((y_l-1)\vb*{e}_l,\vb*{0}),\gamma(\vb*{e}_l,\vb*{0}), y_l \leq \delta_1$ and $\gamma(y_l \vb*{e}_l,\vb*{0}), y_l \geq 1$ we obtain
    \begin{align*}
        \frac{|\hat{H}(\vb*{a},\vb*{b},\gamma(\cdot))|}{|\hat{H}(\vb*{a},\vb*{b},\gamma'(\cdot))|} &= \frac{\gamma((y_l-1)\vb*{e}_l,\vb*{0})\gamma(\vb*{e}_l,\vb*{0}) \mathds{E}[(f_1^{(l)}(U))^{y_l}](y_l-1)}{\gamma(y_l \vb*{e}_l) y_l n\mathds{E}[(f_1^{(l)}(U))^{y_l-1} ] \mathds{E}[f_1^{(l)}(U) ]} \leq \frac{\delta_1^2 \mathds{E}[(f^{(l)}(U))^{y_l}](\delta_1-1)}{n\mathds{E}[(f^{(l)}(U))^{y_l-1} ] \mathds{E}[f^{(l)}(U) ]}.
    \end{align*}
    Using that $\delta_1 \in O(1)$, $ \mathds{E}[(f^{(l)}(U))^{y_l}] \in o(n)$ and for $n$ large enough 
    $\mathds{E}[(f^{(l)}(U))^{y_l-1} ] \geq \mathds{P}((f^{(l)}(U))^{y_l-1}\geq 1) \geq \mathds{P}(f^{(l)}(U) \geq 1) \geq c$ and $\mathds{E}[f^{(l)}(U) ] \geq \mathds{P}(f^{(l)}(U) \geq 1) \geq c > 0,$ we obtain
        \begin{align*}
        \frac{|\hat{H}(\vb*{a},\vb*{b},\gamma(\cdot))|}{|\hat{H}(\vb*{a},\vb*{b},{\gamma}'(\cdot))|} &\leq  \frac{o(n)}{n} = o(1).
    \end{align*}
\end{proof}

\begin{claim}
\label{claim:b->only_pure_e_2}
    Let $\vb*{e}_i$ be the $i$th standard basis vector. Let $\vb*{a} \in \mathds{N}^{\delta_1}: \sum_{i=1}^{\delta_1} ia_i=\delta_1$ and $\vb*{b} \in \mathds{N}^{\delta_2}: \sum_{j=1}^{\delta_1} jb_i=\delta_2$. Let $\gamma(\cdot) \in \hat{R}(\vb*{a},\vb*{b})$ with $\gamma(\vb*{0}, z_k \vb*{e}_k) \geq 1$ for some $z_k \geq 2$. If
    \begin{enumerate}
        \item $\delta_2 \in O(1)$
        \item $ \mathds{E}[(f_2^{(k)}(U))^{z_k}] \in o(n)$
        \item $\exists c>0 :  \lim_{n \rightarrow \infty} \mathds{P}(f_2^{(k)}(U) \geq 1) \geq c$
    \end{enumerate}
    then 
      \begin{equation*}
    \hat{H}(\vb*{a},\vb*{b},\gamma(\cdot)) = o\Big(\hat{H}(\vb*{a},\vb*{b},\gamma'(\cdot)) \Big),
    \end{equation*}
    where 
    \begin{align*}
        \begin{cases}
            \gamma'(\vb*{0},z_k \vb*{e}_k)=\gamma(\vb*{0},z_k \vb*{e}_k)-1\\
            \gamma'(\vb*{0},(z_k-1) \vb*{e}_k)=\gamma(\vb*{0},(z_k-1)\vb*{e}_k) + 1\\
            \gamma'(\vb*{0},\vb*{e}_k)=\gamma(\vb*{0},\vb*{e}_k) + 1
        \end{cases}
    \end{align*}
    and $\gamma'(\tilde{\vb*{y}},\tilde{\vb*{z}}) = \gamma(\tilde{\vb*{y}},\tilde{\vb*{z}})$ for all other $\tilde{\vb*{y}},\tilde{\vb*{z}}$.
\end{claim}
\begin{proof}
    Using symmetry, the proof of Claim \ref{claim:b->only_pure_e_1} can be applied with the arguments within $\gamma$ and $\gamma'$ swapped.
\end{proof}

\begin{claim}
\label{claim:a->only_first_index_1}
    Let $\vb*{e}_i$ be the $i$th standard basis vector. Let $\vb*{a} \in \mathds{N}^{\delta_1}: \sum_{i=1}^{\delta_1} ia_i = \delta_1$, $\vb*{b} \in \mathds{N}^{\delta_2}: \sum_{j=1}^{\delta_2} jb_j = \delta_2$, and let $\gamma(\cdot) \in \hat{R}(\vb*{a},\vb*{b})$ with $\forall i \in [\delta_1]: \gamma(\vb*{e}_i,\vb*{0})=a_i$ and $\forall j \in [\delta_2]: \gamma(\vb*{0},\vb*{e}_j)=b_j$. Let $a_l \geq 1$ for some $l \geq 2$. If
    \begin{enumerate}
        \item $\delta_1 \in O(1)$
        \item $\mathds{E}[f_1^{(l)}(U)] \in o(n)$
        \item $\exists c>0 \textnormal{ s.t. }  \lim_{n \rightarrow \infty}  \mathds{P}(f_1^{(1)}(U) \geq 1), \lim_{n \rightarrow \infty} \mathds{P}(f_1^{(l-1)}(U) \geq 1) \geq c$
    \end{enumerate}
    then
        \begin{equation*}
    \hat{H}(\vb*{a},\vb*{b},\gamma(\cdot)) = o\Big(\hat{H}(\vb*{a}', \vb*{b},\gamma(\cdot)) \Big),
    \end{equation*}
    where $\vb*{a}' = \vb*{a}-\vb*{e}_l + \vb*{e}_{l-1} + \vb*{e}_1$.
\end{claim}
\begin{proof} 
    Using~\eqref{eq:Hterms} gives
    \begin{align*}
        |\hat{H}(\vb*{a},\vb*{b},\gamma(\cdot))| &=\frac{\delta_1!^{1+w}}{\prod_{\substack{k=2\\k \neq l-1,l}}^{\delta_1} k!^{(1+w)a_k} (l-1)!^{(1+w)a_{l-1}}l!^{(1+w)a_l}}  \frac{\delta_2!^{1+w}}{\prod_{k=1}^{\delta_2} k!^{(1+w)b_k} } \\
        &\quad \times\prod_{\substack{i=2\\ i \neq l-1,l}}^{\delta_1} \Bigg( \frac{1}{a_i!} \Big(n\mathds{E}[f_1^{(i)}(U)] \Big)^{a_i} \Bigg) \prod_{j=1}^{\delta_2} \Bigg( \frac{1}{b_j!} \Big(n\mathds{E}[f_2^{(j)}(U)] \Big)^{b_j} \Bigg)\\
        &\quad \times\frac{1}{a_1!} \Big(n\mathds{E}[f_1^{(1)}(U) ] \Big)^{a_1} \frac{1}{a_{l-1}!} \Big(n\mathds{E}[f_1^{(l-1)}(U) ] \Big)^{a_{l-1}} \frac{1}{a_l!} \Big(n\mathds{E}[f_1^{(l)}(U) ] \Big)^{a_l}
    \end{align*}
    and
    \begin{align*}
        &|\hat{H}({\vb*{a}}', \vb*{b},\gamma(\cdot))| \\
        &= \frac{\delta_1!^{1+w}}{\prod_{\substack{k=2\\k \neq l-1,l}}^{\delta_1} k!^{(1+w)a_k} (l-1)!^{(1+w)(a_{l-1}+1)}l!^{(1+w)(a_l-1)}}  \frac{\delta_2!^{1+w}}{\prod_{k=1}^{\delta_2} k!^{(1+w)b_k} } \\
        &\quad \times \prod_{\substack{i=2\\ i \neq l-1,l}}^{\delta_1} \Bigg( \frac{1}{a_i!} \Big(n\mathds{E}[f_1^{(i)}(U)] \Big)^{a_i} \Bigg) \prod_{j=1}^{\delta_2} \Bigg( \frac{1}{b_j!} \Big(n\mathds{E}[f_2^{(j)}(U)] \Big)^{b_j} \Bigg)\\
        &\quad \times \frac{1}{(a_1+1)!} \Big(n\mathds{E}[f_1^{(1)}(U) ] \Big)^{a_1+1} \frac{1}{(a_{l-1}+1)!} \Big(n\mathds{E}[f_1^{(l-1)}(U) ] \Big)^{a_{l-1}+1} \frac{1}{(a_l-1)!} \Big(n\mathds{E}[f_1^{(l)}(U) ] \Big)^{a_l-1}.
    \end{align*}
    Using $a_1, a_{l-1} \leq \delta_1$,  $l, a_l \geq 1$ and $w \geq 0$ we obtain
    \begin{align*}
        \frac{|\hat{H}(\vb*{a},\vb*{b},\gamma(\cdot))|}{|\hat{H}(\vb*{a}',\vb*{b},\gamma(\cdot))|} &= \frac{(a_1+1)(a_{l-1}+1) \mathds{E}[f_1^{(l)}(U)]}{l^{1+w} a_l n \mathds{E}[f_1^{(1)}(U)] \mathds{E}[f_1^{(l-1)}(U)]} \leq \frac{\delta_1^2 \mathds{E}[f_1^{(l)}(U)]}{n \mathds{E}[f_1^{(1)}(U)] \mathds{E}[f_1^{(l-1)}(U)]}.
    \end{align*}
    Using that $\delta_1 \in O(1)$, $\mathds{E}[f_1^{(l)}(U)] \in o(n)$ and for $n$ large enough \[\mathds{E}[f_1^{(1)}(U)] \geq \mathds{P}(f_1^{(1)}(U) \geq 1) \geq c>0\] and $\mathds{E}[f_1^{(l-1)}(U)] \geq \mathds{P}(f_1^{(l-1)}(U) \geq 1) \geq c > 0,$ we obtain
    \begin{align*}
        \frac{|\hat{H}(\vb*{a},\vb*{b},\gamma(\cdot))|}{|\hat{H}(\vb*{a}',\vb*{b},\gamma(\cdot))|} &\leq \frac{o(n)}{n} = o(1).
    \end{align*}
\end{proof}

\begin{claim}
\label{claim:a->only_first_index_2}
    Let $\vb*{e}_i$ be the $i$th standard basis vector. Let $\vb*{a} \in \mathds{N}^{\delta_1}: \sum_{i=1}^{\delta_1} ia_i = \delta_1$, $\vb*{b} \in \mathds{N}^{\delta_2}: \sum_{j=1}^{\delta_2} jb_j = \delta_2$, and let $\gamma(\cdot) \in \hat{R}(\vb*{a},\vb*{b})$ with $\forall i \in [\delta_1]: \gamma(\vb*{e}_i,\vb*{0})=a_i$ and $\forall j \in [\delta_2]: \gamma(\vb*{0},\vb*{e}_j)=b_j$. Let $b_k \geq 1$ for some $k \geq 2$. If
    \begin{enumerate}
        \item $\delta_2 \in O(1)$
        \item $\mathds{E}[f_2^{(k)}(U)] \in o(n)$
        \item $\exists c>0 \textnormal{ s.t. } \\ \lim_{n \rightarrow \infty}  \mathds{P}(f_2^{(2)}(U) \geq 1), \lim_{n \rightarrow \infty} \mathds{P}(f_2^{(k-1)}(U) \geq 1) \geq c$
    \end{enumerate}
    then
        \begin{equation*}
    \hat{H}(\vb*{a},\vb*{b},\gamma(\cdot)) = o\Big(\hat{H}(\vb*{a}, \vb*{b}',\gamma(\cdot)) \Big),
    \end{equation*}
    where $\vb*{b}' = \vb*{b}-\vb*{e}_k + \vb*{e}_{k-1} + \vb*{e}_1$.
\end{claim}
\begin{proof}
    Using symmetry, the proof of Claim \ref{claim:a->only_first_index_1} can be applied with the arguments $\vb*{a}$ and $\vb*{b}$ swapped.
    By symmetry, the same argument holds if $b_k \geq 1$ for some $k \geq 2$.
\end{proof}

We are now ready to prove Lemma \ref{lemma:general_order}, which shows how we can approximate Equation~\eqref{eq:main_lemma} over all $\hat{H}(\cdot)$ in terms of the functions $f_1^{(1)}$ and $f_2^{(1)}$.

\begin{proof}
We can use Claims \ref{claim:b->only_e}-\ref{claim:a->only_first_index_2}, since all requirements are met. We have $\delta_1,\delta_2 \in O(1)$ (Claims \ref{claim:b->only_e}-\ref{claim:a->only_first_index_2} assumption 1) and $\forall k \in [\delta_1], \forall l \in [\delta_2],  \forall y_k \in [\lfloor \frac{\delta_1}{k} \rfloor],\forall z_l \in [\lfloor \frac{\delta_2}{l} \rfloor] $ 
\begin{align*}
\mathds{E}[(f_1^{(k)}(U))^{y_k}] &\leq \mathds{E}\big[(f_1^{(k)}(U)^{\lfloor \frac{\delta_1}{k} \rfloor} \big] = o(n)\\
\mathds{E}[(f_2^{(l)}(U))^{z_l}] &\leq \mathds{E}\big[(f_2^{(l)}(U)^{\lfloor \frac{\delta_2}{l} \rfloor} \big] = o(n)
\end{align*}
(Claims \ref{claim:b->only_pure_e_1}-\ref{claim:a->only_first_index_2} assumption 2). Then, $\forall k \in [\delta_1], \forall l \in [\delta_2]$ and for $n$ large enough:
\[
\mathds{P}[f_1^{(k)}(U) \geq 1],\mathds{P}[f_2^{(l)}(U) \geq 1] \geq c
\]
(Claims \ref{claim:b->only_e}-\ref{claim:a->only_first_index_2} assumption 3). Lastly, $\forall \vb*{y} \in \mathds{N}^{\delta_1}$ with $\sum_{i=1}^{\delta_1} iy_i \leq \delta_1$ and $\forall \vb*{z} \in \mathds{N}^{\delta_2}$ with $\sum_{j=1}^{\delta_2} jz_j \leq \delta_2$ and $\sum_{i=1}^{\delta_1} \mathds{1}_{\{y_i \geq 1\}}+\sum_{j=1}^{\delta_2} \mathds{1}_{\{z_j \geq 1\}} \geq 2$:
\[
\mathds{E}\Big[\prod_{i=1}^{\delta_1} (f_1^{(i)} (U))^{\tilde{y}_i}\prod_{j=1}^{\delta_2} (f_2^{(j)} (U))^{\tilde{z}_j} \Big] \in o(n)
\]
(Claim \ref{claim:b->only_e} assumption 2).

By Claims \ref{claim:b->only_e}-\ref{claim:a->only_first_index_2}, all terms are of a smaller order of magnitude than $\hat{H}(\vb*{a}^*,\vb*{b}^*,\gamma^*(\cdot))$, where $\vb*{a}^* = \delta_1\vb*{e}_1$,$\vb*{b}^* = \delta_2\vb*{e}_1$ and $\gamma^*(\cdot)$ is such that $\gamma^*(\vb*{e}_1,\vb*{0}) = a_1$, $\gamma^*(\vb*{0},\vb*{e}_1) = b_1$ and $\gamma^*(\vb*{y},\vb*{z})=0$ for all other $\vb*{y},\vb*{z}$. Moreover, all remaining terms are of a smaller order of magnitude than $\hat{H}(\vb*{a}^*,\vb*{b}^*,\hat{\gamma}(\cdot))$, $\hat{H}(\vb*{a}^*,\vb*{b}^*,\tilde{\gamma}(\cdot))$, $\hat{H}(\vb*{a}^*,\vb*{b}^*,\gamma'(\cdot))$, $\hat{H}(\hat{\vb*{a}},\vb*{b}^*,\gamma^*(\cdot))$ or $\hat{H}(\vb*{a}^*,\hat{\vb*{b}},\gamma^*(\cdot))$, where
\begin{align*}
    &\hat{\gamma}(\vb*{e}_1,\vb*{0}) = \delta_1-2, \hat{\gamma}(2\vb*{e}_1,\vb*{0}) = 1, \hat{\gamma}(\vb*{0},\vb*{e}_1) = \delta_2\\
    &\tilde{\gamma}(\vb*{0},\vb*{e}_1) = \delta_2-2, \tilde{\gamma}(\vb*{0},2\vb*{e}_1) = 1, \tilde{\gamma}(\vb*{e}_1,\vb*{0}) = \delta_1\\
    &\gamma'(\vb*{e}_1,\vb*{0}) = \delta_1-1, \gamma'(\vb*{0},\vb*{e}_1) = \delta_2-1, \gamma'(\vb*{e}_1,\vb*{e}_1) = 1\\
    &\hat{\vb*{a}} = (\delta_1-2)\vb*{e}_1 + \vb*{e}_2\\
    &\hat{\vb*{b}}_1 = (\delta_2-2)\vb*{e}_1 + \vb*{e}_2.
\end{align*}
Lastly, since $\delta_1,\delta_2 \in O(1)$, there is a finite number of terms. Then, 
\begin{align*}
 &\sum_{\substack{\vb*{a} \in \mathds{N}^{\delta_1}:\\\sum_{i=1}^{\delta_1} ia_i = \delta_1}} \sum_{\substack{\vb*{b} \in \mathds{N}^{\delta_2}:\\\sum_{j=1}^{\delta_2} jb_j = \delta_2}} \sum_{\gamma(\cdot) \in \hat{R}(\vb*{a},\vb*{b})} \hat{H}(\vb*{a},\vb*{b},\gamma(\cdot))  \\
 &= \hat{H}(\vb*{a}^*,\vb*{b}^*,\gamma^*(\cdot)) +\Big(\hat{H}(\vb*{a}^*,\vb*{b}^*,\hat{\gamma}(\cdot)) + \hat{H}(\vb*{a}^*,\vb*{b}^*,\tilde{\gamma}(\cdot)) +\hat{H}(\vb*{a}^*,\vb*{b}^*,\gamma'(\cdot))  \\
 &\hspace{0.5cm}+ \hat{H}(\hat{\vb*{a}},\vb*{b}^*,\gamma^*(\cdot)) + \hat{H}(\vb*{a}^*,\hat{\vb*{b}},\gamma^*(\cdot))\Big)(1+o(1)).
\end{align*}

Although the definition of $\hat{H}(\vb*{a}, \vb*{b}, \gamma(\cdot))$ is involved, the term becomes a lot more manageable for specific values of $\vb*{a}, \vb*{b}$ and $\gamma(\cdot)$. For example, in the case of $\hat{H}(\vb*{a}^*,\vb*{b}^*,\gamma^*(\cdot))$, the sum $\sum_{k=1}^{\delta_1} k!^{(1+w)a_k}$ reduces to 1, as does the sum $\sum_{k=1}^{\delta_d} k!^{(1+w)b_k}$. In addition, the exponent of $(-1)$ reduces to 0. Moreover, the only terms in the product over $\vb*{y}$ and $\vb*{z}$ that are not equal to 1 are the terms with $\vb*{y}=\vb*{e}_1, \vb*{z}=\vb*{0}$ and with $\vb*{y}=\vb*{0}, \vb*{z}=\vb*{e}_1$. The expectation in the fraction then reduces to $\mathds{E}[f_1^{(1)}(U)]$ for the first term and to $\mathds{E}[f_2^{(1)}(U)]$ for the second term. In both terms, $(\sum_{i} y_i + \sum_{j} z_j-1)!/(\prod_i y_i! \prod_j z_j!) = 1$. In summary, $\hat{H}(\vb*{a}^*,\vb*{b}^*,\gamma^*(\cdot)) = (\delta_1!\delta_2!)^w (n\mathds{E}[f_1^{(1)}(U) ])^{\delta_1} (n\mathds{E}[f_2^{(1)}(U) ])^{\delta_2}$. Similarly, for the other relevant parameters of $\hat{H}$, the term reduces a lot.
\end{proof}

\printbibliography

\end{document}